\documentclass[10pt]{amsart}

\makeatletter
\long\def\emph#1{\ifmmode\nfss@text{\em #1}\else\hmode@bgroup\text@command{#1}\em\check@icl #1\check@icr\expandafter\egroup\fi}
\makeatother

\usepackage{microtype}
\usepackage{amsmath,amsthm,amsfonts,amssymb,mathtools,mathscinet,tikz,anysize,combelow}
\usepackage[bookmarks=false]{hyperref}
\usepackage{stmaryrd}
\usepackage[normalem]{ulem}
\usepackage[all]{xy}
\SelectTips{cm}{}
\usepackage{tikz}
\tikzset{anchorbase/.style={baseline={([yshift=-0.5ex]current bounding box.center)}}}
\tikzset{
    partial ellipse/.style args={#1:#2:#3}{
        insert path={+ (#1:#3) arc (#1:#2:#3)}
    }
}
\usetikzlibrary{calc}
\usetikzlibrary{decorations.markings}
\usetikzlibrary{decorations.pathreplacing}
\usetikzlibrary{arrows,shapes,positioning}
\tikzstyle directed=[postaction={decorate,decoration={markings,
    mark=at position #1 with {\arrow{>}}}}]
\tikzstyle rdirected=[postaction={decorate,decoration={markings,
    mark=at position #1 with {\arrow{<}}}}]
    
\usetikzlibrary{cd,external}

\newcommand{\annfront}[3]{
\draw (0,0)to(#1,0) ;
\draw [red, thick, rdirected=.3, rdirected=.25] (#1+0.5*#2,#2) to (#1,0);
as\\draw (0,0+#3) to (#1,0+#3);
\draw (0+0.5*#2,#2+#3) to (#1+0.5*#2,#2+#3);
\draw [red, thick, rdirected=.3, rdirected=.25] (#1+0.5*#2,#2+#3) to (#1,0+#3);
\draw [red, thick, rdirected=.3, rdirected=.25] (0+0.5*#2,#2+#3) to (0,0+#3);
\draw (0,0) to (0,0+#3);
\draw(#1,0) to(#1,0+#3);
}

\newcommand{\annback}[3]{
\draw [opacity=.5] (0.5*#2,#2)  to  (0.5*#2,#2+#3); 
\draw (#1+0.5*#2,#2) to (#1+0.5*#2,#2+#3);
\draw [opacity=.5] (0+0.5*#2,#2)to (#1+0.5*#2,#2) ;
\draw [red,opacity=.5, thick, rdirected=.3, rdirected=.25] (0+0.5*#2,#2) to (0,0);
}

\usepackage{etoolbox}
\AtBeginEnvironment{tikzcd}{\tikzexternaldisable}
\AtEndEnvironment{tikzcd}{\tikzexternalenable}
    
\usepackage{bbm}

\usepackage[enableskew]{youngtab}

\theoremstyle{plain}
\newtheorem{thm}{Theorem}[section]
\newtheorem{thm-nono}{Theorem}
\newtheorem{prop}[thm]{Proposition}
\newtheorem{cor}[thm]{Corollary}
\newtheorem{lem}[thm]{Lemma}
\newtheorem{conj}[thm]{Conjecture}

\newtheorem{defi}[thm]{Definition}
\newtheorem{rem}[thm]{Remark}
\newtheorem{exa}[thm]{Example}

\def\comm#1{}%

\DeclareMathOperator{\Sk}{Sk}
\DeclareMathOperator{\End}{End}
\DeclareMathOperator{\Hom}{Hom}
\DeclareMathOperator{\Kar}{Kar}
\DeclareMathOperator{\Rep}{Rep}

\DeclareMathOperator{\vvTr}{vTr}
\DeclareMathOperator{\hTr}{hTr}
\DeclareMathOperator{\Cone}{Cone}

\DeclareMathOperator{\Hilb}{Hilb}

\DeclareMathOperator{\SBim}{SBim}
\DeclareMathOperator{\SSBim}{SSBim}
\DeclareMathOperator{\Coh}{Coh}
\DeclareMathOperator{\Tr}{Tr}

\newcommand{\CA}{\mathcal{A}}

\newcommand{\CL}{\mathcal{L}}
\newcommand{\CT}{\mathcal{T}}
\newcommand{\CI}{\mathcal{I}}
\newcommand{\CX}{\mathcal{X}}
\newcommand{\one}{\mathbf{1}}

\newcommand{\link}{\boldsymbol{\mathrm{Link}}} 
\newcommand{\Alink}{\mathrm{A}\link} 
\newcommand{\Alinkp}{\Alink^{+}} 
\newcommand{\tang}{\boldsymbol{\mathrm{Tan}}} 
\newcommand{\braid}{\boldsymbol{\mathrm{Braid}}} 
\newcommand{\foam}{\boldsymbol{\mathrm{Foam}}} 
\newcommand{\Afoamp}{\mathrm{A}\foamp} 
\newcommand{\AKhR}{\mathrm{AKhR}} 
\newcommand{\KhR}{\mathrm{KhR}} 

\newcommand{\Prop}{\cat{P}}
\newcommand{\Proph}{\hat{\cat{P}}}
\newcommand{\Propo}{\overline{\cat{P}}}

\newcommand{\Hgenuf}[1]{\left\llbracket #1 \right\rrbracket}
\newcommand{\Hgen}[1]{\left\llbracket #1 \right\rrbracket^{\mathrm{fr}}}
\newcommand{\Stang}{\mathrm{S}\tang} 
\newcommand{\Slink}{\mathrm{S}\link} 
\newcommand{\Web}{\boldsymbol{\mathrm{Web}}}
\newcommand{\DecWeb}{\boldsymbol{\mathrm{DecWeb}}}

\newcommand{\foamp}{\foam^+} 

\newcommand{\NAfoamp}{N\Afoamp} 
\newcommand{\iAfoam}{\Afoamp} 
\newcommand{\iAfoamcc}{\iAfoam_{S_1}} 
\newcommand{\eqSfoam}{N\foam(\mathrm{S})} 
\newcommand{\Nfoam}{N\foam} 
\newcommand{\NAfoam}{N\mathrm{A}\foam} 

\newcommand{\Kom}{\cat{Kom}} 
\newcommand{\Komh}{\cat{K}^b} 

\newcommand{\Sym}{\mathrm{Sym}} 

\newcommand{\bV}{\mathchoice{{\textstyle\bigwedge}}%
    {{\bigwedge}}%
    {{\textstyle\wedge}}%
    {{\scriptstyle\wedge}}}

\newcommand{\bVq}{\bigwedge_q}
\newcommand{\Uq}{{\bf U}_q}

\newcommand{\slN}{\mf{sl}_N}

\newcommand{\slnn}[1]{\mf{sl}_{#1}}
\newcommand{\glN}{\mf{gl}_N}

\newcommand{\glnn}[1]{\mf{gl}_{#1}}
\newcommand{\cat}[1]{\ensuremath{\mbox{\bfseries {\upshape {#1}}}}}
\newcommand{\hh}{\mf{h}}
\def\C{{\mathbbm C}}
\def\N{{\mathbbm N}}
\def\R{{\mathbbm R}}
\def\Z{{\mathbbm Z}}
\def\Q{{\mathbbm Q}}

\def\cal#1{\mathcal{#1}}%

\def\mf{\mathfrak}

\newcommand{\CC}{\mathcal{C}}
\newcommand{\CE}{\mathcal{E}}

\newcommand{\Cube}{\mathrm{Cube}}
\newcommand{\KK}{\mathcal{K}}
\newcommand{\CF}{\mathcal{F}}
\newcommand{\E}{E}

\newcommand{\e}{\mathbf{e}}
\newcommand{\Sch}{\mathbb S}
\newcommand{\kk}{\mathbf{k}}
\newcommand{\kb}{\mathbbm{k}}
\newcommand{\id}{\mathrm{id}}

\newcommand{\Ob}{\mathrm{Ob}}

\author{Eugene Gorsky}
\address{Department of Mathematics, University of California, Davis}
\email{egorskiy@math.ucdavis.edu}

 \author{Paul Wedrich}
 \address{Former: Mathematical Sciences Institute\\ The Australian National University \\ Australia.
 Current: Fachbereich Mathematik, Universit\"{a}t Hamburg, Bundesstra\ss{}e 55, 20146 Hamburg, Germany}
 \email{paul.wedrich@uni-hamburg.de}
 \urladdr{https://paul.wedrich.at}

\title{Evaluations of annular Khovanov--Rozansky homology}

\begin{document}

\begin{abstract}
We describe the universal target of annular Khovanov--Rozansky link homology functors as the homotopy category of a free symmetric monoidal linear category generated by one object and one endomorphism. This categorifies the ring of symmetric functions and admits categorical analogues of plethystic transformations, which we use to characterize the annular invariants of Coxeter braids. Further, we prove the existence of symmetric group actions on the Khovanov--Rozansky invariants of cabled tangles and we introduce spectral sequences that aid in computing the homologies of generalized Hopf links. Finally, we conjecture a characterization of the horizontal traces of Rouquier
complexes of Coxeter braids in other types.\end{abstract}

\maketitle

\tableofcontents

\section{Introduction}

The positive part of the HOMFLY-PT skein algebra of the annulus is defined as a linear span of annular closures of braids, modulo certain skein relations. A classical result of Turaev \cite{Turaev} states that this skein algebra is isomorphic to the algebra $\Lambda_q$ of symmetric functions in infinitely many variables over $\Z[q^{\pm 1}]$. 
In particular, to any braid one can associate a symmetric function which is invariant under conjugation of the braid. 

Conversely, many interesting symmetric functions and relationships between them can be represented in terms of (colored) braid closures. For example, if Schur functions correspond to the colored unknots, then certain ``plethystically transformed" skew Schur functions $s_{\lambda/\mu}[X(q-q^{-1})]$ are represented by ``Coxeter braids"  (see Section~\ref{sec:classical} for precise definitions).

Furthermore, the skein of the annulus acts on the (relative) skein of the disk. In particular, after an extension of scalars, there is a homomorphism of $\Lambda_q$ to the Hecke algebra $H_n$ for any $n$, and its image coincides with the center of $H_n$.

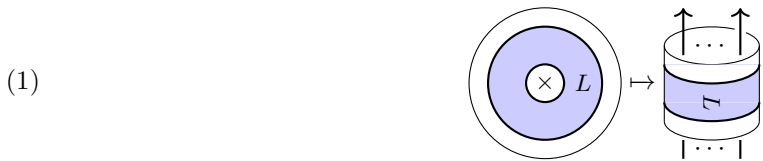
\begin{equation}
\label{eq:encircle}
\begin{tikzpicture}[scale=.5,anchorbase]
 \draw (0,0) circle (2);
 \draw[thick,->, fill=blue, fill opacity=.2] (0,0) circle (1.5);
 \draw[thick,->, fill=white, opacity=1] (0,0) circle (.5);
 \node at (0,0) {\small $\times$};
\node at (1,0) {\small $L$};
\end{tikzpicture}
\xmapsto{}
\begin{tikzpicture}[scale=.5,anchorbase]
 \draw[thick] (0.75,1) to (0.75,0.5);
 \draw[thick] (-0.75,1) to (-0.75,0.5);
 \draw (0,3.5) ellipse (1.25 and .5);
\draw (0,1.5) [partial ellipse=180:360:1.25 and .5];
\fill[blue, fill opacity=.2] (0,2) [partial ellipse=180:360:1.25 and .5];
\fill[blue, fill opacity=.2] (-1.25,2) rectangle (1.25,3);
\fill[white, fill opacity=1] (0,3) [partial ellipse=180:360:1.25 and .5];
\draw[thick] (0,2) [partial ellipse=180:360:1.25 and .5];
\draw[thick] (0,3) [partial ellipse=180:360:1.25 and .5];
\draw (-1.25,3.5) to (-1.25,1.5) (1.25,3.5) to (1.25,1.5) ;
\draw[white, line width=.15cm] (0.75,3.8) to (.75,4.5);
\draw[thick,<-](.75,4.5) to (0.75,3.25);
\draw[thick,<-](-.75,4.5) to (-0.75,3.25);
\node at (0,2) {\small \rotatebox{270}{$L$}};
\node at (0,3.5) {$\cdots$};
\node at (0,.75) {$\cdots$};
\end{tikzpicture}
\end{equation}

The motivation for this paper is to study lifts of this homomorphism to the categorified level.

\subsection{The annular category}

In a series of recent papers \cite{QR2,QRS} Queffelec, Rose and Sartori categorified the skein of the annulus using {\it annular Khovanov-Rozansky homology}. The target for this annular link homology functor is a monoidal category whose objects are (complexes of) oriented webs in the annulus, and the morphisms are given by annular foams. They prove that this category is generated by collections of $\bV^k$-colored essential unknots, and provide an explicit algorithm of simplification of a given web to this basis. The monoidal structure is given by placing one annulus inside another.
We reformulate their result and prove the following:

\begin{thm}
\label{th: intro schur category}
The Karoubi completion (or bounded homotopy category) of the category of positive annular webs and foams is equivalent to (the bounded homotopy category of) the free symmetric monoidal graded Karoubian category $\Proph$ generated by a single object $E$ (corresponding to the uncolored essential circle) with an endomorphism
$x\in \End(E)$ (corresponding to a dotted cylinder on the circle) of degree two. Under this equivalence, the $\bV^k$-colored unknot corresponds to the antisymmetric component in $E^{\otimes k}$. 
\end{thm}

In other words, the target of the annular Khovanov--Rozansky invariant can be thought of as a category of complexes of Schur functors of $E$, which categorify the corresponding symmetric functions in $\Lambda_q$. We will call the bounded homotopy category $\Komh(\Proph)$ the \textit{annular category}.

\begin{rem}
It is important to mention that we work in characteristic zero, where the representation theory of $S_n$ is semisimple, and Schur functors are well-defined. In finite characteristic, one may need to use the formalism of {\it strict polynomial functors} \cite{FS,HY,HY2,HTY}, but we do not pursue it in this paper.
\end{rem}

It is conjectured \cite[Conjecture 5.4]{QR2} that every annular web is actually isomorphic to a direct sum of collections of $\bV^k$-colored essential unknots (that is, to a complex concentrated in one homological degree). Here we prove that at least after Karoubi completion this is indeed the case:

\begin{thm}
\label{th:intro web as one term complex}
Every positive annular web is isomorphic in the Karoubi completion of the positive annular foam category to a direct sum of Schur functors of the uncolored essential circle $E$.
\end{thm}

We prove this theorem as Corollary~\ref{cor:annwebSchur}.
As a consequence, any annular chain complex, in particular the invariant of an annular braid closure, is isomorphic (and not just homotopy equivalent) to a complex of such Schur functors.  Two further consequences are the following.

\begin{cor} The symmetric function corresponding to any annular web in the skein of the annulus is Schur positive.
\end{cor}

\begin{cor} The Karoubi completion of the horizontal (i.e. monoidal) trace of the monoidal category of Soergel bimodules of type $A_{n-1}$ is equivalent to $\C[S_n]\ltimes \C[x_1,\dots,x_n]-\mathrm{gpmod}$.
\end{cor}

See Section \ref{sec: traces} for more details on horizontal traces.

\subsection{Spectral sequences}

The annular simplification of Khovanov--Rozansky invariants is still possible if the annular link appears as a cabling of a component of a framed link in $\R^3$. To make sense of this claim, we first need to explain what Schur functors are in this framework. Let $L$ be a framed link and $K$ a distinguished component and $\lambda$ a partition of $n$. Consider the link $L(K^n)$ given by the $n$-fold parallel cabling of the component $K$ in $L$ (this uses the framing). By functoriality of the Khovanov--Rozansky functor, the braid group $B_n$ on $n$-strands acts on $\KhR(L(K^n))$ by braiding parallel circles in the cabling around each other through isotopy-cobordisms. Moreover, the braid group actions associated to cablings of different components of $L$ commute. In fact, these braid group actions factor through the symmetric group. This generalizes a result of Grigsby--Licata--Wehrli \cite{GLW} for $\slnn{2}$ Khovanov homology.

\begin{thm}
\label{thm:intro symmetric glw}
The action of $B_n$ on $\KhR(L(K^n))$ factors through the symmetric group $S_n$. 
\end{thm}

We can now define colorings by Young diagrams. For this, let $L$ be a link with components $K_1,\dots,K_l$ and we denote by $L(K_1^{n_1},\dots,K_l^{n_l})$ the result of $n_i$-fold parallel cabling of the components $K_i$ in $L$ for $1\leq i\leq l$.

\begin{defi}
Let $L$ be as above and $\lambda_1\dots,\lambda_l$ Young diagrams with $|\lambda_i|=n_i$. Let $L(K_1^{\lambda_1}\cdots K_l^{\lambda_l})$ denote the link $L$ with \textit{color} label $\lambda_i$ on the component $K_i$ for $1\leq i \leq l$. Then we define $\KhR(L(K_1^{\lambda_1}\cdots K_l^{\lambda_l}))$ as the image of the tensor product of Young idempotents of shape $\lambda_i$ in $\C[S_{n_i}]$ for $1\leq i \leq l$ on $\KhR(L(K_1^{n_1}\cdots K_l^{n_l}))$.  
\end{defi}

Note that these colored link homologies are distinct from the colored homologies constructed by inserting categorified projectors into cables of knots. In particular, in the case of finite-rank Khovanov--Rozansky homology, the colored homologies described here are finite-dimensional for all colors. 

Next, we show how annular simplification can be used to approximate the homology of links $L$ which split into a Hopf pairing of sub links $L_1$ and $L_2$ via a spectral sequence. 

\begin{thm} 
\label{th:intro filtered}
Let $L$ be a link which is a satellite of a framed Hopf link $H(L_1,L_2)$ where $L_1$ and $L_2$ are annular links and $L_1$ is a braid closure. Suppose that the annular invariant of $L_1$ is isomorphic to a chain complex $C^*(L_1)$ of Schur functors of $E$. Then the chain complex associated to $H(L_1,L_2)$ is homotopy equivalent to a filtered chain complex, whose associated graded is given by a direct sum of complexes associated to $H(C^i(L_1),L_2)$, where $C^i(L_1)$ is a direct sum of Schur-colored unknots as specified by the chain groups of $C^*(L_1)$. Moreover, the differential of filtration degree one is induced by the differential on $C^*(L_1)$.
\end{thm}
The following is a direct consequence.

\begin{cor} 
\label{th: intro spectral sequence}
For $L=H(L_1,L_2)$ and the annular complex of Schur functors $C^*(L_1)$ as in the theorem above, there exists a spectral sequence computing $\KhR(H(L_1,L_2))$, whose $E_1$ page has chain groups $\KhR(H(\Sch^\lambda,L_2))$ where the Schur functors $\Sch^\lambda$ range through the chain groups of $C^*(L_1)$, and the differential $d_1$ is induced by the differential in $C^*(L_1)$.
\end{cor}

\begin{rem}
An important caveat regarding Theorem \ref{th:intro filtered} is that the annular chain complex of $L_1$ may in general not be assumed to be a minimal complex. Gaussian elimination on the annular complex of $L_1$ typically breaks the filtration which is the main point of the theorem. Thus we restrict to isomorphic replacements by complexes of Schur functors.
\end{rem}

\subsection{Positive Coxeter braids}

Next, we describe another natural generating set in the annular category, which appears in the image of annular Khovanov--Rozansky functors, namely the images of closures of Coxeter braids. 

\begin{thm}
\label{th: intro positive coxeter}
Let $C^+_n$ denote the annular Khovanov--Rozansky invariant of the closure of the braid $\sigma_{n-1}\cdots\sigma_1$ on $n$ strands in $\Komh(\Proph)$. Then 
$$
C^+_n\simeq \left[q^{n-1}\uwave{\bV^{n}(E)}\to q^{n-3}\Sch^{2,1^{n-2}}(E)\to \cdots\to q^{3-n}\Sch^{n-1,1}(E)\to q^{1-n}S^n(E)\right].
$$
\end{thm}

We describe the differential in this complex explicitly.
We also describe the spaces of morphisms between various products of $C^+_n$.

\begin{thm}
\label{th: intro morphisms coxeters}
We have
$$
\End(C^+_n)=\bV(\xi_1,\ldots,\xi_{n-1})\otimes \C[x],
$$
where $\xi_i$ are odd variables of homological degree $-1$ and $q$-degree $2i-2$ and $x$ has $q$-degree $2$. 
\end{thm}

Moreover, there are natural ``merge" and ``split maps"
$$
M_{m,n}\colon C^+_m\otimes C^+_n\to q C^+_{m+n}[1],\quad  S_{m,n}\colon C^+_{m+n}\to q C^+_m\otimes C^+_{n}.
$$
and we expect that all morphisms between tensor products of $C^+_n$ are generated by these and the action of $\xi_i$ and $x$.
\subsection{Other Coxeter lifts}

We also describe the annular homology of other lifts of the Coxeter element $s_{n-1}\cdots s_1\in S_n$ to the braid group. Such lifts $\sigma_\epsilon :=\sigma_{n-1}^{\epsilon_{n-1}}\cdots \sigma_{1}^{\epsilon_{1}}$ are parametrized by binary sequences $\epsilon\in \{+1,-1\}^{n-1}$. Given such a sequence $\epsilon$, consider a {\it ribbon} $\nu({\epsilon})$, a skew Young diagram obtained by the following rule: we start from a box,  move right if we see a $+1$ in $\epsilon$ and move down if we see a $-1$. 
For example, for $\epsilon=(+1,+1,-1,+1,-1,-1,+1)$ we get the following shape (which represents the skew shape $5443/332$):

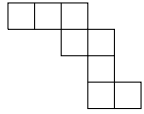
\begin{figure}[ht!]
\begin{center}
\begin{tikzpicture}[scale=.35]
\draw(0,3)--(0,4)--(3,4)--(3,3)--(0,3);
\draw(1,3)--(1,4);
\draw(2,4)--(2,2)--(3,2)--(3,3);
\draw (3,2)--(4,2)--(4,3)--(3,3);
\draw (3,2)--(3,0)--(4,0)--(4,2);
\draw (3,1)--(5,1)--(5,0)--(4,0);
\end{tikzpicture}
\end{center}
\label{fig:ribbon}
\end{figure}

To such skew shape one can associate a skew Schur function $s_{\nu(\epsilon)}$ \cite{MD} which decomposes into usual Schur functions with positive coefficients. For example, the shape above corresponds to 
$$
s_{5443/332}=s_{3, 2, 2, 1}+s_{3, 3, 1, 1} + 2s_{4, 2, 1, 1} + s_{4, 2, 2} + s_{4, 3, 1} + s_{5, 1, 1, 1} + s_{5, 2, 1}.
$$
More precisely (see Section \ref{sec:coxlift} for details) for ribbon skew shapes there exists a canonical left ideal $V_{\epsilon}\subset \C[S_n]$ with Frobenius character $s_{\nu(\epsilon)}$,
and  $\C[S_n]\cong \oplus_{\epsilon}V_{\epsilon}$.  
We let $p_{\epsilon}\in \C[S_n]$ denote the idempotent projecting to $V_{\epsilon}$.

\begin{thm}[{Theorem~\ref{thm:allCoxeter}}]
\label{thm:intro other coxeters}
Let $U_n=\text{Span}(x_i-x_{i+1})$ be the $(n-1)$-dimensional reflection representation of $S_n$, there is a natural $S_n$-equivariant map $D:U_n\otimes E^{\otimes n}\to E^{\otimes n}$.
Consider the Koszul complex 
$$
\Cube_n:=(\bV^{\bullet}U_n\otimes E^{n}, D).
$$
Then the annular Khovanov-Rozansky complex $C_\epsilon$ of $\sigma_1^{\epsilon_1}\cdots \sigma_{n-1}^{\epsilon_{n-1}}$ satisfies
$$
C_\epsilon[|\epsilon|_+] \simeq p_{\epsilon}\cdot \Cube_n.
$$
 where $|\epsilon|_+$ denotes the number of entries $+1$ in $\epsilon$. 
\end{thm}
\begin{exa}
For $\epsilon=(+1,\cdots,+1)$ the skew shape $\nu(\epsilon)$ has one row, so $s_{\nu(\epsilon)}=s_{n}$,
the corresponding representation $V_{\epsilon}$ is trivial, and the corresponding projector $p_{\epsilon}$ is
the symmetrizer. Therefore $p_{\epsilon}\cdot \Cube_n=(\Cube_n)^{S_n}$. In Lemma \ref{lem:invcube}
 we check that this indeed agrees, up to a homological shift, with the description of the annular complex for the positive Coxeter lift in Theorem \ref{th: intro positive coxeter}, and yields immediately the differentials in it. 
\end{exa}

By a result of Solomon, the analogues of the projectors $p_{\epsilon}$ can be defined for all finite Coxeter groups.
We conjecture that Theorem \ref{thm:intro other coxeters} can be generalized too, see Conjecture \ref{conj: coxeters other types}.

\subsection{Organization of the paper}

In Section \ref{sec:classical} we list various important results about the skein algebra of the annulus, following Turaev \cite{Turaev}, Aiston and Morton \cite{AistonMorton,Morton}. We identify this skein with the algebra of symmetric functions in infinitely many variables, and identify certain closed braids with explicit symmetric functions. In particular, we prove Theorem \ref{thm:gencoxdecat} which is a decategorified version of Theorem \ref{thm:intro other coxeters}. In Section \ref{sec: symmetric categories} we use Schur functors in symmetric monoidal categories to describe an explicit categorification of the algebra of symmetric functions and a plethystic transformation. In 
Section \ref{sec:KhR} we define and study the category of webs and foams and the corresponding Khovanov-Rozansky functor. We prove Theorems \ref{th: intro schur category} and \ref{th:intro web as one term complex}.

In Section \ref{sec:categorified Coxeter} we identify the annular complexes for all lifts of the Coxeter element to the braid group and prove Theorems \ref{th: intro positive coxeter} and \ref{thm:intro other coxeters}. 

In Section \ref{sec:wrapped} we describe the operation of ``wrapping'' an annular link around a braid, and prove Theorem \ref{th: intro spectral sequence}. In Section \ref{sec:other types} we briefly discuss a conjectural description of annular homology (or, rather, a class in the horizontal trace) for Coxeter lifts outside of type A. Finally, in the appendix we list some useful facts from homological algebra, in particular, on splitting of homotopy idempotents and triangulated Karoubian categories.

\section*{Acknowledgements}

This paper would not have happened without Jake Rasmussen, and started as a joint project with him.
We would like to thank Anna Beliakova, Michael Ehrig, Ben Elias, Matt Hogancamp, Tony Licata, Scott Morrison, Andrei Negu\cb t, Hoel Queffelec, Dave Rose, and Catharina Stroppel for useful discussions.  Part of the work on this article was conducted during the program ``Homology theories in low dimensional topology'' at the Isaac Newton Institute for Mathematical Sciences [EPSRC Grant Number EP/K032208/1] and workshops at the American Institute of Mathematics, the Erwin Schr\"{o}dinger Institute, and the Hausdorff Research Institute for Mathematics. We thank these institutions for their hospitality and excellent working conditions. 

\section*{Funding}
The work of E.~G. was
partially supported by the NSF grants DMS-1700814 , DMS-1760329, and the Russian Academic Excellence Project 5-100.
The work of E.~G. in section \ref{sec:categorified Coxeter} was also supported by the grant RSF-16-11-10018. The work of P.~W. was supported by the  Leverhulme Trust [Research Grant RP2013-K-017] and the Australian Research Council Discovery Projects ``Braid groups and higher representation theory'' and ``Low dimensional categories'' [DP140103821, DP160103479].

\section{The classical story}
\label{sec:classical}

In this section we recall the classical constructions related to the skein algebra of the annulus. 
 
\subsection{The skein of the annulus}

Let $A$ denote an annulus on the plane.
The closure of a braid is a link in $A\times [0,1]$. We define the {\it  positive part of the skein of the annulus} $\Sk^{+}(A)$
as the $\Z[q^{\pm 1}]$-linear span of all braid closures, considered up to regular isotopy, modulo the HOMFLY skein relation:
$$
\begin{tikzpicture}[anchorbase,scale=.25]
	\draw [very thick, directed=0.85] (1,-1.7) to [out=90,in=270](-1,1.7);
	\draw [white, line width=.15cm] (-1,-1.7) to [out=90,in=270](1,1.7);
	\draw [very thick, directed=0.85] (-1,-1.7) to [out=90,in=270](1,1.7);
	\draw [thick, dashed, opacity=0.4] (0,0) circle (2cm);
\end{tikzpicture}
\quad - \quad 
\begin{tikzpicture}[anchorbase,scale=.25]
	\draw [very thick, directed=0.85] (-1,-1.7) to [out=90,in=270](1,1.7);
	\draw [white, line width=.15cm] (1,-1.7) to [out=90,in=270](-1,1.7);
	\draw [very thick, directed=0.85] (1,-1.7) to [out=90,in=270](-1,1.7);
	\draw [thick, dashed, opacity=0.4] (0,0) circle (2cm);
\end{tikzpicture}
\quad = (q-q^{-1}) \quad
\begin{tikzpicture}[anchorbase,scale=.25]
	\draw [very thick, directed=0.85] (1,-1.7) to [out=90,in=270](1,1.7);
	\draw [very thick, directed=0.85] (-1,-1.7) to [out=90,in=270](-1,1.7);
	\draw [thick, dashed, opacity=0.4] (0,0) circle (2cm);
\end{tikzpicture}.
$$
This can be given an algebra structure by stacking
$(A\times [0,1])\sqcup (A\times [1,2])=A\times [0,2]$. We will refer to this operation as to {\it  skein product}, 
which should not be confused with the product of braids. The skein product of two braid closures is isotopic to the disjoint union of the two braid closures, considered as living in two annuli, one outside of another. An Eckmann-Hilton argument then implies that $\Sk^{+}(A)$ is a commutative algebra with respect to the skein product.

\begin{thm}\cite{Turaev}
\label{thm: lambda}
The skein algebra $\Sk^{+}(A)$ is isomorphic to the algebra $\Lambda_q$ of symmetric functions in infinitely many variables over $\Z[q^{\pm 1}]$.
\end{thm}

There are several versions of the isomorphism in Theorem \ref{thm: lambda} which differ by automorphisms of the symmetric function ring, possibly after extending scalars. We outline one of them in the next section.

\subsection{Universal Hecke trace and symmetric functions} 
The Hecke algebra $H_n$ is defined as the quotient of $\Z[q^{\pm 1}]\mathrm{Br}_n$ by the HOMFLY skein relation shown above. It is easy to see that $H_n$ is spanned by the images of positive permutation braids. Moreover, taking braid closures in the annulus defines a linear map $\Tr:H_n\to \Sk^{+}(A)$. Since the closures of conjugate braids represent the same link in the annulus, we have $\Tr(ab)=\Tr(ba)$. In fact, it is easy to see from the construction that 
$$
\Sk^{+}(A) \cong {H_n/[H_n,H_n]}.
$$ 
In other words, any linear map $f:H_n\to V$ such that $f(ab)=f(ba)$ factors through the map $\Tr:H_n\to \Sk^{+}(A)$.
  
The identification of $\Sk^{+}(A)$ with $\Lambda_q$ is also transparent in this construction. Indeed, the irreducible representations $V_{\lambda}$ of $H_n$ are classified by Young diagrams $\lambda$ with $n$ boxes. Define the map
$$
\Tr_{\Lambda_q}:H_n\to \Lambda_q, \Tr_{\Lambda_q}(x)=\sum_{\lambda}\Tr(x,V_{\lambda})s_{\lambda}
$$
where $s_{\lambda}$ is the Schur function. Clearly, $\Tr_{\Lambda_q}(ab)=\Tr_{\Lambda_q}(ba)$, so by the above 
$\Tr_{\Lambda_q}$ factors through $\Sk^{+}(A)$:
\begin{equation}
\label{eq: tr lambdaq}
\Tr_{\Lambda_q}(x)=i(\Tr(x)),\ i:\Sk^{+}(A)\to \Lambda_q.
\end{equation}
Theorem \ref{thm: lambda} states that $i$ is an isomorphism.

\begin{rem}
The Hecke algebra can also be used to study invariants of oriented tangles with $n$ inputs and $n$ outputs. More precisely, we consider the ring  $\kb:=\Z[q^{\pm 1}, a^{\pm 1}, (q^k-q^{-k})^{-1}]$ for all $k>1$, and the $\kb$-module $\Sk(n,n)$ spanned by all framed oriented tangles in an axis-parallel rectangle in $\R^2$, with $n$ inputs on the bottom boundary and $n$ outputs on the top, modulo the HOMFLY skein relation and: 
$$
\begin{tikzpicture}[anchorbase,scale=.25]
	\draw [very thick] (0,0) circle (1cm); 
	\draw [thick, dashed, opacity=0.4] (0,0) circle (2cm);
\end{tikzpicture}
\quad =\quad \frac{a-a^{-1}}{q-q^{-1}} \;\begin{tikzpicture}[anchorbase,scale=.25]
	\draw [thick, dashed, opacity=0.4] (0,0) circle (2cm);
\end{tikzpicture}
\quad, \quad
\begin{tikzpicture}[anchorbase,scale=.4]
\draw [very thick, directed=0.9] 	(0.5,-.5) to [out=180,in=270] (0,1) to (0,1.25);
	\draw [white, line width=.15cm] (0,-1) to [out=90,in=180](0.5,0.5) ;
	\draw [very thick] (0,-1.25) to (0,-1) to [out=90,in=180](0.5,0.5) to [out=0,in=90] (1,0) to [out=270,in=0] (0.5,-.5);
	\draw [thick, dashed, opacity=0.4] (0,0) circle (1.25cm);
\end{tikzpicture}
\quad =\quad
  -a^{-1}\;
\begin{tikzpicture}[anchorbase,scale=.4]
	\draw [very thick,directed=.55] (0,-1.25) to (0,1.25) ;
	\draw [thick, dashed, opacity=0.4] (0,0) circle (1.25cm);
\end{tikzpicture}
$$
It is known \cite{AistonMorton,MortonTraczyk} that $\Sk(n,n)$  (with respect to
composition) is isomorphic to the Hecke algebra $H_n\otimes \kb$ with scalars
extended to $\kb$. The extended trace is denoted by $\Tr_{\Lambda_{a,q}}\colon
\Sk(n,n) \to \Lambda_q\otimes \kb =: \Lambda_{a,q}$. 
\end{rem}

The universal trace can be specialized to the Jones-Ocneanu trace on the Hecke algebra which yields the HOMFLY-PT polynomial or $\slnn{N}$ Reshetikhin-Turaev invariants of links $L\subset S^3$ presented as braid closures. 

\begin{prop}
\label{prop:evaluation}
Let $f_L\in \Lambda_q$ correspond to a braid closure $L$ in the thickened annulus under the isomorphism \eqref{eq: tr lambdaq}.  Then the $\slnn{N}$ Reshetikhin-Turaev invariants $\langle L\rangle_{N}$ and the HOMFLY-PT polynomial $\langle L\rangle$ can be computed as follows:
\begin{enumerate}
\item[(a)] $\langle L\rangle_{N}=f_L(q^{N-1},q^{N-3},\ldots,q^{1-N})$
\item[(b)] $\langle L\rangle = \varepsilon(f_L),$ where $\varepsilon:\Lambda_q\to \Lambda_{a,q}$ is the ring homomorphism
defined by $$\varepsilon(p_k)=(a^{k}-a^{-k})/(q^k-q^{-k}).$$
\end{enumerate}
Here $p_k$ denotes the k-th power sum symmetric function.
\end{prop}

\begin{proof}
Part (a) is well-known. To obtain (b), observe that by (a) 
$$
\langle p_k \rangle_{N}= p_k(q^{N-1},\ldots,q^{-N-1})=q^{k(N-1)}+\ldots+q^{-k(N-1)}=\frac{q^{kN}-q^{-kN}}{q^k-q^{-k}}=\frac{a^{k}-a^{-k}}{q^k-q^{-k}}\vline_{\;a=q^N}.\vspace{-5mm}
$$
\end{proof}

\subsection{Coxeter braids}
\label{sec:coxeter decategorified} In this section we compute the images of
braid lifts of Coxeter elements in $\Lambda_q$. To this end, we introduce a
particular plethysm operation. Recall that the power sum symmetric functions
$p_n$ for $n>1$ give an algebraically independent set of generators of
$\Lambda_q\otimes_{\Z} \Q$.

\begin{lem} There exists a unique $\Z[q^{\pm 1}]$-algebra endomorphism of
$\Lambda_q$ which sends $p_k$ to its scalar multiple $p_k(q^{-k}-q^{k})$ for all $k\geq 1$. 
\end{lem}
If $f\in \Lambda_q$ is a symmetric function, we denote its image under this endomorphism by $f[X(q^{-1}-q)]$.
\begin{proof}
After extending to scalars to $\Q[q^{\pm 1}]$, it is clear that there is a unique endomorphism with these properties. The fact that it is well-defined over $\Z[q^{\pm 1}]$ follows from the following lemma, which can be used to compute the images of the algebraically independent integral generators given by the elementary (or complete) symmetric functions $e_n$ (or $h_n$) for $n\geq 0$.
\end{proof}
 
\begin{lem} \label{lem:plethcomplete}
We have
$$
\frac{1}{q^{-1}-q}e_n[X(q^{-1}-q)] = \sum_{i=0}^{n-1} (-1)^{n-1-i}q^{n-1-2i}s_{n-i,1^i},
$$
$$
\frac{(-1)^{n-1}}{q^{-1}-q}h_n[X(q^{-1}-q)] = \sum_{i=0}^{n-1} (-1)^{i-n+1}q^{2i-n+1}s_{n-i,1^i}.
$$

\end{lem}
\begin{proof}
Consider the identity of generating functions
$$
\sum_{k=0}^{\infty}e_kz^k=\exp\left(\sum_{k=1}^{\infty}\frac{p_k(-z)^k}{k}\right),
$$
which implies
\begin{multline}
\label{eh}
\sum_{k=0}^{\infty}e_k[X(q^{-1}-q)]z^k=\exp\left(\sum_{k=1}^{\infty}\frac{p_k(q^{-k}-q^{k})(-z)^k}{k}\right)=\\
\frac{\exp\left(\sum_{j=1}^{\infty}\frac{p_kq^{-k}(-z)^k}{k}\right)}{\exp\left(\sum_{j=1}^{\infty}\frac{p_kq^{k}(-z)^k}{k}\right)}=
\left(\sum_{k=0}^{\infty}e_k q^{-k} z^k\right)\left(\sum_{k=0}^{\infty}(-1)^k h_kq^{k}z^k\right).
\end{multline}
By taking the coefficient at $z^n$, we get
\begin{align*}
e_n[X(q^{-1}-q)]&=\sum_{i=0}^{n}(-1)^{n-i}q^{n-2i}e_{i}h_{n-i}\\&=(-1)^nq^n h_n +\sum_{i=1}^{n-1}(-1)^{n-i}q^{n-2i}(s_{n-i,1^i}+s_{n-i+1,1^{i-1}})+q^{-n} e_n
\\
&=(q^{-1}-q)\sum_{i=0}^{n-1} (-1)^{n-1-i} q^{n-1-2i}s_{n-i,1^i}. 
\end{align*}
The other identity admits an analogous proof.
\end{proof}

The following proposition describes the trace of the positive Coxeter braid $\sigma_{n-1}\cdots \sigma_{1}$ in terms of the plethysm operation introduced above.

\begin{prop}
\label{positive coxeter decategorified}
We have $\Tr_{\Lambda_q}(\sigma_{n-1}\cdots \sigma_{1})=(-1)^{n-1}h_n[X(q^{-1}-q)]/(q^{-1}-q)$.
\end{prop}

\begin{proof}
The traces of $\sigma_{n-1}\cdots \sigma_{1}$  in various representations of the Hecke algebra can be found in \cite[Section 9]{Jones}. Such a trace in $V_{\lambda}$ vanishes if $\lambda$ is not a hook, and equals $(-1)^{i}q^{n-1-2i}$ for the hook $\lambda_i=(n-i,1^i)$. It remains to apply Lemma \ref{lem:plethcomplete}.  
\end{proof}

\begin{rem}
In \cite[Theorem 3.6]{Morton} Morton obtained this result (in the form of equation \eqref{eh}) by purely skein-theoretic methods.
\end{rem}

\begin{rem}
In \cite{Turaev} Turaev identified the entire skein of the annulus $\Sk(A)$ with a polynomial algebra in variables $l_k$ with $k\in \Z\setminus \{0\}$. These $l_k$ can be chosen to be the images of the closures of positive Coxeter braids on $|k|$ strands, winding positively or negatively around the annulus. The positive half has generators $l_k$ for $k\geq 1$ and is thus isomorphic to the ring of symmetric functions. 
\end{rem}

We can also describe the annular invariants for all lifts of the Coxeter element $s_{n-1}\cdots s_{1}\in S_n$.
Such a lift has the form $\sigma_{n-1}^{\epsilon_{n-1}}\cdots \sigma_{1}^{\epsilon_{1}}$ for some $\epsilon_i=\pm 1$.

\begin{defi}\label{def:epsilona} Consider the bijection between the set $\{\pm 1\}^{n-1}$ and the set $C(n)$ of compositions of $n$ with strictly positive parts, given as follows. To a sequence $\epsilon=(\epsilon_1,\dots, \epsilon_{n-1})$ we associate the composition $(a_1,\ldots,a_s)$ of $n$, which is determined by
$$
\epsilon_{a_1}=\epsilon_{a_1+a_2}=\ldots=\epsilon_{a_1+\ldots+a_{s-1}}=-1,
$$ and $\epsilon_i=+1$ for all other $i$. Note that this implies $a_s=n-(a_1+\ldots+a_{s-1})$. \end{defi}
For example,  
$$
\epsilon=(+1,+1,-1,+1,-1,-1,+1)\ \leftrightarrow\ a=(3,2,1,2).
$$
For a composition $(a_1,\ldots,a_s)$ we define its length $l(a)=s$. We will use the partial order on $C(n)$: 
$a\preceq b$ if $a$ refines $b$. In this order, $(n)$ is the maximal element (it corresponds to a sequence of $+1$'s) and $(1,\ldots,1)$ is the minimal one (it corresponds to a sequence of $-1$'s).

\begin{defi}
Let $a$ be a composition of $n$. We define a symmetric function
$$
\Psi(a)=\sum_{a\preceq b\in C(n)} (-1)^{l(b)-l(a)}h_{b_1}\cdots h_{b_{l(b)}},
$$
where $h_k$ are complete symmetric functions and $l(a), l(b)$ are the lengths of $a$ and $b$ as above.
\end{defi}

\begin{exa}
In the above example $a=(3,2,1,2)$ we get 
$$
\Psi(a)=h_3h_2h_1h_2-h_5h_1h_2-h_3h_3h_2-h_3h_2h_3+h_6h_2+h_3h_5+h_5h_3-h_8.
$$
\end{exa}

\begin{lem}
\label{lem:JT}
We have $\Psi(a)=\det M(a)$, where
$$
M_{ij}(a)=\begin{cases}
h_{a_i+\ldots+a_j} & \text{if}\quad i\le j,\\
1 & \text{if}\quad i=j+1,\\
0 & \text{otherwise}.\\
\end{cases}
$$
\end{lem}

\begin{proof}
Straightforward from the recursive formula
\begin{equation}
\label{phi recursion}
\Psi(a_1,\ldots,a_s)=\Psi(a_1,\ldots,a_{s-1})h_s-\Psi(a_1,\ldots,a_{s-2},a_{s-1}+a_s).
\end{equation}
\end{proof}

\begin{exa}
In our running example we get
$$
\Psi(3,2,1,2)=\left|\begin{matrix}
h_3 & h_5 & h_6 & h_8\\
1 & h_2 & h_3 & h_5\\
0 & 1 & h_1 & h_3\\
0 & 0 & 1 & h_2\\
\end{matrix}
\right|
$$
\end{exa}

\begin{cor}
\label{cor:hooks}
For $\epsilon=(\underbrace{+1,\ldots,+1}_{k},\underbrace{-1,\ldots, -1}_{n-k-1})$ we have
$a=(k+1,\underbrace{1,\ldots,1}_{n-k-1})$ and $\Psi(a)=s_{k+1,1^{n-k-1}}$.
\end{cor}

\begin{proof}
Follows from the determinantal formula for $\Psi(a)$ and the Jacobi-Trudy formula for $s_{k+1,1^{n-k-1}}$.
\end{proof}

For general $a$, the Schur expansion for $\Psi(a)$ is more complicated.

\begin{exa}
One can check that
$$
\Psi(3,2,1,2)=s_{3, 2, 2, 1}+s_{3, 3, 1, 1} + 2s_{4, 2, 1, 1} + s_{4, 2, 2} + s_{4, 3, 1} + s_{5, 1, 1, 1} + s_{5, 2, 1}.
$$
\end{exa}
In particular, $\Psi(3,2,1,2)$ expands as a non-negative linear combination of Schur functions. To see that this is the case for any composition $a=(a_1,\ldots,a_s)$, consider a pair of partitions
$$
\lambda=(a_1+\ldots+a_s-s+1,a_1+\ldots+a_{s-1}-s+2,\ldots,a_1+a_2-1,a_1),
$$ 
$$
\mu=(a_1+\ldots+a_{s-1}-s+1,\ldots,a_1+a_2-2,a_1-1).
$$
It is easy to see that $\mu\subset \lambda$ and $\lambda-\mu$ is a connected $n$-ribbon with rows of size $a_i$. Now, the determinantal expression for $\Psi(a)$ in Lemma~\ref{lem:JT} is precisely the Jacobi-Trudi formula defining the skew Schur function $s_{\lambda/\mu}$ \cite{MD}:

$$
s_{\lambda/\mu}=\det(\lambda_i-\mu_j-i+j)=\Psi(a).
$$

\begin{exa}
In our running example for $a=(3,2,1,2)$ we get $\lambda=(5,4,4,3)$, $\mu=(3,3,2)$ (see Figure \ref{fig:ribbon}) and
$\Psi(a)=s_{5443/332}$.
\end{exa}

\begin{cor}
For all compositions $a$ the coefficients of $\Psi(a)$ in the Schur basis are nonnegative.
\end{cor}
\begin{proof}
We have $\Psi(a)=s_{\lambda/\mu}$ with $\lambda$ and $\mu$ as described above. The lemma now follows since skew Schur polynomials expand in Schur polynomials with nonnegative coefficients given by the Littlewood-Richardson rule:
$(s_{\lambda/\mu},s_{\nu})=(s_{\lambda},s_{\mu}s_{\nu}).$
\end{proof}

\begin{exa}
Let us determine the polynomials $\Psi(a)$ for $n=4$. We get the following table:

\begin{center}
\begin{tabular}{|c|c|c|c|c|}
\hline
$\epsilon$ & $a$ & $\lambda$ & $\mu$ & $\Psi(a)=s_{\lambda/\mu}$\\
\hline
$(+1,+1,+1)$ & $4$ & $4$ & $\emptyset$ & $s_4$ \\
\hline
$(+1,+1,-1)$ &  $31$ & $33$ & $2$ &  $s_{3,1}$ \\
\hline
$(+1,-1,+1)$ & $22$ & $32$ & $1$ & $s_{2,2}+s_{3,1}$ \\
\hline
$(-1,+1,+1)$ & $13$ & $31$ & $\emptyset$ & $s_{3,1}$  \\
\hline
$(+1,-1,-1)$ & $211$ & $222$ & $11$ & $s_{2,1,1}$  \\
\hline
$(-1,+1,-1)$ & $121$ & $221$ & $1$ &  $s_{2,1,1}+s_{2,2}$ \\
\hline
$(-1,-1,+1)$ & $112$ & $211$ & $\emptyset$ & $s_{2,1,1}$ \\
\hline
$(-1,-1,-1)$ & $1111$ & $1111$ & $\emptyset$ &  $s_{1,1,1,1}$ \\
\hline
\end{tabular}
\end{center}

\end{exa}

\begin{lem}
We have $\sum_{a\in C(n)} \Psi(a)=h_1^n$.
\end{lem}

\begin{proof}
By definition, we have
$$
\sum_{a\in C(n)} \Psi(a)= \sum_{a\in C(n)} \sum_{a\preceq b\in C(n)} (-1)^{l(b)-l(a)} h_{b_1}\cdots h_{b_{l(b)}}.
$$
If one fixes $b$, it is easy to see that the sum over $a$ with ${a\preceq b}$ vanishes unless all parts of $b$ have size $1$.
\end{proof}

We are ready to connect these combinatorial results to knot theory.

\begin{thm}
\label{thm:gencoxdecat}
The annular invariant of the generalized Coxeter braid $\sigma_{\epsilon}=\sigma_{n-1}^{\epsilon_{n-1}}\cdots \sigma_{1}^{\epsilon_{1}}$ 
equals 
$$
\Tr_{\Lambda_q}(\sigma_{\epsilon})= \frac{(-1)^{|\epsilon|_+}}{q^{-1}-q}\Psi(a)[X(q^{-1}-q)]$$ where the composition $a$ corresponds to $\epsilon$ as in Definition~\ref{def:epsilona} and $|\epsilon|_+$ is the number of $+1$ entries in $\epsilon$.
\end{thm}
In the following, we use the notation $c_a$ for the annular closure of $\sigma_\epsilon$ and $|a|_+:=|\epsilon|_+$. 

\begin{proof}
Let us prove the statement by induction on the number of entries $-1$ in $\epsilon$.
If $\epsilon=(+1,\cdots,+1)$, this follows from Proposition \ref{positive coxeter decategorified}.  Otherwise, let $a=(a_1,\ldots,a_s)$ be the corresponding composition. The rightmost negative crossing in $\sigma_{\epsilon}$ is at position $a_1+\ldots+a_{s-1}$. If we replace it by a positive one, we get the composition $a'=(a_1,\ldots,a_{s-2},a_{s-1}+a_{s})$. If we erase that crossing, we get a disjoint union of a Coxeter braid for the composition $a''=(a_1,\ldots,a_{s-1})$ and a positive Coxeter braid on $a_s$ strands. Now, by the skein relation the we get
\begin{align*}
(-1)^{|a|_+}\Tr_{\Lambda_q}(c_{a}) &= (q^{-1}-q)(-1)^{|a''|_+}\Tr_{\Lambda_q}(c_{a''})(-1)^{a_s}\Tr_{\Lambda_q}(c_{a_s}) - (-)^{|a'|_+}\Tr_{\Lambda_q}(c_{a'}) 
\\
&=\frac{1}{q^{-1}-q}\Psi(a'')[X(q^{-1}-q)]h_{a_s}[X(q^{-1}-q)]-\frac{1}{q^{-1}-q}\Psi(a')[X(q^{-1}-q)]
\\
&=\frac{1}{q^{-1}-q}\left(\Psi(a'')[X(q^{-1}-q)]h_{a_s}[X(q^{-1}-q)]-\Psi(a')[X(q^{-1}-q)]\right)
\\ &=\frac{1}{q^{-1}-q}\Psi(a)[X(q^{-1}-q)].
\end{align*}
The last equation follows from the recursive formula \eqref{phi recursion}.
\end{proof}

\begin{cor}
\label{cor:hookcoxeter}
For $\epsilon=(\underbrace{+1,\ldots,+1}_{k},\underbrace{-1,\ldots,-1}_{n-k-1})$ we get
$$
\Tr_{\Lambda_q}(\sigma_{\epsilon})=\frac{(-1)^k}{q^{-1}-q}s_{k+1,1^{n-k-1}}[X(q^{-1}-q)]
$$
\end{cor}
\begin{proof}
Follows from Theorem \ref{thm:gencoxdecat} and Corollary \ref{cor:hooks}.
\end{proof}

\begin{cor}
We have 
$$
(-1)^{n-1}\sum_{k=0}^n\Tr_{\Lambda_q}(\sigma_1\cdots \sigma_k \sigma_{k+1}^{-1}\cdots \sigma_{n-1}^{-1})= [n] p_n
$$
\end{cor}
\begin{proof}
This follows from Corollary~\ref{cor:hookcoxeter} and the equations
$$
p_n=\sum_{k}(-1)^{n-1-k} s_{k+1,1^{n-k-1}},\quad p_n[X(q^{-1}-q)]=p_n(q^{-n}-q^{n}).\vspace{-.5cm}
$$
\end{proof}

\begin{rem}
This corollary was proved by Aiston \cite{Aiston} by different methods, see also \cite{Morton}.
\end{rem}

\subsection{From skein to the center of Hecke algebra} 
\label{sec:wrapmorphism}

The skein of the annulus is closely related to the center of the Hecke algebra, as exemplified by Morton \cite{Morton}.
Recall that the Jucys-Murphy braids are defined as $\CL_i=\sigma_{i-1}\cdots \sigma_1\sigma_1\cdots \sigma_{i-1}$.
It is easy to see that $\CL_i\CL_j=\CL_j\CL_i$ for all $1\le i,j\le n$. Note that $\CL_1$ is a trivial braid.
It is well known that the center of $H_n$ is spanned by the symmetric polynomials in $\CL_1,\ldots,\CL_n$. 

There is a natural homomorphism $T_n$ from $\Sk^{+}(A)$ to  $H_n\otimes \kb\cong \Sk(n,n)$ given by wrapping annular links $L$ around the identity braid on $n$ strands as in the following picture:
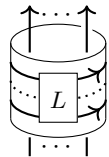
\begin{equation}
\label{eqn:wrap}
\begin{tikzpicture}[scale=.5]
 \draw[thick] (0.75,1) to (0.75,0.5);
 \draw[thick] (-0.75,1) to (-0.75,0.5);
 \draw (0,3.5) ellipse (1.25 and .5);
\draw (0,1.5) [partial ellipse=180:360:1.25 and .5];
\draw (-1.25,3.5) to (-1.25,1.5) (1.25,3.5) to (1.25,1.5) ;
\draw[thick, directed=.9] (0,2) [partial ellipse=180:360:1.25 and .5];
\draw[thick, directed=.9] (0,3) [partial ellipse=180:360:1.25 and .5];
\draw[thick, dotted] (0,2.5)[partial ellipse=180:360:1.25 and .5];
\draw[white, line width=.15cm] (0.75,3.8) to (.75,4.5);
\draw[thick,<-](.75,4.5) to (0.75,3.25);
\draw[thick,<-](-.75,4.5) to (-0.75,3.25);
\draw[fill=white] (-.5,1.4) rectangle (0.5,2.7);
\node at (0,2) {\small $L$};
\node at (0,4.25) {$\cdots$};
\node at (0,.75) {$\cdots$};
\end{tikzpicture}
\end{equation}
It is easy to see that for any annular link $L$ the tangle $T_n(L)$ is central in $\Sk(n,n)$ (and hence in $H_n\otimes \kb$),
and $T_n(L_1\sqcup L_2)=T_n(L_1)T_n(L_2)$.

\begin{thm}\cite[Theorem 3.9]{Morton}
Under the identification $\Sk^{+}(A)\otimes \kb \cong \Lambda_{a,q}$ one has $T_n(f)=\phi_n(f)(\CL_1,\ldots,\CL_n)$,
where $f\in \Lambda_q$ and $\phi:\Lambda_{a,q}\to \Lambda_{a,q}$ is an endomorphism defined by 
\begin{equation}
	\label{power sum in center}
\phi_n(p_k)=-(q^k-q^{-k})a^{-k}p_k+\varepsilon(p_k),
\end{equation}
where $\varepsilon$ is the evaluation homomorphism defined in Proposition \ref{prop:evaluation}.\footnote{To compare with Morton's conventions, note that his crossings are the negatives of ours, $s=q^{-1}$, and $v=a$.}
\end{thm}

It is sometimes helpful to rewrite \eqref{power sum in center} in terms of the eigenvalues of central elements $T_n(f)$.
Recall that $\CL_i$ can be simultaneously diagonalized using Jones-Wenzl-type projectors. For each standard Young tableau $T$ 
there is an element $p_T\in \Sk(n,n)$ such that $\CL_i\cdot p_T=q^{2 c_i(T)} p_T$, where $c_i$ denotes the content of the box labeled by $i$ in $T$, see e.g. \cite[Equation (3.20)]{Ram}. 

\begin{figure}[ht!]
\begin{center}
\begin{tikzpicture}[scale=.4]
\draw(0,0)--(0,4)--(2,4)--(2,0)--(0,0);
\draw(1,0)--(1,4);
\draw(0,3)--(3,3)--(3,0)--(2,0);
\draw(0,2)--(3,2);
\draw(0,1)--(4,1)--(4,0)--(3,0);
\node at (.5,.5) {\tiny $1$};
\node at (.5,1.5) {\tiny $5$};
\node at (.5,2.5) {\tiny $8$};
\node at (.5,3.5) {\tiny $11$};
\node at (1.5,.5) {\tiny $2$};
\node at (1.5,1.5) {\tiny $6$};
\node at (1.5,2.5) {\tiny $9$};
\node at (1.5,3.5) {\tiny $12$};
\node at (2.5,.5) {\tiny $3$};
\node at (2.5,1.5) {\tiny $7$};
\node at (2.5,2.5) {\tiny $10$};
\node at (3.5,.5) {\tiny $4$};
\end{tikzpicture}
\quad,\quad\begin{tikzpicture}[scale=.4]
\draw(0,0)--(0,4)--(2,4)--(2,0)--(0,0);
\draw(1,0)--(1,4);
\draw(0,3)--(3,3)--(3,0)--(2,0);
\draw(0,2)--(3,2);
\draw(0,1)--(4,1)--(4,0)--(3,0);
\node at (.5,.5) {\tiny $0$};
\node at (.5,1.5) {\tiny $1$};
\node at (.5,2.5) {\tiny $2$};
\node at (.5,3.5) {\tiny $3$};
\node at (1.5,.5) {\tiny $-1$};
\node at (1.5,1.5) {\tiny $0$};
\node at (1.5,2.5) {\tiny $1$};
\node at (1.5,3.5) {\tiny $2$};
\node at (2.5,.5) {\tiny $-2$};
\node at (2.5,1.5) {\tiny $-1$};
\node at (2.5,2.5) {\tiny $0$};
\node at (3.5,.5) {\tiny $-3$};
\end{tikzpicture}
\end{center}
\label{fig:ribbon2}
\caption{A standard Young tableau and the content filling for the Young diagram for $\lambda=(4,3,3,2)$.}
\end{figure}

\begin{lem}
\label{lem: q lambda rho}
Assume that $\lambda$ has at most $N$ parts.
Given a symmetric function $f\in \Lambda_q$ and a standard tableau $T$ of shape $\lambda$, one has
$$
T_n(f)\cdot p_T|_{a=q^N}=f(q^{-2\lambda_1+(1-N)},q^{-2\lambda_2+(3-N)}\ldots,q^{-2\lambda_N+(N-1)})\cdot p_T.
$$
\end{lem}
\begin{proof}
Since $p_T$ is an eigenvector for all $\CL_i$, by Theorem \ref{power sum in center} it is an eigenvector for $T_n(f)$ for any $f$, so
$$
T_n(f)\cdot p_T=\mu_T(f)\cdot p_T
$$
for some scalar $\mu_T(f)$. Clearly, the assignment $f\mapsto \mu_T(f)$ is a ring homomorphism, so it is sufficient to 
compute the image of power sums. We have 
$$
(\CL_1^k+\ldots+\CL_n^{k})\cdot p_T=\sum_{i=1}^{n}q^{2kc_i(T)}\cdot p_T,
$$
so the eigenvalue of $(\CL_1^k+\ldots+\CL_n^{k})$ on $p_T$ equals
$$
\sum_{j=1}^{N} q^{2k(j-1)}(1+q^{-2k}+\ldots+q^{-2k(\lambda_j-1)})= \sum_{j=1}^{N} q^{2k(j-1)}\frac{q^{-2k\lambda_j}-1}{q^{-2k}-1}= \sum_{j=1}^{l(\lambda)}  \frac{q^{-2k(\lambda_j+1-j)}-q^{-2k(1-j)}}{q^{-2k}-1}
$$
By applying \eqref{power sum in center} we get
$$
T_n(p_k)|_{a=q^N}=-(q^k-q^{-k})q^{-kN}(\CL_1^k+\ldots+\CL_n^{k})+(q^{k(N-1)}+\ldots+q^{-k(N-1)}),
$$
and
\begin{align*}
\mu_T(p_k)&=-(q^k-q^{-k})q^{-kN}\sum_{j=1}^{N}  \frac{q^{-2k(\lambda_j+1-j)}-q^{-2k(1-j)}}{q^{-2k}-1}+(q^{k(N-1)}+\ldots+q^{-k(N-1)})\\
&=q^{k-kN}\sum_{j=1}^{N} (q^{-2k(\lambda_j+1-j)}-q^{-2k(1-j)}) +(q^{k(N-1)}+\ldots+q^{-k(N-1)})\\
&=
\sum_{j=1}^{N} q^{-2k\lambda_j+k(1-N+2(j-1))}-\sum_{j=1}^{N} q^{k(1-N+2(j-1))}+\sum_{j=1}^{N} q^{k(1-N+2(j-1))}
\\
&=
p_k(q^{-2\lambda_1+(1-N)},q^{-2\lambda_2+(3-N)},\ldots,q^{-2\lambda_N+(N-1)}).
\end{align*}
\end{proof}

\subsection{Generalized Hopf links}
\label{sec:S matrix}

We can use the above results to describe the polynomial invariants of generalized Hopf links. Consider the standard genus one Heegaard decomposition of $S^3$ with two annular links $L_1,L_2$ in the two genus one handlebodies. Their union $H(L_1,L_2)$ is a link in $S^3$ which we call a generalized Hopf link (indeed, the cores of the two solid tori yield a Hopf link). Note that it is naturally framed by framings of $L_1$ and $L_2$. The following is clear from the definition:

\begin{prop}
The $\slnn{N}$ polynomial $\langle H(L_1,L_2)\rangle_N$ depends only on classes of $L_1$ and $L_2$ in  $\Sk^{+}(A)$, and it is bilinear in these classes. 
\end{prop}

To compute this invariant, it is then sufficient to choose a basis in $\Sk^{+}(A)\simeq \Lambda_q$ and to compute the bilinear form in this basis. Lemma \ref{lem: q lambda rho} immediately implies the following:

\begin{prop}
The invariants of the generalized Hopf links are completely determined by either  of the following:
\begin{enumerate}
\item[(a)] If both components are colored by Schur functions then
\begin{equation}
\label{eq: S matrix}
\langle H(s_{\lambda},s_{\mu})\rangle_N=s_{\lambda}(q^{-(N-1)},\ldots,q^{(N-1)})s_{\mu}(q^{-2\lambda_1-(N-1)},\ldots,q^{-2\lambda_N+(N-1)}).
\end{equation}

\item[(b)] If one component is  colored by a Schur function $s_{\lambda}$ and the other by an arbitrary symmetric function $f$ then
\begin{equation}
\langle H(s_{\lambda},f)\rangle_N=s_{\lambda}(q^{-(N-1)},\ldots,q^{(N-1)})f(q^{-2\lambda_1-(N-1)}\ldots,q^{-2\lambda_N+(N-1)}).
\end{equation}
\end{enumerate}
\end{prop} 

\begin{rem}
It follows that the right hand side of \eqref{eq: S matrix} is symmetric in $\lambda$ and $\mu$ for all $N$. 
\end{rem}

\section{General facts about symmetric monoidal categories}
\label{sec: symmetric categories}

\subsection{A free symmetric monoidal category}

We start by defining a useful PROP --- a graded, additive version of a product and permutations category \cite[Chapter V, 2.4]{McL}.

\begin{defi}
Let $\Prop$ denote the graded,  strict symmetric monoidal $\C$-linear additive category that is freely generated by a single object $E$ and a degree two endomorphism $x$. We will use the notation $\Proph=\Kar(\Prop)$ for its idempotent completion. 
\end{defi}

The objects of $\Prop$ are formal direct sums of grading shifts of tensor powers of $E$ and we denote such grading shifts by powers of $q$. The morphisms of $\Prop$ are matrices whose entries can be interpreted as $\C$-linear combinations of string diagrams built from identity endomorphisms of copies of $E$, the morphism $x\colon q^k E\to q^{k-2} E$ and the basic braiding morphism $\sigma\colon q^k E\otimes E \to q^k E \otimes E$. (We think of such string diagrams as \textit{dotted permutations}). Explicitly, we have:

\begin{equation}
\label{eqn:Phom}
\Hom_{\Prop}(q^{k}E^{\otimes m},q^{l}E^{ \otimes n})=\begin{cases}
(\C[x_1,\ldots,x_n]\rtimes \C[S_n])_{k-l} & \text{if}\; m=n,\\
0 & \text{otherwise}.\\
\end{cases}
\end{equation}
Here, the subscript $k-l$ indicates taking the degree $k-l$ component of this algebra, which is graded by putting all $x_i$ in degree two and all permutations in degree zero. In other words, we have:

\begin{lem} $\Proph$ is equivalent to $\bigoplus_{n\geq 0}\C[S_n]\ltimes \C[x_1,\dots,x_n]-\mathrm{gpmod}$.
\end{lem}

In the following $K_0(\CC)$ denotes the split Grothendieck group (ring) of an additive (monoidal) category $\CC$ and $\Lambda_q$ is the $\Z[q^{\pm 1}]$-algebra of symmetric functions in infinitely many variables.
 
\begin{lem}
 We have ring isomorphisms $K_0(\Prop)\cong \Z[q^{\pm 1},e]$ where $[E]\mapsto e$, and $K_0(\Proph)\cong \Lambda_q$ where $[E]\mapsto e_1$. 
\end{lem}

\begin{proof}
By definition, $\Prop$ is additively generated by $q^k E^{\otimes n}$ and there are no isomorphisms between distinct such objects, so 
$K_0(\Prop)\cong \Z[q^{\pm 1},e]$ and $[E^{\otimes k}]=e^k$. 

To compute the Grothendieck group of $\Proph$, we need to classify the idempotent endomorphisms of objects of the form $q^k E^{\otimes n}$ in $\Prop$. Since $x$ has positive degree, \eqref{eqn:Phom} implies that idempotents appear only in $\C[S_n]$ and they are exactly the Young idempotents $\e_{\lambda}$, which are parametrized (up to isomorphism) by Young diagrams $\lambda$ with $|\lambda|=n$. Then $K_0(\Proph)$ has a basis given by the classes of such pairs $(q^k E^{\otimes n},\e_{\lambda})$.  The fact that this gives a ring homomorphism follows from the next section.
\end{proof}

\subsection{Schur functors and evaluation}

Let $\CC$ be a  $\C$-linear strict symmetric monoidal Karoubian category, and let $\CE$ be an object in $\CC$. For every $n\geq 1$ there is an
action of $S_n$ on $\CE^{\otimes n}$ given by the permutation of the factors. In other words, we have a homomorphism $\phi_n:S_n\to \End(\CE^{\otimes n})$. For every partition $\lambda$ of $n$ we pick the primitive Young idempotent $\e_{\lambda}\in \C[S_n]$ corresponding to a fixed Young tableau of shape $\lambda$. Its image $\phi_n(\e_{\lambda})$ is an idempotent endomorphism of $E^{\otimes n}$. Since $\CC$ is Karoubian, we can define the {\it Schur functor} of $\CE$ as the image of this idempotent:
$$
\Sch^{\lambda}(\CE):=\phi_n(\e_{\lambda})\CE^{\otimes n}.
$$ 
For more details on Schur functors see \cite{Deligne}. We will write $\bV^n(\CE)=\Sch^{(1^n)}(\CE)$ and $S^n(\CE)=\Sch^{(n)}(\CE)$.

\begin{defi}
\label{def:rank}
We say that the object $\CE$ has rank at most $N$, if $\bV^{N+1}(\CE)\cong 0$.
\end{defi}

For example, $\C^N$ is of rank at most $N$ in the symmetric monoidal category of complex vector spaces.

\begin{prop}
If $\CE$ is an object of rank at most $N$ and $\lambda$ is a partition with more than $N$ parts then $\Sch^{\lambda}(\CE)\cong 0$.
\end{prop}

\begin{proof}
Follows from \cite[Corollaire 1.7]{Deligne}.
\end{proof}

\begin{prop}
\label{prop: evaluation}
Let $\CC$ be a graded, strict symmetric monoidal $\C$-linear additive category, and let $\CE$ be an object in $\CC$ with an endomorphism $X$. Then
there is a unique braided monoidal $\C$-linear additive functor $\Prop\to \CC$ which sends $E$ to $\CE$ and $x$ to $X$. If, in addition, $\CC$ is Karoubian then this functor extends to a functor $\Proph\to \CC$. 
\end{prop}

\begin{proof}
By the assumptions, there is an action of $\C[X_1,\ldots,X_n]\rtimes \C[S_n]$ on $\CE^{\otimes n}$, so we can define a monoidal functor
$\Prop\to \CC$ sending $E^{\otimes n}$ to $\CE^{\otimes n}$. It uniquely extends to the Karoubi completions.
\end{proof}

\begin{rem}
\label{rem:self-commuting}
More generally, let $\CC$ be a $\C$-linear additive monoidal (but not necessary symmetric) Karoubian category. We shall say that an object $\CE\in \CC$ with an endomorphism $X$ is {\it self-commuting} with symmetry $s:\CE\otimes \CE\to \CE\otimes \CE$ if there is an additive monoidal functor $\Proph\to \CC$ sending $E$ to $\CE$, $\sigma$ to $s$, and $x$ to $X$.
\end{rem}

\subsection{Complexes}

The constructions from the previous subsection directly extend to the category $\Kom(\CC)$ of complexes of objects in $\CC$
and to the homotopy category of complexes $\Komh(\CC)$. We will frequently use the following fact which is well-known to experts (e. g.\cite{Balmer}). For completeness, we prove it in the appendix as Theorem \ref{th:homotopy karoubian}.

\begin{thm}
 The bounded homotopy category of a Karoubian category is Karoubian.
\end{thm}

The category of complexes $\Kom(\CC)$ is symmetric monoidal if the original category $\CC$ was so. To fix the sign conventions,
we define the differential on the tensor product by the equation
\begin{equation}
\label{d tensor product}
d_{A_i\otimes B_j}=d_{A_i}\otimes \id_{B_j}+(-1)^{i}\id_{A_i}\otimes d_{B_j}.
\end{equation}
The braiding $\Sigma$ on $\Kom(\CC)$ differs from the braiding $\sigma$ in $\CC$ by sign placements.
\begin{equation}
\label{braiding complexes}
\Sigma_{A_i, B_j} = (-1)^{i j} \sigma_{A_i,B_j}
\end{equation}
This allows one to define arbitrary Schur functors for complexes.
One can check that Schur functors of homotopy equivalent complexes are homotopy equivalent, see e.g. Theorem~\ref{thm:Schurhomotopy}. We refer to the appendix for more details on Schur functors for complexes.
Also, we record the following fact which immediately follows from the above discussion.

\begin{prop}
Let $\CC$ be a  $\C$-linear additive monoidal (but not necessary symmetric) Karoubian category,
assume $\CE$ is a self-commuting complex in the bounded homotopy category $\Komh(\CC)$.
Then the Schur functors $\Sch^{\lambda}(\CE)$ are well defined.
\end{prop}

The Schur functors interact non-trivially with the shift functor $[1]$, for which we use the convention $A[1]=\id[1]\otimes A$. First, note that for two complexes $A$ and $B$ the isomorphism
$s\colon A[1]\otimes B[1]\to (A\otimes B)[2]$
sends $a\otimes b\mapsto (-1)^{deg(a)-1}a\otimes b$. 
Indeed, in agreement with \eqref{d tensor product}, the isomorphism
$$
A[1]\otimes B[1]=\id[1]\otimes A\otimes \id[1]\otimes B\cong \id[1]\otimes \id[1]\otimes A\otimes B= (A\otimes B)[2]
$$
is given by the braiding $c_{23}$.

Similarly, one can check that the chain of isomorphisms
$$
A\otimes B[2]\cong A[1]\otimes B[1]\cong B[1]\otimes A[1]\cong B\otimes A[2]
$$
differs from the composition of the braiding $A\otimes B\cong B\otimes A$ and the shift $[2]$ by a factor of $-1$.
Therefore the representations of $S_k$ on $(A^{\otimes k})[k]$ and on $(A[1])^{\otimes k}$ differ by sign, and 
\begin{equation}
\label{schur shift}
\Sch^{\lambda}(A[1])=\Sch^{\lambda^t}(A)[|\lambda|].
\end{equation}
This shows that the notion of the Schur functor of a complex is sensitive to the parity of homological degrees of its terms.
\begin{exa}
\label{example: sym of two term complex}
Let $A=[\CE \to \uwave{\CF}]$, where $\CE$ is in homological degree $1$ and $\CF$ is in degree $0$. Then:
$$
\bV^k(A)=[S^k(\CE)\to S^{k-1}(\CE)\otimes \CF\to \cdots \to \CE\otimes \bV^{k-1}(\CF)\to \uwave{\bV^k(\CF)}],
$$
where $S^k(\CE)$ has homological degree $k$. However,  
$$
\bV^k(A[-1])=\bV^k[\uwave{\CE}\to \CF]= [\uwave{\bV^k(\CE)}\to \bV^{k-1}(\CE)\otimes \CF\to \cdots \to \CE\otimes S^{k-1}(\CF)\to S^k(\CF)]
$$
where $S^k(\CF)$ has homological degree $-k$.
\end{exa}

\begin{exa}
\label{example Sym A equals A}
Consider a two-term complex over the category $\CC[t]$ of $\C[t]$-modules
$$
A=[\C[t]\xrightarrow{t^k}\uwave{\C[t]}].
$$
Since $\C[t]\otimes_{\C[t]}\C[t]=\C[t]$, we have $S^2(\C[t])\cong\C[t]$ and $\bV^2(\C[t])\cong 0$.
Similarly, $S^k(\C[t])=\C[t]$ and $\bV^k(\C[t])=0$ for $k\ge 2$.
Therefore
\begin{align*}
S^k(A) =[\bV^k(\C[t])\to\cdots \to \bV^1(\C[t])\otimes_{\C[t]} S^{k-1}(\C[t])\to  \uwave{S^{k}(\C[t])}] \cong [\C[t]\otimes_{\C[t]}\C[t]\to \uwave{\C[t]}]= A.
\end{align*}
\end{exa}

We will need the following result:

\begin{thm}
\label{th: wedges in homotopy karoubi}
Let $\Propo$ be the full tensor subcategory of $\Proph$ generated by $\bV^i(E)$. Then the bounded homotopy categories of $\Propo$ and of $\Proph$ are equivalent.
\end{thm}

\begin{proof}
Since $\Propo$ is a full subcategory  of $\Proph$, the homotopy category of $\Propo$ is a full subcategory of the homotopy category of $\Proph$. Furthermore, $\Komh(\Propo)$ is dense (in the sense of \cite{Thomason}) in $\Komh(\Proph)$ since every complex in $\Komh(\Proph)$ is even isomorphic to a direct summand in a complex in $\Komh(\Prop)$, i.e. a complex built out of several copies of $E^{\otimes n}$. 

Every Schur functor of $E$ is homotopy equivalent to a complex built out of $\bV^i(E)$. Indeed, the Schur functor $\Sch^{\lambda}(E)$ appears as a unique summand in  $\bigotimes_{j}\bV^{\lambda_j}(E)$ and all other summands are smaller than $\lambda$ in dominance order, so we can inductively resolve  $\Sch^{\lambda}(E)$ by the products of $\bV^i(E)$.

This means that $K_0(\Komh(\Propo))\cong K_0(\Komh(\Proph))$ and by Theorem \ref{thm:thomason} we get 
$\Komh(\Propo)\simeq\Komh(\Proph)$.
\end{proof}

\subsection{Affine extensions and plethysms}

Consider a symmetric monoidal Karoubian $\C$-linear category $\CC$. We define its \textit{affine extension} $\CC[t]$ as (the Karoubi completion of) the category with the objects $\CE[t]$ where $\CE$ ranges over objects of $\CC$, and the hom spaces have the form 
$$
\Hom_{\CC[t]}(\CE[t],\CF[t])=\Hom_{\CC}(\CE,\CF)\otimes_{\C} \C[t].
$$
In particular, each object $\CE[t]$ in $\CC[t]$ has  endomorphisms $t^k$ for $k\ge 0$. 
The tensor product on $\CC$ naturally induces a tensor product in $\CC[t]$. We define \textit{pullback} and \textit{pushforward} functors
\begin{align*}
\pi^*\colon & \CC\to \CC[t], \quad \CE\mapsto \CE[t]\\
\pi_*\colon & \CC[t]\to \CC, \quad\CE[t]\mapsto \CE\otimes \C[t]\cong \oplus_{k\ge 0}\CE
\end{align*}

We assume that $\CC$ is graded, and $t$ has some nontrivial grading, so that the direct sum in the definition of $\pi_*(\CE)$ makes sense in an appropriate completion with respect to this grading (we allow infinite direct sums which are finite in each grading). 

Clearly, $\pi^*$ is monoidal, and left adjoint to $\pi_*$. These functors naturally extend to functors between the homotopy categories of complexes of objects in $\CC$ and $\CC[t]$, respectively. 

\begin{exa}
If $R$ is an algebra and $\CC=R-mod$, then $\CC[t]\simeq R[t]-mod$. The functors $\pi_*$ and $\pi^*$ are given by (derived) restriction and induction functors. In particular, if $\CE$ is a free $R$-module then $\CE[t]$ is a free $R[t]$-module, and all free $R[t]$-modules appear this way. Furthermore, the restriction of $\CE[t]$ to $R$ is isomorphic (as an $R$-module) to $\CE\otimes \C[t]$, and 
$$
\Hom_{R[t]}(\CE[t],\CF[t])=\Hom_{R}(\CE,\CF)\otimes \C[t].
$$
\end{exa}

We now use affine extensions to define functors which model certain plethystic transformations. We define a two-term complex over $\Proph[t]$:
$$
K(E,x):=[q \pi^*(E)\xrightarrow{x-t} \uwave{q^{-1}\pi^*(E)}]
$$

Observe that $K(E,x)$ still has an action of $x$ as an endomorphism of a complex. 
By Proposition \ref{prop: evaluation}, we can define an evaluation functor from $\Proph$ to $\Komh(\Proph[t])$
which sends an object $F$ of $\Proph$ to $F(K(E,x))$. 

\begin{defi}
We define the functor $\Phi: \Proph\to \Komh(\Proph)$ as the composite:
\[
\Phi\colon F\mapsto F(K(E,x))\mapsto \pi_*(q F(K(E,x))).
\] 
\end{defi}

\begin{exa}
Recall that we have $K_0(\Proph)\cong \Lambda_q$ and the functor $\Phi$ induces the following map on the level of Grothendieck rings: 
\[
\Phi\colon f\mapsto f[X(q^{-1}-q)]\mapsto \frac{f[X(q^{-1}-q)]}{q^{-1}-q}.
\]
Note that the first map is a ring homomorphism (induced by a monoidal functor), but the second is not.
\end{exa}

The ``plethysm'' functor $\Phi$ can be combined with the evaluation in the following way. Let $\CE$ be an object in a symmetric monoidal Karoubian category $\CC$ with an endomorphism $X$. As above, this data defines a braided monoidal functor 
$\Proph\to \CC$ which sends $E$ to $\CE$ and $x$ to $X$, which can be extended to a functor from $\Komh(\Proph)$ to
$\Komh(\CC)$. By the functoriality of affine extension, we can also construct functors $\Proph[t]\to \CC[t]$ and
$\Komh(\Proph[t])\to \Komh(\CC[t])$. It is easy to see that for any object $F$ of $\Proph$ these send
\begin{align*}
K(E,x)&\mapsto K(\CE,X)=[\pi^*(\CE)\xrightarrow{X-t} \uwave{\pi^*(\CE)}] \text{ in }\Komh(\CC[t]) ,
\\
F(K(E,x))&\mapsto F(K(\CE,X)\in \Komh(\CC[t])\quad\textrm{and}\quad \Phi(F)\mapsto \pi_*(qF(K(\CE,X)))\in \Komh(\CC).
\end{align*}

\subsection{Examples of plethysms}

Let us compute the action of $\Phi$ on some objects and morphisms. 

\begin{exa} We have
$$
 \Phi(E)=\pi_*(q K(E,x))=\left[q^2 E\otimes \C[t]\xrightarrow{x-t} \uwave{E\otimes \C[t]}\right]\simeq E,
$$
where the last homotopy equivalence follows from ``infinite Gaussian elimination''.
\end{exa}

\begin{defi}
Let $U\simeq \C^{n-1}$ denote the reflection representation of $S_n$. Then we define the Koszul complex
\[
\Cube_n= \left[q^{n-1} \bV^{n-1}(U)\otimes E^{\otimes n}\to \ldots \to q^{3-n} U\otimes E^{\otimes n}\to q^{1-n}\uwave{E^{\otimes n}}\right],
\]
where the differential 
\[
d_{\Cube_n}\colon q^{2i+1-n}\bV^{i}(U)\otimes E^{\otimes n}\to q^{2i-1-n} \bV^{i-1}(U)\otimes E^{\otimes n}
\]
is induced by the linear map $U\to \End(E^{\otimes n})$  which sends the i-th basis vector in $U$ to $x_i-x_{i+1}$.
\end{defi}
From the definition it is immediate that $\Cube_n$ admits an action of $S_n$, which restricts to the symmetry-induced action $S_n\to \End(E^{\otimes n})$ in homological degree zero.

\begin{prop}
$\Phi(E^{\otimes n})= \pi_*(q K(E,x)^{\otimes n}) = \pi_*(q [q E\xrightarrow{x-t} q^{-1}\uwave{E}]^{\otimes n}) \simeq \Cube_n$
\end{prop}
\begin{proof}
Note that $[q \uwave{E}\xrightarrow{x-t} q^{-1}E]^{\otimes n}$ is also a Koszul complex, and as such it can be recovered from its last differential, which is the $\C$-linear map 
\[S\colon (q^{2-n} E^{\otimes n})^{\oplus n}\xrightarrow{(x_1-t,\dots,x_{n-1}-t,x_n-t)}q^{-n}E^{\otimes n},\] 
by taking the exterior algebra on $(E^{\otimes n})^{\oplus n}$ and defining the differential as contraction with $S$. We can obtain an isomorphic Koszul complex after a change of basis from:
\[S^\prime \colon (q^{2-n} E^{\otimes n})^{\oplus n}\xrightarrow{(x_1-x_2,\dots,x_{n-1}-x_n,x_n-t)}q^{-n}E^{\otimes n}\]
Considering this as a complex of $\C[x_1,\dots,x_n]$-modules, we can apply Gaussian elimination along the component $-t$ of the differential to obtain $\Cube_n$.
\end{proof}

\begin{cor} 
Let $\Cube_n^\lambda$ denote the chain complex obtained as the image of our chosen Young idempotent $\e_{\lambda}$ of shape $\lambda$ acting on $\Cube_n$. Then we have:
$$
\Phi(\Sch^{\lambda}(E))=\pi_*(q \Sch^{\lambda}(K(E,x))) \cong \Cube_n^\lambda
$$ 
\end{cor}
\begin{proof}
The functor $\pi_*$ commutes with the action of $\C[S_n]$, so
$$
\Phi(\Sch^{\lambda}(E))=\pi_*(q \Sch^{\lambda}(K(E,x)))\cong \pi_*(q\, \e_{\lambda} K(E,x)^{\otimes n})
\cong \e_{\lambda}(\pi_*(q K(E,x)^{\otimes n})\cong\Cube_n^{\lambda}.\vspace{-.5cm}
$$
\end{proof}

We now describe a categorified version of the identity in Lemma~\ref{lem:plethcomplete}.

\begin{lem}\label{lem:invcube}
The $S_n$--invariant part of $\Cube_n$ can be written as
\[
\Cube_n^{S_n}\cong \Phi(S^n(E))\cong \left[ q^{n-1}\bV^n(E)\to \ldots \to q^{3-n} \Sch^{n-1,1}(E)\to q^{1-n} \uwave{S^{n}(E)} \right].
\]
\end{lem}
\begin{proof}
It is well known that the exterior powers of $U$ are irreducible representations of $S_n$ labeled by the hook Young diagrams.
Then we have $(\bV^{i}U\otimes E^{\otimes n})^{S_n}\cong\Hom_{S_n}(\bV^iU,E^{\otimes n})\cong\Sch^{n-i,1^{i}}(E)$.
\end{proof}
Similarly, one can prove the following.
\begin{lem}\label{lem:signcube}
The sign-isotypic component of $\Cube_n$ can be written as
\[
 \Cube_n^{\textrm{sign}}\simeq \Phi(\bV^n(E)) \simeq \left[ q^{n-1}S^n(E)\to \ldots \to q^{3-n} \Sch^{2,1^{n-2}}(E)\to q^{1-n} \uwave{\bV^{n}(E)} \right].
\]
\end{lem}
As we will see in Theorem~\ref{thm:posnegcox}, the complexes shown in the previous lemmas agrees (up to a homological shift) with the annular invariants of the $(n-1)$-fold positively and negatively stabilized unknots respectively.

Next, we consider particular evaluations of these complexes.
\begin{exa}
\label{exa:eval stab unknot}
Let $\CC=\mathrm{Vect}_\C$. Consider an object $\CE=\C[X]/X^k$. Observe that $\CC[t]\simeq\C[t]-mod$, and
$$
K(\CE,X)=[\pi^*(\CE)\xrightarrow{X-t}\uwave{\pi^*(\CE)}]=[\C[X,t]/X^k\xrightarrow{X-t} \uwave{\C[X,t]/X^k}]\simeq _{\CC[t]}[\C[t]\xrightarrow{t^k}\uwave{\C[t]}].
$$
The shown homotopy equivalence holds in the category of complexes of free $\C[t]$-modules. We can write $\C[X,t]/X^k$ as a direct sum of $k$ copies of $\C[t]$ with the action of $X$ shifting them by one.
Then we get the following complex of $\C[t]$-modules:
\begin{center}
\begin{tikzcd}
\C[t] \arrow[r] &[15pt] \C[t]\\[-15pt]
\C[t] \arrow[ur] \arrow[r] & \C[t]\\[-15pt]
\cdots & \,\,\cdots\,\, \\[-15pt]
\C[t] \arrow[ur] \arrow[r] & \uwave{\C[t]}
\end{tikzcd}
\end{center}
Here the horizontal arrows are given by multiplication by $t$ and the diagonal ones correspond to $X$ and hence are multiplications by $(\pm 1)$. Gaussian elimination cancels everything except the top left and bottom right copies of $\C[t]$, which are then connected by $t^k$.

Now by Example \ref{example Sym A equals A} we have 
$$
S^n(K(\CE,X))\simeq K(\CE,X),\ \pi_*(S^n(K(\CE,X)))\simeq \pi_*(K(\CE,X))\simeq \CE.
$$
for all $n\ge 1$.
\end{exa}

\begin{rem}
The same proof applies to $\CE=\C[X]/p(X)$ for an arbitrary polynomial $p(X)$. Indeed, 
$$
S^n(K(\CE,X))\simeq K(\CE,X)\simeq \C[t]\xrightarrow{p(t)}\uwave{\C[t]},
$$
so 
$$
\pi_*(S^n(K(\CE,X)))\simeq \pi_*(K(\CE,X))\simeq \CE.
$$
\end{rem}

Generalizing the previous example, let $\CE$ be a vector space with the action of a nilpotent operator $X$ with Jordan blocks of size $k_1,\ldots,k_n$. Then we can write $\CE=\oplus_{i}\C[X]/X^{k_i}$, and $K(\CE,X)\cong \oplus_{i}[\C[t]\xrightarrow{t^{k_i}}\uwave{\C[t]}]$. 
Therefore
$$
S^n(K(\CE,X))\cong \bigoplus_{\sum n_i=n}\bigotimes S^{n_i}[\C[t]\xrightarrow{t^{k_i}}\uwave{\C[t]}].
$$
The effect of $\pi_*$ on the terms in the sum can be computed using the previous example.

\begin{exa}
Suppose that $\CE$ is a vector space with an endomorphism $X$ which has two Jordan blocks of sizes $k_1$ and $k_2$. Then $S^n(K(\CE,X))$ has $(n+1)$ direct summands: $S^n[\C[t]\xrightarrow{t^{k_1}}\uwave{\C[t]}]\simeq \C[t]/t^{k_1}$,
$S^n[\C[t]\xrightarrow{t^{k_2}}\uwave{\C[t]}]\simeq \C[t]/t^{k_2}$ and $(n-1)$ more summands of the form
$$
S^{n_1}[\C[t]\xrightarrow{t^{k_1}}\uwave{\C[t]}]\otimes S^{n_2}[\C[t]\xrightarrow{t^{k_2}}\uwave{\C[t]}]\simeq 
[\C[t]\xrightarrow{t^{k_1}}\uwave{\C[t]}]\otimes [\C[t]\xrightarrow{t^{k_2}}\uwave{\C[t]}],\ n_1+n_2=n,\;\; n_1,n_2>0.
$$
After applying the forgetful functor $\pi_*$, the latter complexes are isomorphic to their homology which 
have dimension $\min(k_1,k_2)$ both in homological degrees one and zero. Therefore
$$
\pi_*S^n K(\CE,X)\simeq \C[t]/t^{k_1}\oplus \C[t]/t^{k_2}\oplus (n-1)\C[t]/t^{\min(k_1,k_2)} \oplus (n-1)\C[t]/t^{\min(k_1,k_2)}[1].
$$
\end{exa}

 \section{Khovanov--Rozansky theory}
\label{sec:KhR} 
 
\subsection{Webs}
\label{sec:webs}
The Reshetikhin-Turaev invariants of knots, links and tangles are defined as certain intertwiners of representations of quantum groups. In type A, these intertwiners and the relations satisfied by them can be described by a graphical calculus of webs, see \cite{MOY, CKM}. The basic building blocks in the cases of $\slN$ and $\glN$ are the fundamental representations $\bVq^a \C_q^N$ and their identity endomorphisms, as well as two types of natural intertwiners $\bVq^a \C_q^N\otimes \bVq^b \C_q^N \to \bVq^{a+b} \C_q^N$ and $\bVq^{a+b} \C_q^N\to \bVq^a \C_q^N\otimes \bVq^b \C_q^N$ which are called merge and split respectively:
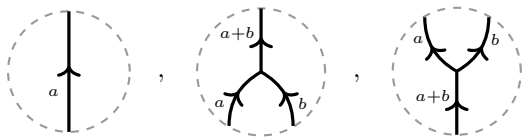
\[
\begin{tikzpicture}[anchorbase,scale=.4]
	\draw [thick, dashed, opacity=0.4] (0,2) circle (2cm);
	\draw [very thick, directed=0.55] (0,0) to (0,4);
	\node at (-.5,1.3) {\tiny $a$};
\end{tikzpicture}
\quad,\quad
\begin{tikzpicture}[anchorbase,scale=.421]
	\draw [thick, dashed, opacity=0.4] (0,2) circle (2cm);
	\draw [very thick, directed=0.55] (0,2) to (0,4);
	\draw [very thick, directed=0.55] (1,.3) to [out=90,in=330] (0,2);
	\draw [very thick, directed=0.55] (-1,.3) to [out=90,in=210] (0,2);
	\node at (-.75,3.2) {\tiny $a{+}b$};
	\node at (-1.3,1) {\tiny $a$};
	\node at (1.3,1) {\tiny $b$}; 
\end{tikzpicture}
\quad,\quad
\begin{tikzpicture}[anchorbase,scale=.421]
	\draw [thick, dashed, opacity=0.4] (0,2) circle (2cm);
	\draw [very thick, directed=0.55] (0,0) to (0,2);
	\draw [very thick, directed=0.55] (0,2) to [out=30,in=270] (1,3.7);
	\draw [very thick, directed=0.55] (0,2) to [out=150,in=270] (-1,3.7); 
	\node at (-.75,1.2) {\tiny $a{+}b$};
	\node at (-1.2,3) {\tiny $a$};
	\node at (1.2,3) {\tiny $b$};
\end{tikzpicture}\]
Other intertwiners can be built by taking tensor products and composites of identities, merges and splits, and such composites quickly become linearly dependent. Analogously, complicated webs can be built by gluing together the shown basic pieces, which then satisfy corresponding linear relations. We illustrate a few relations here and refer to \cite{CKM} for a complete list of web relations for $\slN$ and to \cite{TVW} for the case of $\glnn{N}$. 
\begin{gather}\label{eq:webrel}
\begin{tikzpicture}[anchorbase,scale=.4]
	\draw [thick, dashed, opacity=0.4] (0,2) circle (2cm);
	\draw [very thick, directed=0.55] (0,0) to (0,1);
	\draw [very thick, directed=0.55] (0,3) to (0,4);
	\draw [very thick, directed=0.55] (0,1) to [out=150,in=210] (0,3);
	\draw [very thick, directed=0.55] (0,1) to [out=30,in=330] (0,3);
	\node at (-.5,0.7) {\tiny $a$};
	\node at (-1.3,2.2) {\tiny $a{-}b$};
	\node at (0.9,2.2) {\tiny $b$};
\end{tikzpicture} 
= {a \brack b}~
\begin{tikzpicture}[anchorbase,scale=.4]
	\draw [thick, dashed, opacity=0.4] (0,2) circle (2cm);
	\draw [very thick, directed=0.55] (0,0) to (0,4);
	\node at (-.5,0.7) {\tiny $a$};
\end{tikzpicture}
\quad,\quad
\begin{tikzpicture}[anchorbase,scale=.4]
	\draw [thick, dashed, opacity=0.4] (0,2) circle (2cm);
	\draw [very thick, directed=0.55] (0,0) to (0,1);
	\draw [very thick, directed=0.55] (0,3) to (0,4);
	\draw [very thick, directed=0.55] (0,1) to [out=150,in=210] (0,3);
	\draw [very thick, rdirected=0.5] (0,1) to [out=30,in=330] (0,3);
	\node at (-.5,0.7) {\tiny $a$};
	\node at (-1.3,2.2) {\tiny $a{+}b$};
	\node at (0.9,2.2) {\tiny $b$};
	\end{tikzpicture}
	= {N-a \brack b}~
	\begin{tikzpicture}[anchorbase,scale=.4]
	\draw [thick, dashed, opacity=0.4] (0,2) circle (2cm);
	\draw [very thick, directed=0.55] (0,0) to (0,4);
	\node at (-.5,0.7) {\tiny $a$};
\end{tikzpicture} \\
\nonumber
\begin{tikzpicture}[anchorbase,scale=.4]
	\draw [thick, dashed, opacity=0.4] (0,2) circle (2cm);
	\draw [very thick, directed=0.55] (0,0) to (0,1);
	\draw [very thick, directed=0.55] (0,1) to [out=150,in=270] (-.75,2);
	\draw [very thick, directed=0.55] (-.75,2) to [out=30,in=270] (0,4);
	\draw [very thick, directed=0.55] (-.75,2) to [out=150,in=270] (-1.3,3.5);
	\draw [very thick, directed=0.55] (0,1) to [out=30,in=270] (1.3,3.5);
	\node at (-1.4,2.1) {\tiny $a$};
	\node at (0,2.1) {\tiny $b$};
	\node at (1.4,2.1) {\tiny $c$};
\end{tikzpicture}
=
\begin{tikzpicture}[anchorbase,scale=.4]
	\draw [thick, dashed, opacity=0.4] (0,2) circle (2cm);
	\draw [very thick, directed=0.55] (0,0) to (0,1);
	\draw [very thick, directed=0.55] (0,1) to [out=30,in=270] (.75,2);
	\draw [very thick, directed=0.55] (.75,2) to [out=150,in=270] (0,4);
	\draw [very thick, directed=0.55] (.75,2) to [out=30,in=270] (1.3,3.5);
	\draw [very thick, directed=0.55] (0,1) to [out=150,in=270] (-1.3,3.5);
	\node at (-1.4,2.1) {\tiny $a$};
	\node at (0,2.1) {\tiny $b$};
	\node at (1.4,2.1) {\tiny $c$};
\end{tikzpicture}
\quad,\quad
\begin{tikzpicture}[anchorbase,scale=.4]
	\draw [thick, dashed, opacity=0.4] (0,0) circle (2cm);
	\draw [very thick, directed=0.55] (1,.5) to [out=120,in=300] (-1,1);
	\draw [very thick, directed=0.55] (-1,-1) to [out=60,in=240] (1,-.5);
	\draw [very thick, directed=0.55] (1,-1.7) to (1,1.7);
	\draw [very thick, directed=0.55] (-1,-1.7) to (-1,1.7);
	\node at (-.7,-1.2) {\tiny $k$};
	\node at (0,1.2) {\tiny $r$};
	\node at (0,-.3) {\tiny $s$};
	\node at (.7,-1.2) {\tiny $l$};
\end{tikzpicture}
= \sum_t {k-l+r-s\brack t}
\begin{tikzpicture}[anchorbase,scale=.4]
	\draw [thick, dashed, opacity=0.4] (0,0) circle (2cm);
	\draw [very thick, directed=0.55] (1,-1) to [out=120,in=300] (-1,-.5);
	\draw [very thick, directed=0.55] (-1,.5) to [out=60,in=240] (1,1);
	\draw [very thick, directed=0.55] (1,-1.7) to (1,1.7);
	\draw [very thick, directed=0.55] (-1,-1.7) to (-1,1.7);
	\node at (-.7,-1.2) {\tiny $k$};
	\node at (0,1.2) {\tiny $s-t$};
	\node at (0,-.3) {\tiny $r-t$};
	\node at (.7,-1.2) {\tiny $l$};
\end{tikzpicture}
\end{gather}

\begin{defi} Let $N\Web$ denote the additive, $\C(q)$-linear category with:
\begin{itemize}
\item Objects: finite sequences $\underline{a}:=(a_1,\dots, a_m)$ with $a_i\in \{1,\dots, N\}$.
\item Morphisms: $\Hom_{\Web}(\underline{a},\underline{b})$ is the $\C(q)$-module of webs properly embedded in the horizontal strip $\R\times [0,1]$, with upward pointing boundary points with labels $\underline{a}$ in $\R\times \{0\}$  and $\underline{b}$ in $\R\times \{1\}$, considered modulo planar isotopy and the $\glN$ web relations from \cite{TVW}. 
\item Composition: the $\C(q)$-bilinear extension of stacking webs.
\end{itemize}
\end{defi}

\begin{thm} $N\Web$ is equivalent to the full subcategory of representations of $\Uq(\glN)$ whose objects are the tensor products of exterior power representations $\bVq^a \C_q^N$ for $0\leq a \leq N$. The equivalence sends the object $\underline{a}:=(a_1,\dots, a_m)$ to $\bVq^{a_1} \C_q^N\otimes \cdots \otimes \bVq^{a_m} \C_q^N$.
\end{thm}
\begin{proof} This is a $\glN$ variant of the main result of \cite{CKM}, see also \cite{QS,TVW}. 
\end{proof}

Now let $S$ be an oriented surface of finite type, possibly with marked points on the boundary with a labeling and a choice of inward or outward orientation. We denote by $N\Web(S)$ the $\Z[q^{\pm 1}]$-module spanned by properly embedded webs in $S$, with boundary matching the data on the marked points, modulo isotopy rel boundary and web relations supported in discs $D^2\subset S$. 

$N\Web(S)$ is a version of the $\glN$ skein module of the surface $S$. Oriented, framed links embedded in $S\times [0,1]$ can be evaluated in $N\Web(S)$ by projecting to $S$ (enforcing the blackboard framing) and resolving all crossings into alternating sums of webs according to the following rule. 
\begin{gather}\label{eq:crossing}
\begin{tikzpicture}[anchorbase,scale=.4]
	\draw [very thick, directed=0.85] (1,-1.7) to [out=90,in=270](-1,1.7);
	\draw [white, line width=.15cm] (-1,-1.7) to [out=90,in=270](1,1.7);
	\draw [very thick, directed=0.85] (-1,-1.7) to [out=90,in=270](1,1.7);
	\node at (-.6,-1.2) {\tiny $k$};
	\node at (.6,-1.2) {\tiny $l$};
	\draw [thick, dashed, opacity=0.4] (0,0) circle (2cm);
\end{tikzpicture}
=  \!\!\!\sum_{s-r=k-l} \!\!
(-q)^{s-k}
 \; \begin{tikzpicture}[anchorbase,scale=.4]
	\draw [thick, dashed, opacity=0.4] (0,0) circle (2cm);
	\draw [very thick, directed=0.55] (1,.5) to [out=120,in=300] (-1,1);
	\draw [very thick, directed=0.55] (-1,-1) to [out=60,in=240] (1,-.5);
	\draw [very thick, directed=0.55] (1,-1.7) to (1,1.7);
	\draw [very thick, directed=0.55] (-1,-1.7) to (-1,1.7);
	\node at (-.7,-1.2) {\tiny $k$};
	\node at (0,1.2) {\tiny $r$};
	\node at (0,-.3) {\tiny $s$};
	\node at (.7,-1.2) {\tiny $l$};
\end{tikzpicture}
\end{gather} Negative crossings are resolved using an analogous formula with $q$ inverted.

The class in $N\Web(S)$ represented by an embedded link is invariant under regular isotopy in $S\times [0,1]$. Framing changes and fork twists act by powers of $q$, but all fork slides hold on the nose:
\begin{gather}
\nonumber
\begin{tikzpicture}[anchorbase,scale=.4]
\draw [very thick, directed=0.75] 	(0.5,-.5) to [out=180,in=270] (0,1) to (0,2);
	\draw [white, line width=.15cm] (0,-1) to [out=90,in=180](0.5,0.5) ;
	\draw [very thick] (0,-2) to (0,-1) to [out=90,in=180](0.5,0.5) to [out=0,in=90] (1,0) to [out=270,in=0] (0.5,-.5);
\node at (-.5,-1.3) {\tiny $a$};
	\draw [thick, dashed, opacity=0.4] (0,0) circle (2cm);
\end{tikzpicture}
= (-1)^a q^{-a(N-a+1)}
\begin{tikzpicture}[anchorbase,scale=.4]
	\draw [very thick,directed=.55] (0,-2) to (0,2) ;=
\node at (-.5,-1.3) {\tiny $a$};
	\draw [thick, dashed, opacity=0.4] (0,0) circle (2cm);
\end{tikzpicture}
\\
\label{eq:forkslide}\begin{tikzpicture}[anchorbase,scale=.4]
	\draw [very thick, directed=0.85] (1,-1.7) to [out=90,in=270](-1,0) to [out=90,in=240] (0,1);
	\draw [white, line width=.15cm] (-1,-1.7) to [out=90,in=270](1,0);
	\draw [very thick, directed=0.85] (-1,-1.7) to [out=90,in=270](1,0) to [out=90,in=300] (0,1);
	\draw [very thick, directed=0.55](0,1)to (0,2);
	\node at (-.6,-1.5) {\tiny $k$};
	\node at (.6,-1.5) {\tiny $l$};
	\draw [thick, dashed, opacity=0.4] (0,0) circle (2cm);
\end{tikzpicture}
= q^{k l}\;
\begin{tikzpicture}[anchorbase,scale=.4]
	\draw [very thick, directed=0.55] (1,-1.7) to [out=90,in=300] (0,.5);
	\draw [very thick, directed=0.55] (-1,-1.7) to [out=90,in=240] (0,.5);
	\draw [very thick, directed=0.55](0,.5)to (0,2);
	\node at (-.6,-1.5) {\tiny $k$};
	\node at (.6,-1.5) {\tiny $l$};
	\draw [thick, dashed, opacity=0.4] (0,0) circle (2cm);
\end{tikzpicture}
\quad, \quad
\begin{tikzpicture}[anchorbase,scale=.4]
	\draw [very thick, directed=0.55] (1,-1.7) to [out=90,in=300] (0,1);
	\draw [very thick, directed=0.75] (-1,-1.7) to [out=90,in=240] (0,1);
	\draw [very thick, directed=0.55](0,1)to (0,2);
	\draw [white, line width=.15cm] (-1.7,-1) to (1.7,1);
	\draw [very thick, directed=0.85] (-1.7,-1) to (1.7,1);
	\node at (-.6,-1.5) {\tiny $k$};
	\node at (.6,-1.5) {\tiny $l$};
	\node at (-1.5,-0.6) {\tiny $m$};
	\draw [thick, dashed, opacity=0.4] (0,0) circle (2cm);
\end{tikzpicture}
= 
\begin{tikzpicture}[anchorbase,scale=.4]
	\draw [very thick, directed=0.75] (1,-1.7) to [out=90,in=300] (0,-.5);
	\draw [very thick, directed=0.75] (-1,-1.7) to [out=90,in=240] (0,-.5);
	\draw [very thick, directed=0.81](0,-.5)to (0,2);
	\draw [white, line width=.15cm] (-1.7,-1) to (1.7,1);
	\draw [very thick, directed=0.85] (-1.7,-1) to (1.7,1);
	\node at (-.6,-1.5) {\tiny $k$};
	\node at (.6,-1.5) {\tiny $l$};
	\node at (-1.5,-0.6) {\tiny $m$};
	\draw [thick, dashed, opacity=0.4] (0,0) circle (2cm);
\end{tikzpicture}
\end{gather}

\subsection{Foams} We still let $S$ denote an oriented surface of finite type. 
The $\Z[q^{\pm 1}]$-module $N\Web(S)$ admits a graded, additive, $\C$-linear categorification $\eqSfoam$ that is closely related to the canopolis $N\foam$ of $\glN$ foams defined in \cite{ETW} using the closed foam evaluation formula of Robert--Wagner \cite{RoW}. Here we only describe the essential features of $\eqSfoam$ and comment on the necessary variations relative to $N\foam$. 
\begin{defi} $\eqSfoam$ is the additive closure of the graded, additive, $\C$-linear category determined by the following data:
\begin{itemize}
\item The objects are formal $q$-grading shifts of webs $q^k W$ embedded in $S$, without allowing any isotopies.

\item The morphisms are $\C$-linear combinations of degree-preserving foams $F\colon q^l V\to q^k W$ embedded in $S\times [0,1]$, modulo isotopy relative to the boundary and modulo additional local relations supported in embedded 3-balls $B^3\subset S\times [0,1]$. 
\item The composition is given by the bilinear extension of the natural gluing of foams.
\end{itemize}
The foams making up the morphisms are decorated 2-dimensional CW-complexes, which are carefully defined in \cite[Definition 2.7]{ETW}. They are graded and their facets are labeled and may carry decorations by symmetric polynomials as explained in and just before \cite[Definition 2.11]{ETW}\footnote{Note, however, that we use the opposite convention to denote grading shifts. E.g. a foam $F$ of degree 2 maps from a shifted web $q^k W$ to another shifted web $q^{k-2} V$.}. The first elementary symmetric polynomial on a $1$-labeled facet is called a \emph{dot}. The local foam relations in embedded 3-balls $B^3\subset S\times [0,1]$ are precisely the relations that hold in the canopolis $\foam$ as defined in \cite[Definition 2.14]{ETW}.
\end{defi}

A direct consequence of the local foam relations in $\eqSfoam$ is that  we have explicit isomorphisms between webs, which induce the web relations~\eqref{eq:webrel} after passing to the Grothendieck group.

\begin{rem} The use of foams in the categorification of link and tangle invariants has a long history, starting with Bar-Natan's use of linearized cobordism categories in his description of Khovanov homology \cite{BN1}. Khovanov's categorification of the $\slnn{3}$ link polynomial \cite{Kho3} is the first one that uses foams with singularities. The matrix factorization categories underlying Khovanov--Rozansky $\glnn{N}$ link homologies were given a topological interpretation via foams in \cite{KR3}, which was used in a new construction of $\glnn{N}$ link homologies by Mackaay--Sto{\v{s}}i{\'c}--Vaz \cite{MSV}. Blanchet demonstrated that $\glnn{2}$ foams support a version of Khovanov homology that is functorial under link cobordisms \cite{Bla}. Better control over $\glnn{N}$ foam categories was gained by Lauda--Queffelec--Rose through their connections to categorified quantum groups \cite{LQR,QR}. Finally, Robert--Wagner \cite{RoW} found a mathematically rigorous and entirely combinatorial construction of $\glnn{N}$ foams, which is the basis for the foam categories used here and in the proof of the functoriality of Khovanov--Rozansky homologies under cobordisms in \cite{ETW}.
\end{rem}

\subsection{Categorical invariants for links in a thickened surface}
It is now a routine task to define a categorical invariant of links (or tangles) in $S\times [0,1]$ (with boundary in $\partial(S)\times\{1/2\}$) that takes values in $\Komh(\eqSfoam)$, the homotopy category of chain complexes over $\eqSfoam$. Indeed, for a generic tangle embedding, the natural projection $S\times [0,1] \to S$ gives a tangle diagram.  The alternating sum in the crossing formula~\eqref{eq:crossing} gets lifted to a chain complex and if several crossings occur, the alternating multi-sums become tensor product chain complexes. In fact, there are two natural conventions for the chain complexes that can be associated to a positive\footnote{Sometimes the shown complexes are associated to negative crossings, but we strongly prefer the convention here. For this convention, the triply-graded homology of positive torus knots satisfies a parity condition.} uncolored crossing:

\begin{align}
 \label{eq:crossingcxuf}
 \Hgenuf{\begin{tikzpicture}[anchorbase,scale=.25]
	\draw [very thick, ->] (2,1)to(1.7,1) to [out=180,in=0] (-1.7,-1) to (-2,-1);
	\draw [white, line width=.15cm] (2,-1) to (1.7,-1) to [out=180,in=0]  (-1.7,1) to (-2,1);	
	\draw [very thick, ->] (2,-1) to(1.7,-1) to [out=180,in=0]  (-1.7,1) to (-2,1);	
\end{tikzpicture}
}
 \;\;\;\; &= \;\;
q\;
\begin{tikzpicture}[anchorbase,scale=.25]
	\draw [very thick, ->] (2,1)to(1.7,1) to [out=180,in=45] (.5,0) to (-.5,0) to [out=135,in=0] (-1.7,1) to (-2,1);
	\draw[very thick] (.5,0) to (-.5,0);
	\draw [very thick, ->] (2,-1) to(1.7,-1) to [out=180,in=315] (.5,0) to (-.5,0) to [out=225,in=0] (-1.7,-1) to (-2,-1);	
	\node at (0, -.5) {\tiny $2$};
\end{tikzpicture}
\to 
\uwave{
\begin{tikzpicture}[anchorbase,scale=.25]
	\draw [very thick, <-] (-2,1) to (2,1);
	\draw [very thick, <-] (-2,-1) to (2,-1);
 \end{tikzpicture}
 }
 \\
 \label{eq:crossingcx}
\Hgen{\begin{tikzpicture}[anchorbase,scale=.25]
	\draw [very thick, ->] (2,1)to(1.7,1) to [out=180,in=0] (-1.7,-1) to (-2,-1);
	\draw [white, line width=.15cm] (2,-1) to (1.7,-1) to [out=180,in=0]  (-1.7,1) to (-2,1);	
	\draw [very thick, ->] (2,-1) to(1.7,-1) to [out=180,in=0]  (-1.7,1) to (-2,1);	
\end{tikzpicture}
}
 \;\;&= \;\; 
\uwave{\begin{tikzpicture}[anchorbase,scale=.25]
	\draw [very thick, ->] (2,1)to(1.7,1) to [out=180,in=45] (.5,0) to (-.5,0) to [out=135,in=0] (-1.7,1) to (-2,1);
	\draw[very thick] (.5,0) to (-.5,0);
	\draw [very thick, ->] (2,-1) to(1.7,-1) to [out=180,in=315] (.5,0) to (-.5,0) to [out=225,in=0] (-1.7,-1) to (-2,-1);	
	\node at (0, -.5) {\tiny $2$};
\end{tikzpicture}
}
\to q^{-1}\;
\begin{tikzpicture}[anchorbase,scale=.25]
	\draw [very thick, <-] (-2,1) to (2,1);
	\draw [very thick, <-] (-2,-1) to (2,-1);
 \end{tikzpicture}
\end{align}
In both cases the differential is given by an \textit{unzip foam}. For more details about these \textit{Khovanov--Rozansky constructions} using foams, see e.g. \cite[Section 3.1]{ETW} and \cite[Section 4]{QW3}.

\begin{defi} Let $T_D$ be a tangle diagram in $S$, then we denote the chain complexes constructed based on the local models \eqref{eq:crossingcxuf} and \eqref{eq:crossingcx} (and their colored versions) by $\Hgenuf{T_D}$ and $\Hgen{T_D}$ respectively. We shall consider these complexes as objects in $\Komh(\eqSfoam)$.
\end{defi}

The chain complex $\Hgenuf{T_D}$ is invariant under all Reidemeister moves up to chain homotopy equivalence, see e.g. \cite[Theorem 3.5]{ETW}. The chain complex $\Hgen{T_D}$ is invariant under framed Reidemeister moves up to chain homotopy equivalence. While we favour the framed version $\Hgen{T_D}$ in this paper, we also introduce $\Hgenuf{T_D}$ since it is known to admit a functorial assignment of chain maps to tangle cobordisms as we describe next.

\begin{defi}
We denote by $\Stang$ the category with objects given by tangles that are properly embedded in $S\times [0,1]$ and with morphisms given by isotopy classes of tangle cobordisms embedded in $S\times [0,1]^2$. For surfaces without specified boundary points, we also denote $\Stang$ by $\Slink$.
\end{defi}

\begin{thm}[{\cite[Theorem 4.5]{ETW}}, {\cite[Theorem 4.4]{QW3}}]
\label{thm: ETW functoriality}
The Khovanov--Rozansky construction extends to a functor $\Stang\to \Komh( \eqSfoam)$, under which the  image of a tangle $T$ with diagram $T_D$ is given by $\Hgenuf{T_D}$.
\end{thm} 

Since $\Hgen{T_D}$ differs from $\Hgenuf{T_D}$ only in grading shifts in tensor factors, this implies that $\Hgen{-}$ can also be equipped with functorial cobordism maps. However, we currently do not know whether there is a unique (or at least a distinguished) way of lifting Theorem~\ref{thm: ETW functoriality} to the framed setting. Another open question is the following.

\begin{conj}[{\cite[Conjecture 4.8]{QW3}}]
\label{conj: naturality for foams}
The Khovanov--Rozansky functor extends to a functor from the category of tangled webs and framed foams in four-dimensional space to $\Komh( \eqSfoam)$.
\end{conj}

\subsection{Annular links, webs, and foams}

In this section we consider the case $S=A:=S^1\times [0,1]$ without marked points on the boundary, and fix an orientation of the core circle of $A$. 

We define a monoidal structure on $\Alink$ as follows. Given two annular links $L_1$ and $L_2$ in $A\times [0,1]=S^1 \times [0,1] \times [0,1]$, we relabel the copy in which $L_2$ lives as $S^1 \times [1,2]\times [0,1]$. The tensor product $L_1\boxtimes L_2$ is defined by taking the disjoint union $L_1\sqcup L_2$ in $S^1 \times [0,2]\times [0,1]$ and shrinking the second coordinate back to $S^1 \times [0,1]\times [0,1]$. The definition of $\boxtimes$ on morphisms is analogous. It is a simple exercise to check that this defines a monoidal structure with unit given by the empty link and with unitors and associators given by isotopies. In fact, the existence of ``vertical'' and ``horizontal'' isotopies give rise to a (non-symmetric) braiding on $\Alink$.

We say an annular link is consistently oriented if it is given as the closure of a braid with orientation matching the orientation of the core circle. We then denote by $\Alinkp$ the subcategory of $\Alink$ given by consistently oriented links and cobordisms whose time-slices are consistently oriented.

It is clear that two braids closures are isotopic in the annulus (and the corresponding objects in $\Alinkp$ are isomorphic)
if and only if the braids are conjugate.  

For consistently oriented annular links, there exists a universal categorified link invariant from which all annular and planar Khovanov--Rozansky homologies can be recovered. In order to  describe its target category, we say a web $W$ in $A$ is consistently oriented if the tangent vectors project positively to the core circle.

The subcategory $\NAfoamp$ of $\NAfoam$ is cut out by requiring webs to be consistently oriented and foams to have generic cross-sections that are isotopic to such consistently oriented webs.

We denote by $\iAfoam$ the category obtained from $\NAfoamp$ by stabilizing $N\to \infty$. In other words, $\iAfoam$ is the category of consistently oriented annular webs and foams, without restriction on the labeling set and with a free action of the dot on $1$-labeled facets. The component of $\iAfoam$ of winding degree $n$ can be identified with the horizontal trace (see Section \ref{sec: traces}) of the category of singular Soergel bimodules of type $A_{n-1}$.

\begin{thm}[\cite{QR2}] The annular Khovanov--Rozansky homologies factor through the functor $$\Hgen{-}\colon \Alinkp\to \Komh( \iAfoam).$$
\end{thm}

Furthermore, the annular disjoint union yields a natural monoidal structure on $\iAfoam$ and its homotopy category, which is respected by the Khovanov--Rozansky functor. 

\begin{prop}
The annular Khovanov--Rozansky functor $\Alinkp\to \Komh( \iAfoam)$ is monoidal.
\end{prop}

\subsection{Reduction to essential circles}

\begin{defi}
We define $\iAfoamcc$ to be the full subcategory of $\iAfoam$ whose objects are direct sums of grading shifts of webs that are collections of essential concentric circles in the annulus.   
\end{defi} 

The notation $\iAfoamcc$ is to suggest that the objects in this category are $S^1$-equivariant, i.e. that they are invariant under rotation along the core of the annulus. In fact, the same is true for morphisms.

\begin{thm}[{\cite[Theorem 3.2]{QRS}}] The morphism spaces in $\iAfoamcc$ are generated by $S^1$-equivariant, decorated foams. In particular, they are non-negatively graded.
\end{thm}

Queffelec--Rose conjecture that the inclusion $\iAfoamcc\hookrightarrow \iAfoam$ is an equivalence of categories \cite[Conjecture 5.4]{QR2}. They prove a slightly weaker result.

\begin{prop}[{\cite[Proposition 5.1]{QR2}}]\label{prop:rotationequiv} The inclusion $\iAfoamcc\hookrightarrow \iAfoam$ induces an equivalence of categories $\Komh( \iAfoam) \simeq \Komh (\iAfoamcc)$. 
\end{prop}
The main step in the proof of this result is that each annular web, considered as a complex concentrated in homological degree zero, is isomorphic in $\Komh( \iAfoam)$ to a chain complex built out of concentric circle webs. In fact, this is true more generally, see Proposition~\ref{prop:filtration}. For now, we take note of the implication that the categorical invariants of braid closures can be assumed to take values in $\Komh( \iAfoamcc)$. In the next session we will obtain an alternative description of this category.

\subsection{Decorated webs}

We can now take quotients of the webs and foams in $\iAfoamcc$ by their free
$S^1$-symmetry. Under this dimensional reduction, collections of labeled
concentric circles are mapped to finite sequences of labeled points on a line
$\R$. Rotationally symmetric foams are mapped to isotopy classes of webs in the
strip $\R\times [0,1]$, whose edges inherit the decorations by symmetric
functions of the foam facets. 

\begin{defi} Let $\DecWeb$ denote the non-negatively graded, additive, $\C$-linear category of decorated webs in $\R\times [0,1]$ that is isomorphic to $\iAfoamcc$ via the functor 
\[\DecWeb \xrightarrow{-\times S^1} \iAfoamcc\] 
that takes boundary sequences to collections of concentric circles and decorated webs to decorated rotationally symmetric foams. \end{defi}

\begin{lem}\label{lem:webrep} The degree zero part of $\DecWeb$ satisfies the first, third and fourth web relation from \eqref{eq:webrel} and isotopies relative to the boundary which preserve the upward-directedness of webs.
\end{lem}
\begin{proof} See \cite[Section 4.5]{QRS}.
\end{proof}

\begin{lem} 
The following relations hold in $\DecWeb$.
\begin{gather}\label{eq:decwebrel}
\begin{tikzpicture}[anchorbase,scale=.4]
	\draw [thick, dashed, opacity=0.4] (0,2) circle (2cm);
	\draw [very thick] (0,2) to (0,4);
	\draw [very thick] (1,0.3) to [out=90,in=330] (0,2);
	\draw [very thick] (-1,.3) to [out=90,in=210] (0,2);
	\draw[fill=white] (-.5,2.4) rectangle (0.5,3);
	\node at (0,2.7) {\tiny $e_{r}$};
		\node at (0,4.3) {\tiny $a{+}b$};	
	\node at (-1.1,-.2) {\tiny $a$};
	\node at (1.1,-.2) {\tiny $b$}; 
\end{tikzpicture}
= \sum_{s+t=r}
\begin{tikzpicture}[anchorbase,scale=.4]
	\draw [thick, dashed, opacity=0.4] (0,2) circle (2cm);
	\draw [very thick] (0,2) to (0,4);
	\draw [very thick] (1,0.3) to [out=90,in=330] (0,2);
	\draw [very thick] (-1,.3) to [out=90,in=210] (0,2);
	\draw[fill=white] (.2,1.6) rectangle (1.2,1);
	\node at (0.7,1.3) {\tiny $e_{t}$};	
	\draw[fill=white] (-.2,1) rectangle (-1.2,1.6);
	\node at (-0.7,1.3) {\tiny $e_{s}$};	
	\node at (0,4.3) {\tiny $a{+}b$};
	\node at (-1.1,-.2) {\tiny $a$};
	\node at (1.1,-.2) {\tiny $b$}; 
\end{tikzpicture}
\\
\label{eq:decwebrel2}
\begin{tikzpicture}[anchorbase,scale=.4]
	\draw [thick, dashed, opacity=0.4] (0,2) circle (2cm);
	\draw [very thick] (-1,0.3) to [out=90,in=210] (0,1.5);
	\draw [very thick] (1,0.3) to [out=90,in=330] (0,1.5);
	\draw [very thick] (0,2.5) to [out=150,in=270] (-1,3.7);
	\draw [very thick] (0,2.5) to [out=30,in=270] (1,3.7);
	\draw [very thick, directed=.55] (0,1.5)to  (0,2.5);
	\node at (-0.7,3) {$\bullet$};	
	\node at (-.5,2) {\tiny $2$};
	\node at (0.9,-.2) {\tiny $1$};
	\node at (-0.9,-.2) {\tiny $1$};
\end{tikzpicture} 
- 
\begin{tikzpicture}[anchorbase,scale=.4]
	\draw [thick, dashed, opacity=0.4] (0,2) circle (2cm);
	\draw [very thick] (-1,0.3) to [out=90,in=210] (0,1.5);
	\draw [very thick] (1,0.3) to [out=90,in=330] (0,1.5);
	\draw [very thick] (0,2.5) to [out=150,in=270] (-1,3.7);
	\draw [very thick] (0,2.5) to [out=30,in=270] (1,3.7);
	\draw [very thick, directed=.55] (0,1.5)to  (0,2.5);
	\node at (0.7,1) {$\bullet$};	
	\node at (-.5,2) {\tiny $2$};
	\node at (0.9,-.2) {\tiny $1$};
	\node at (-0.9,-.2) {\tiny $1$};
\end{tikzpicture} 
=
\begin{tikzpicture}[anchorbase,scale=.4]
	\draw [thick, dashed, opacity=0.4] (0,2) circle (2cm);
	\draw [very thick, directed=.75] (-1,0.3) to (-1,3.7);
	\draw [very thick, directed=.75] (1,0.3) to (1,3.7);
	\node at (-1,2) {$\bullet$};	
	\node at (0.9,-.2) {\tiny $1$};
	\node at (-0.9,-.2) {\tiny $1$};
\end{tikzpicture} 
-
\begin{tikzpicture}[anchorbase,scale=.4]
	\draw [thick, dashed, opacity=0.4] (0,2) circle (2cm);
	\draw [very thick, directed=.75] (-1,0.3) to (-1,3.7);
	\draw [very thick, directed=.75] (1,0.3) to (1,3.7);
	\node at (1,2) {$\bullet$};	
	\node at (0.9,-.2) {\tiny $1$};
	\node at (-0.9,-.2) {\tiny $1$};
\end{tikzpicture} 
= 
\begin{tikzpicture}[anchorbase,scale=.4]
	\draw [thick, dashed, opacity=0.4] (0,2) circle (2cm);
	\draw [very thick] (-1,0.3) to [out=90,in=210] (0,1.5);
	\draw [very thick] (1,0.3) to [out=90,in=330] (0,1.5);
	\draw [very thick] (0,2.5) to [out=150,in=270] (-1,3.7);
	\draw [very thick] (0,2.5) to [out=30,in=270] (1,3.7);
	\draw [very thick, directed=.55] (0,1.5)to  (0,2.5);
	\node at (-0.7,1) {$\bullet$};	
	\node at (-.5,2) {\tiny $2$};
	\node at (0.9,-.2) {\tiny $1$};
	\node at (-0.9,-.2) {\tiny $1$};
\end{tikzpicture} 
-
\begin{tikzpicture}[anchorbase,scale=.4]
	\draw [thick, dashed, opacity=0.4] (0,2) circle (2cm);
	\draw [very thick] (-1,0.3) to [out=90,in=210] (0,1.5);
	\draw [very thick] (1,0.3) to [out=90,in=330] (0,1.5);
	\draw [very thick] (0,2.5) to [out=150,in=270] (-1,3.7);
	\draw [very thick] (0,2.5) to [out=30,in=270] (1,3.7);
		\draw [very thick, directed=.55] (0,1.5)to  (0,2.5);
	\node at (0.7,3) {$\bullet$};	
	\node at (-.5,2) {\tiny $2$};
	\node at (0.9,-.2) {\tiny $1$};
	\node at (-0.9,-.2) {\tiny $1$};
\end{tikzpicture} 
\end{gather}
\end{lem}

\begin{lem}\label{lem:decwebsym} $\DecWeb$ admits a symmetric monoidal structure.
\end{lem}
\begin{proof}
The tensor product is given by placing webs side by side. The symmetry is an isomorphism of degree zero and given on objects of the form $(k,l)$ by the $q=1$ specialization of \eqref{eq:crossing}, with a \uwave{sign correction}: 
\begin{gather*}
\begin{tikzpicture}[anchorbase,scale=.4]
	\draw [very thick, directed=0.85] (1,-1.7) to [out=90,in=270](-1,1.7);
	\draw [very thick, directed=0.85] (-1,-1.7) to [out=90,in=270](1,1.7);
	\node at (-.6,-1.2) {\tiny $k$};
	\node at (.6,-1.2) {\tiny $l$};
	\draw [thick, dashed, opacity=0.4] (0,0) circle (2cm);
\end{tikzpicture}
=  
\uwave{(-1)^{kl}}\!\!\!
\sum_{s-r=k-l} \!\!(-1)^{s-k} \; \begin{tikzpicture}[anchorbase,scale=.4]
	\draw [thick, dashed, opacity=0.4] (0,0) circle (2cm);
	\draw [very thick, directed=0.55] (1,.5) to [out=120,in=300] (-1,1);
	\draw [very thick, directed=0.55] (-1,-1) to [out=60,in=240] (1,-.5);
	\draw [very thick, directed=0.55] (1,-1.7) to (1,1.7);
	\draw [very thick, directed=0.55] (-1,-1.7) to (-1,1.7);
	\node at (-.7,-1.2) {\tiny $k$};
	\node at (0,1.2) {\tiny $r$};
	\node at (0,-.3) {\tiny $s$};
	\node at (.7,-1.2) {\tiny $l$};
\end{tikzpicture}
\end{gather*} 
The symmetry on other pairs of objects is constructed as composition of these basic symmetries. For checking the naturality of the symmetry, note that vertices still slide through other strands as in \eqref{eq:forkslide} despite the sign correction. It remains to verify that decorations migrate through such crossings. In the case $k=l=1$, this follows directly from \eqref{eq:decwebrel2}. In the more general case, one first blows up both strands into blisters of parallel $1$-labeled strands via relation \eqref{eq:webrel}. These blisters fork-slide underneath the crossing, decorations migrate onto the 1-labeled strands by \eqref{eq:decwebrel} and then through all remaining 1-1-crossings. Then one reverses the process on the other side.
\end{proof}

In Theorem~\ref{thm:DecWeb}, we will get a more intrinsic characterisation of $\DecWeb$. To prove this theorem, we take a technical detour through modules for Schur quotients of current algebras. Let $\dot{\cat{U}}(\glnn{m}[t])$ denote Lusztig's idempotent form of the universal enveloping algebra of the current algebra $\glnn{m}[t]$, which can be considered as a category with objects given by $\glnn{m}$-weights $[a_1,\dots,a_m]$. The superscript $\geq 0$ indicates that we have taken the \textit{Schur quotient} by morphisms which factor through an object with negative entries. For every $m'\geq m$, there exists an embedding $\iota\colon \dot{\cat{U}}(\glnn{m}[t])^{\geq 0} \to \dot{\cat{U}}(\glnn{m'}[t])^{\geq 0}$ given on objects by $[a_1,\dots,a_m]\mapsto [a_1,\dots,a_m,0\dots,0]$.

\begin{prop} $\DecWeb$ is isomorphic to the direct limit $\mathcal{U}$ of $\dot{\cat{U}}(\glnn{m}[t])^{\geq 0}$ for $m\to \infty$ with transition functors $\iota$.
\end{prop}
\begin{proof} 
Queffelec--Rose--Sartori \cite[Diagram (4.6)]{QRS}, building on work of Beliakova--Habiro--Lauda-Webster \cite{BHLW}, proved that there is a system of functors $\vvTr(\Phi_\infty)\colon \dot{\cat{U}}(\glnn{m}[t])^{\geq 0}\to \DecWeb$ compatible with the inclusions $\iota$, which become \textit{eventually full} and \textit{eventually faithful}. Eventual fullness means that for any morphism $F$ in $\DecWeb$ we have $F=\vvTr(\Phi_\infty)(f)$ for a morphism $f$ in $\dot{\cat{U}}(\glnn{m}[t])^{\geq 0}$ in a sufficiently large $m\geq 0$. Eventual faithfulness means that for morphisms with coinciding images $\vvTr(\Phi_\infty)(f)=\vvTr(\Phi_\infty)(g)$, there exists an $m\geq 0$ such that $\iota(f)=\iota(g)$ in $\dot{\cat{U}}(\glnn{m}[t])^{\geq 0}$. This implies that the system of functors $\vvTr(\Phi_\infty)$ defines an isomorphism as claimed.  
\end{proof}

We denote this isomorphism from $\mathcal{U}$ to $\DecWeb$ again by $\vvTr(\Phi_\infty)$.

\begin{thm}\label{thm:DecWeb} $\DecWeb$ is isomorphic to a full subcategory of the symmetric monoidal Karoubian $\C$-linear category $\Proph$, which is freely generated by a single object and an endomorphism of degree $2$. More specifically, it is isomorphic to the full subcategory $\Propo$ whose objects are tensor products of antisymmetric Schur functors in the generating object.
\end{thm}

The following proof is inspired by Cautis--Kamnitzer--Morrison's use of skew Howe duality (a generalisation of Schur-Weyl duality) to describe diagrammatic categories in \cite{CKM}. For an instance of Schur-Weyl duality for current algebras, see \cite[Section 6]{GKS}.

\begin{proof} 
There is an obvious full, essentially surjective functor $\Psi$ from the said full subcategory $\Propo$ of $\Proph$ to $\DecWeb$, but it remains to show that it is faithful. This will follow from the fact that there is an isomorphism $\alpha\colon \Propo\to \mathcal{U}$ such that $\Psi=\vvTr(\Phi_\infty)\circ \alpha$.
It suffices to prove this for $\Prop$, the free symmetric monoidal category on one object $E$ and one endomorphism $x$ (without insisting on any partial idempotent-completeness), and $\mathcal{U}'$, the direct limit of idempotent truncations of the form $1_{[1,\dots, 1,0,\dots 0]}\dot{\cat{U}}(\glnn{m}[t])^{\geq 0}1_{[1,\dots, 1,0,\dots, 0]}$.

We define a functor $\alpha\colon \Prop\to \mathcal{U}'$ by sending:
\begin{itemize}
\item $E^{\otimes m}$ to $1_{[1,\dots, 1]}$ in $\dot{\cat{U}}(\glnn{m}[t])^{\geq 0}$,
\item an $x$ on the $i$-th component of $E^{\otimes m}$ to 
$E_{m}\cdots E_{i}F_{i}[t] F_{i+1}\cdots F_{m}1_{[1,\dots, 1,0]}$ in $\dot{\cat{U}}(\glnn{m+1}[t])^{\geq 0}$,
\item  the transposition $\sigma_i$ on $E^{\otimes m}$ to $1_{[1,\dots,1]}-E_iF_i1_{[1,\dots,1]}$ in $\dot{\cat{U}}(\glnn{m}[t])^{\geq 0}$.
\end{itemize}
 and then onward to $\mathcal{U}'$ via the component maps. With this definition of $\alpha$, we have $\Psi=\vvTr(\Phi_\infty)\circ \alpha$. 
 
 A standard argument shows that $\alpha$ is surjective. Namely, a spanning set for morphism spaces in $\mathcal{U}'$ is given by the images of \textit{dotted permutations} $\bigsqcup_{m\geq 0}\{\alpha(\sigma x_1^{n_1}\cdots  x_m^{n_m})| \sigma \in S_m, n_i\geq 0\}$. It suffices to show that these remain linear independent. To this end, consider $\cat{U}(\glnn{m}[t])$ as an algebra and the $\cat{U}(\glnn{m}[t])$-module $\bV^a(\C^m\otimes \C[X])$, which decomposes into $\glnn{m}$-weight spaces $\bV^{a_1}(\C[X])\otimes \cdots \otimes \bV^{a_m}(\C[X])$. For the weight $[1,\dots,1]$ we simply get the weight space $\C[X_1,\dots, X_m]$. Since only non-negative weights arise, this descends to a $\cat{U}(\glnn{m}[t])^{\geq 0}$-module. It is straightforward to check that pre-composing with $\alpha$, we obtain the natural action of $\Prop$ where permutations act on indices and $x$ on the $i$-th strand acts by multiplication by $X_i$. It is then clear that the $\alpha$-images of dotted permutations act by linearly independent operators, and are thus linearly independent. 
\end{proof}

\begin{cor}\label{cor:annwebSchur} There is an equivalence of graded $\C$-linear tensor categories $\Kar(\iAfoam) \simeq \Proph$.
\end{cor}
\begin{proof} We already know that there exists a fully faithful functor $\Proph \to \Kar(\DecWeb) \to \Kar(\iAfoam)$ and we shall show that it is essentially surjective. To this end, let $W$ be an annular web. Proposition~\ref{prop:rotationequiv} allows us to express $W$ as a chain complex, whose chain groups are collections of concentric circles. After proceeding to the Karoubi envelope, we can decompose these further into Schur functors of a single circle. When considered as a chain complex concentrated in homological degree zero, $W$ is homotopy equivalent to an object $C(W)$ in $\Komh(\Proph)$. We may assume this object to be represented by a minimal chain complex. Since $\Proph$ is non-negatively graded and semi-simple in degree zero, the homotopy equivalence between $W$ and $C(W)$ is an isomorphism of chain complexes, and thus $C(W)$ is concentrated in degree zero. This shows that every object in $\iAfoam$ is isomorphic to an object in $\Proph$, and since the latter is idempotent complete by definition, the same holds for every object in $\Kar(\iAfoam)$. This verifies essential surjectivity and finishes the proof.
\end{proof}

\begin{rem}
\label{rem:annularsimpl}
It might be helpful to give a more direct explanation why an arbitrary annular web is isomorphic to an object in $\Proph$, and not just in the homotopy category. Indeed, we can follow the annular simplification algorithm from \cite{QR2} and use bubble removal and square switch relations to reduce a web to a collection of essential circles. At each step of the algorithm, one either replaces a web by an isomorphic one, or presents it as a direct sum of simpler webs, or presents it as a direct summand in a simpler web. Since $\Proph$ is Karoubian, all these steps show that a web is isomorphic to an object in $\Proph$, if the simpler webs are. 

Note that in \cite{QR2} Queffelec and Rose used a slightly different algorithm where, if a web is presented as a direct summand in a simpler web, it is expressed as a cone of the inclusion of complimentary summands. This way \cite{QR2}
avoids Karoubi completion, but steps into the homotopy category. By Theorem \ref{th: wedges in homotopy karoubi} the two algorithms actually agree in the homotopy category of the Karoubi completion $\Proph$. 
\end{rem}

\subsection{Braiding for annular webs}

The category of annular links and cobordisms between them has a natural braided monoidal structure. The annular Khovanov--Rozansky functor from this category to the homotopy category of complexes of annular webs and foams preserves the monoidal structure, but a priori it is not clear whether the latter has any braiding.

\begin{prop}\label{prop:symbraid}
$\Komh( \iAfoam)$ has a symmetric braiding.
\end{prop}
\begin{proof}
By Lemma~\ref{lem:decwebsym}, $\DecWeb$ and thus $\iAfoamcc$ have a symmetric braiding. This immediately extends to $\Komh( \iAfoamcc)$. Then we use the equivalence of Proposition~\ref{prop:rotationequiv} to transport this symmetric braiding to $\Komh( \iAfoam)$.
\end{proof}

Note that every object $W$ in $\iAfoam$ and thus $\Komh( \iAfoam)$ has a grading $[W]\in \N$ by weighted winding number around the annulus. Besides the braiding $\sigma_{V,W}\colon V\otimes W \to W\otimes V$ on $\Komh( \iAfoam)$ that was obtained in Proposition~\ref{prop:symbraid}, we will also consider the \textit{sign-twisted} braiding $\overline{\sigma}$, which is defined by $\overline{\sigma}_{V,W}=(-1)^{[V][W]}\sigma_{V,W}$. Transported back to $\DecWeb$, this braiding is described by the $q=1$ specialization of \eqref{eq:crossing}, i.e. the formula shown in Lemma~\ref{lem:decwebsym} without \uwave{sign correction}. 

For the following, let $\Alinkp_{S_1}$ denote the full subcategory of $\Alinkp$ with objects being collections of concentric colored circles. 
\begin{thm}
\label{annular functor braided}
The restricted annular Khovanov--Rozansky functor $\Hgenuf{-}\colon \Alinkp_{S_1}\to \Komh( \iAfoam)$ is braided with respect to the standard braiding on $\Komh( \iAfoam)$. The framed version $\Hgen{-}$ is braided with respect to the sign-twisted braiding. 
\end{thm}
\begin{proof}
The braiding on $\Alinkp$ is given by braiding isotopies, i.e. certain cobordisms which braid annular links radially past each other. Under the Khovanov--Rozansky functor $\Hgenuf{-}$, such maps induce invertible morphisms in $\Komh( \iAfoam)$, and we shall check that these morphisms agree with the symmetric braiding morphisms in $\Komh( \iAfoam)$ that were defined in Proposition~\ref{prop:symbraid}. (The case of $\Hgen{-}$ is analogous and will be omitted.) It suffices to compare these braiding morphisms on pairs of monoidal generators, i.e. two colored circles.

For two uncolored circles, the computation of the maps induced by the braiding cobordism and its inverse is simple---both involve two Reidemeister II moves---and they agree with the braiding from Proposition~\ref{prop:symbraid}. A version of this argument (without foams) appears in Grigsby--Licata--Wehrli~\cite{GLW}. We show the details here for convenience. The braiding of two uncolored circles can be described as a movie of annular link diagrams as follows:

\begin{equation}
\label{braidingmovie}
\begin{tikzpicture}[anchorbase, scale=.5]
\draw [red, thick, directed=.6, directed=.55] (0,-.5) to (0,1.5);
\draw [red, thick, directed=.6, directed=.55] (3,-.5) to (3,1.5);
\draw  (3,1.5) to (0,1.5);
\draw  (3,-.5) to (0,-.5);
\draw [very thick, ->] (3,0) to (0,0);
\draw [very thick, ->] (3,1) to (0,1);
\end{tikzpicture}
\xrightarrow{RII}
\begin{tikzpicture}[anchorbase, scale=.5]
\draw [red, thick, directed=.6, directed=.55] (0,-.5) to (0,1.5);
\draw [red, thick, directed=.6, directed=.55] (3,-.5) to (3,1.5);
\draw  (3,1.5) to (0,1.5);
\draw  (3,-.5) to (0,-.5);
\draw [very thick, ->] (3,0) to [out=180,in=0] (1.5,1) to [out=180,in=1]
(0,0);
 \draw [white,line width=.15cm] (3,1) to [out=180,in=0] (1.5,0) to [out=180,in=0] 
(0,1);
\draw [very thick, ->] (3,1) to [out=180,in=0] (1.5,0) to [out=180,in=0]  (0,1);
\draw (.75,1) node {\small $1$};
\draw (2.25,1) node {\small $2$};
\end{tikzpicture}
\xrightarrow{\text{isotopy}}
\begin{tikzpicture}[anchorbase, scale=.5]
\draw [red, thick, directed=.6, directed=.55] (0,-.5) to (0,1.5);
\draw [red, thick, directed=.6, directed=.55] (3,-.5) to (3,1.5);
\draw  (3,1.5) to (0,1.5);
\draw  (3,-.5) to (0,-.5);
\draw [very thick, ->] (3,1) to [out=180,in=0] (1.5,0) to [out=180,in=1]
(0,1);
 \draw [white,line width=.15cm] (3,0) to [out=180,in=0] (1.5,1) to [out=180,in=0] 
(0,0);
\draw [very thick, ->] (3,0) to [out=180,in=0] (1.5,1) to [out=180,in=0]  (0,0);
\draw (.75,1) node {\small $2$};
\draw (2.25,1) node {\small $1$};
\end{tikzpicture}
\xrightarrow{RII}
\begin{tikzpicture}[anchorbase, scale=.5]
\draw [red, thick, directed=.6, directed=.55] (0,-.5) to (0,1.5);
\draw [red, thick, directed=.6, directed=.55] (3,-.5) to (3,1.5);
\draw  (3,1.5) to (0,1.5);
\draw  (3,-.5) to (0,-.5);
\draw [very thick, ->] (3,0) to (0,0);
\draw [very thick, ->] (3,1) to (0,1);
\end{tikzpicture}
\end{equation}

This is a composite of a Reidemeister II move, an isotopy of the positive crossing around the annulus, and an inverse Reidemeister II move. In order to compute the composite chain map, we will recall the Reidemeister II chain maps. Here and in the following, we borrow notation from Soergel bimodules for the webs that appear: \[
R:=\begin{tikzpicture}[anchorbase,scale=.25]
	\draw [very thick, <-] (-2,1) to (2,1);
	\draw [very thick, <-] (-2,-1) to (2,-1);
 \end{tikzpicture}
\quad,\quad
B:=\begin{tikzpicture}[anchorbase,scale=.25]
	\draw [very thick, ->] (2,1)to(1.7,1) to [out=180,in=45] (.5,0) to (-.5,0) to [out=135,in=0] (-1.7,1) to (-2,1);
	\draw[very thick] (.5,0) to (-.5,0);
	\draw [very thick, ->] (2,-1) to(1.7,-1) to [out=180,in=315] (.5,0) to (-.5,0) to [out=225,in=0] (-1.7,-1) to (-2,-1);	
	\node at (0, -.5) {\tiny $2$};
\end{tikzpicture}
\]
Consider the cube of resolutions chain complex for a Reidemeister II tangle:

\begin{gather*}
\xy
(15,20)*{
=
};
(0,20)*{
\Hgen{
\begin{tikzpicture}[anchorbase, scale=.5]
\draw [very thick, ->] (3,0) to [out=180,in=0] (1.5,1) to [out=180,in=1]
(0,0);
 \draw [white,line width=.15cm] (3,1) to [out=180,in=0] (1.5,0) to [out=180,in=0] 
(0,1);
\draw [very thick, ->] (3,1) to [out=180,in=0] (1.5,0) to [out=180,in=0]  (0,1);
\draw (.75,1) node {\small $1$};
\draw (2.25,1) node {\small $2$};
\end{tikzpicture}
}
};
(85,20)*{
q^{-1} R\otimes B
};
(55,15)*{
R\otimes R
};
(55,25)*{
B \otimes B
};
(27,20)*{
q B\otimes R
};
(40,24)*{
\begin{tikzpicture}[scale=.5]
\draw [->] (0,0) to (3,.5);
\end{tikzpicture}
};
(40,27)*{
-\textrm{zip} 
};
(40,16)*{
\begin{tikzpicture}[scale=.5]
\draw [->] (0,0) to (3,-.5);
\end{tikzpicture}
};
(40,13)*{
\textrm{unzip} 
};
(70,24)*{
\begin{tikzpicture}[scale=.5]
\draw [->] (0,0) to (3,-.5);
\end{tikzpicture}
};
(70,27)*{
\textrm{unzip} 
};
(70,16)*{
\begin{tikzpicture}[scale=.5]
\draw [->] (0,0) to (3,.5);
\end{tikzpicture}
};
(70,13)*{
\textrm{zip} 
};
\endxy
\end{gather*}
 Here, the Koszul signs in the tensor product depend on an (arbitrary) ordering of the crossings of the tangle, which is shown on the left. The Reidemeister II chain maps, which connect the complex $R\otimes R$ of the trivial tangle to this complex (and vice versa), are given by identities on $R\otimes R$, as well as the negative of the following more complicated composite foam (and its reflection in a horizontal plane):

\begin{equation}
\label{eqn:r2foam}
\begin{tikzpicture} [anchorbase,scale=.5,fill opacity=0.2]
	\path [fill=red]  (.75,2.5) to [out=270,in=180] (1.5,1.75) to [out=0,in=270] 	(2.25,2.5) to [out=135,in=45](.75,2.5);
	\path [fill=red]  (.75,2.5) to [out=270,in=180] (1.5,1.75) to [out=0,in=270] 	(2.25,2.5) to [out=225,in=315](.75,2.5);
	\path [fill=red] (4.25,2) to (4.25,-.5) to (-.5,-.5) to (-.5,2) to
		[out=0,in=225] (0,2.5) to [out=270,in=180] (1.5,1) to [out=0,in=270] 
			(3,2.5) to [out=315,in=180] (4.25,2);
	\path [fill=red] (3.75,3) to (3.75,.5) to(-1,.5) to (-1,3) to [out=0,in=135]
		(0,2.5) to [out=270,in=180] (1.5,1) to [out=0,in=270] 
			(3,2.5) to [out=45,in=180] (3.75,3);
	\path[fill=yellow] (2.25,2.5) to [out=270,in=0] (1.5,1.75) to [out=180,in=270] (.75,2.5) to (0,2.5) to [out=270,in=180] (1.5,1) to [out=0,in=270] (3,2.5) to  (2.25,2.5) ;
	\draw [very thick,directed=.55] (4.25,-.5) to  (-.5,-.5);
	\draw [very thick, directed=.55] (3.75,.5) to  (-1,.5);
	\draw [very thick, red, directed=.75] (3,2.5) to [out=270,in=0] (1.5,1);
	\draw [very thick, red] (1.5,1) to [out=180,in=270] (0,2.5);
	\draw [very thick, red, rdirected=.75] (2.25,2.5) to [out=270,in=0] (1.5,1.75);
	\draw [very thick, red] (1.5,1.75) to [out=180,in=270] (.75,2.5);
	\draw  (3.75,3) to (3.75,.5);
	\draw (4.25,2) to (4.25,-.5);
	\draw (-1,3) to (-1,.5);
	\draw  (-.5,2) to (-.5,-.5);
	\draw [double,directed=.55] (3,2.5) to (2.25,2.5);
	\draw [double,directed=.55] (.75,2.5) to (0,2.5);
	\draw [very thick,directed=.55] (0,2.5) to [out=135,in=0] (-1,3);
	\draw [very thick,directed=.75] (0,2.5) to [out=225,in=0] (-.5,2);
	\draw [very thick,directed=.55] (3.75,3) to [out=180,in=45] (3,2.5);
	\draw [very thick,directed=.75] (4.25,2) to [out=180,in=315] (3,2.5);
	\draw [very thick,directed=.55] (2.25,2.5) to [out=135,in=45] (.75,2.5);
	\draw [very thick,directed=.55] (2.25,2.5) to [out=225,in=315] (.75,2.5);
\end{tikzpicture}
\end{equation}
Similarly, the other variant of the Reidemeister II move relates the invariant of the tangle

\begin{gather*}
\xy
(15,20)*{
=
};
(0,20)*{
\Hgen{
\begin{tikzpicture}[anchorbase, scale=.5]
\draw [very thick, ->] (3,1) to [out=180,in=0] (1.5,0) to [out=180,in=1]
(0,1);
 \draw [white,line width=.15cm] (3,0) to [out=180,in=0] (1.5,1) to [out=180,in=0] 
(0,0);
\draw [very thick, ->] (3,0) to [out=180,in=0] (1.5,1) to [out=180,in=0]  (0,0);
\draw (.75,1) node {\small $2$};
\draw (2.25,1) node {\small $1$};
\end{tikzpicture}
}
};
(85,20)*{
q^{-1} B\otimes R
};
(55,15)*{
R\otimes R
};
(55,25)*{
B \otimes B
};
(27,20)*{
q R\otimes B
};
(40,24)*{
\begin{tikzpicture}[scale=.5]
\draw [->] (0,0) to (3,.5);
\end{tikzpicture}
};
(40,27)*{
 -\textrm{zip} 
};
(40,16)*{
\begin{tikzpicture}[scale=.5]
\draw [->] (0,0) to (3,-.5);
\end{tikzpicture}
};
(40,13)*{
\textrm{unzip} 
};
(70,24)*{
\begin{tikzpicture}[scale=.5]
\draw [->] (0,0) to (3,-.5);
\end{tikzpicture}
};
(70,27)*{
\textrm{unzip} 
};
(70,16)*{
\begin{tikzpicture}[scale=.5]
\draw [->] (0,0) to (3,.5);
\end{tikzpicture}
};
(70,13)*{
\textrm{zip} 
};
\endxy
\end{gather*}
 to the trivial tangle diagram. The corresponding chain maps are again assembled from the identities on $R\otimes R$, the negative of the foam in \eqref{eqn:r2foam}, or its reflection respectively. Here we have chosen an ordering of the crossings which is compatible with previously chosen ordering under the isotopy in \eqref{braidingmovie}. The composite of the chain maps in \eqref{braidingmovie} thus is a difference of two terms, an identity foam over the two concentric circles, as well as a foam built as a composition of \eqref{eqn:r2foam}, the foam realising the isotopy of one copy of $B$ around the annulus, and a reflected version of \eqref{eqn:r2foam}:

\begin{equation}
\label{eqn:sigmafoam}
\begin{tikzpicture}[fill opacity=.2,anchorbase,xscale=-.7, yscale=0.4]
\annback{3}{2}{3.5}
\fill [fill=red] (0.625,4.75) to (3.625,4.75) to (3.625,1.25) to (0.625,1.25) to (0.625,4.75);
\fill [fill=red] (0.375,4.25) to (3.375,4.25) to (3.375,0.75) to (.375,0.75) to (0.375,4.25) ;
\draw [red] (0.375,4.25) to (0.375,0.75);
\draw [red] (0.625,4.75) to (0.625,1.25);
\draw [red] (3.375,4.25) to (3.375,0.75);
\draw [red] (3.625,4.75) to (3.625,1.25);
\coordinate (a) at (0.125,-2.75);
\draw[very thick, directed=.5] (.25,3.5)+(a) to ($(3.25,3.5)+(a)$);
\draw[very thick, directed=.5] (.5,4)+(a) to ($((3.5,4)+(a)$);
\coordinate (a) at (0.125,0.75);
\draw[very thick, directed=.5] (.25,3.5)+(a) to ($(3.25,3.5)+(a)$);
\draw[very thick, directed=.5] (.5,4)+(a) to ($(3.5,4)+(a)$);
\annfront{3}{2}{3.5}
\end{tikzpicture}
\quad - \quad
\begin{tikzpicture}[fill opacity=.2,anchorbase,xscale=-.7, yscale=0.4]
\annback{3}{2}{3.5}
\draw [very thick, red, directed=.55] (.5,2.25) to (3.5,2.25);
\draw [very thick, red, directed=.55] (3.5,3.25) to  (.5,3.25);
\fill [fill=yellow] (3.5,3.25) to  (.5,3.25) to (.5,2.25) to (3.5,2.25) to (3.5,3.25) ;
\fill [fill=red] (3.5,3.25) to (.5,3.25) to [out=75,in=270] (0.625,4.75) to (3.625,4.75) to [out=270,in=75] (3.5,3.25);
\fill [fill=red] (3.5,3.25) to (.5,3.25) to [out=105,in=270] (0.375,4.25) to (3.375,4.25) to [out=270,in=105] (3.5,3.25);
\fill [fill=red] (3.5,2.25) to (.5,2.25) to [out=255,in=90]  (.375,0.75) to (3.375,0.75) to [out=90,in=255] (3.5,2.25);
\fill [fill=red] (3.5,2.25) to (.5,2.25) to [out=285,in=90]  (.625,1.25) to (3.625,1.25) to [out=90,in=285] (3.5,2.25);
\draw [red] (0.375,4.25) to [out=270,in=105] (.5,3.25);
\draw [red] (0.625,4.75) to [out=270,in=75] (.5,3.25);
\draw [red] (3.625,4.75) to [out=270,in=75] (3.5,3.25);
\draw [red] (3.375,4.25) to [out=270,in=105] (3.5,3.25);
\draw [red] (3.5,3.25) to (3.5,2.25);
\draw [red] (0.5,3.25) to (0.5,2.25);
\draw [red] (.625,1.25) to [out=90,in=285] (.5,2.25);
\draw [red] (.375,0.75) to [out=90,in=255] (.5,2.25);
\draw [red] (3.375,0.75) to [out=90,in=255] (3.5,2.25);
\draw [red] (3.625,1.25) to [out=90,in=285] (3.5,2.25);
\coordinate (a) at (0.125,-2.75);
\draw[very thick, directed=.5] (.25,3.5)+(a) to ($(3.25,3.5)+(a)$);
\draw[very thick, directed=.5] (.5,4)+(a) to ($((3.5,4)+(a)$);
\coordinate (a) at (0.125,0.75);
\draw[very thick, directed=.5] (.25,3.5)+(a) to ($(3.25,3.5)+(a)$);
\draw[very thick, directed=.5] (.5,4)+(a) to ($(3.5,4)+(a)$);
\annfront{3}{2}{3.5}
\end{tikzpicture}
\end{equation}
This agrees with the braiding on $\iAfoam$ defined in Proposition~\ref{prop:symbraid}.

An analogous argument applies in the case of two colored circles---it uses an explicit description of the chain maps associated to colored Reidemeister II moves---and shows that the braiding of such is given by the rotation foam generated by the linear combination of webs shown in the proof of Lemma~\ref{lem:decwebsym}. 
\end{proof}

In fact, we expect that the annular Khovanov--Rozansky functors are braided on the entire annular link category, but we do not know how to prove this without assuming a stronger functoriality property, which has not been established yet.

\begin{conj} 
The annular Khovanov--Rozansky functors $\Alinkp\to \Komh( \iAfoam)$ are braided.
\end{conj}

For later use, we also record the following observation, where we write $\sigma$ for the linear combination of foams shown in \eqref{eqn:sigmafoam}. 

\begin{lem}\label{lem:antisymm-def-matchup} In the Karoubi envelope of $\iAfoam$ the image  of the anti-symmetrizer in $\Q [S_k]$ under $\Hgenuf{-}$ is isomorphic to the k-colored essential circle.
\end{lem}
\begin{proof} 

For $k=2$, the anti-symmetrizer is $(\id-\sigma)/2$. Note that this is exactly $1/2$ times the foam shown on the left-hand side of \eqref{eqn:sigmafoam}. Cutting this foam in half by a horizontal plane produces a merge foam $M$ and a splitter foam $S$. We have $(\id-\sigma)/2 = S\circ M/2$ and $M/2 \circ S=\id_2 $. This implies that $S$ and $M/2$ represent the desired mutually inverse isomorphisms in $\Kar(\iAfoam)$.  

The case $k>2$ follows, since the anti-symmetrizers in $\C[S_k]$ and also the projections onto $k$-colored essential circles in $\iAfoam$ can be constructed from the $k=2$ cases in the same way.
\end{proof}

Everything in this subsection works in the finite-rank case, i.e. for annular $\glnn{N}$ foam categories $\NAfoamp$. In this setting, essential circles of label $N+1$ are isomorphic to the zero object, which implies that the uncolored essential circle is of rank (at most) $N$ in the sense of Definition~\ref{def:rank}.

\begin{rem} An analogue of Lemma~\ref{lem:antisymm-def-matchup} shows that the framed Khovanov--Rozansky functors $\Hgen{-}$ send the symmetrizer in $\C[S_k]$ to the $k$-colored unknot. This is at odds with our interpretation of that colored circle as corresponding to the exterior power $\bV^k$ of the uncolored circle. The origin for this discrepancy is the relative homological shift between the two conventions for crossings \eqref{eq:crossingcxuf} and \eqref{eq:crossingcx}, which translates into a sign-twist on the braiding.

\end{rem}
\subsection{Evaluation}

Here we recall the evaluation of annular homology  developed by Queffelec and Rose.
Let $L$ be an annular link, then Khovanov-Rozansky functor sends it to a complex of webs,
and by Corollary \ref{cor:annwebSchur} we can replace it by a complex of Schur functors of $E$ in $\Komh(\Proph)$. The object $E$ appears as the invariant of the essential planar unknot in the annulus and the endomorphism $X$ encodes information about the $\C[X]$-actions in link homologies that are typically associated with the choice of a base point on the link. Proposition \ref{prop: evaluation} now immediately implies the following: 

\begin{thm}
Let $\CC$ be an arbitrary additive symmetric monoidal category, and let $\Komh(\CC)$ be the corresponding homotopy category. Suppose that $\CE$ is an object of $\CC$ with an endomorphism $X$. Then there is a unique functor 
\[\AKhR(\mathcal{E},X)\colon \Alinkp \to \Komh(\CC)\] which factors through the Khovanov-Rozansky functor, sends the essential planar unknot to $\CE$, the base point action to $X$, and the braiding of two unknots to the symmetry on $\CE\otimes \CE$.
\end{thm}

The results of \cite{QR2} can be then rephrased in the following way:

\begin{thm}[\cite{QR2}]
\label{thm:eval}
If $\CC=\mathrm{gr}^\Z\mathrm{Vect}$, the category of $\Z$-graded vector spaces (with the swap symmetry), $\CE=\C[X]/X^N$, and $X$ is the endomorphism given by multiplication by $x$, then the functor $\AKhR(\CE,X)$ agrees with the $\glN$ Khovanov-Rozansky homology. If $\CE=\C[X]/P(X)$ for a degree $N$ monic polynomial, then $\AKhR(\CE,X)$ agrees with the deformed Khovanov--Rozansky homology studied in \cite{Wu2,RW}.
If $\CC=\mathrm{gr}^\Z\Rep(U(\glnn{N}))$ and $\CE=V=\C^N$ is the vector representation of $U(\glnn{N})$, then the functor $\AKhR(\CE,0)$ agrees with the annular Khovanov-Rozansky homology.
\end{thm}

\section{Coxeter braids, categorified}
\label{sec:categorified Coxeter}

\subsection{Positive Coxeter braids}

The purpose of this section is to prove the following theorem.

\begin{thm}\label{thm:posnegcox} Let now $C_n^-$ and $C_n^+$ denote the annular complexes of the $(n-1)$-fold negatively and positively stabilized unknots. Then we have:
\begin{align*}
C_n^-  &\simeq  [q^{n-1}S^{n}(\E)\to  \ldots \to q^{3-n}~\Sch^{2,1^{n-2}}(\E)\to q^{1-n} \uwave{\bV^n(\E)}] \cong \Cube_n^{\mathrm{sign}}\\
C_n^+[n-1] & \simeq [q^{n-1}\bV^{n}(\E)\to  \ldots \to q^{3-n}~\Sch^{n-1,1}(\E)\to q^{1-n}\uwave{S^n(\E)}]\cong \Cube_n^{S_n}
\end{align*} 
\end{thm}

\begin{lem}
\label{lem:planarunknot}
Upon evaluation as in Theorem~\ref{thm:eval}, these complexes compute the planar $\glnn{N}$ Khovanov--Rozansky homologies of stabilized unknots.
\end{lem}
\begin{proof} This follows from Example~\ref{exa:eval stab unknot}.
\end{proof}

To start the proof of Theorem~\ref{thm:posnegcox}, note that it suffices to prove one of the homotopy equivalences. The other one follows by symmetry. We focus on the positive stabilization and consider the shifted complex $\overline{C}^+_n:=q^{n-1} C_n^+[n-1]$, which has its terminal chain group in homological and $q$-degree zero. We also define annular complexes $X_{n,l}$ as in Figure~\ref{fig:Xannular}. Clearly, $X_{n,0}=\overline{C}^+_n$.

\begin{figure}[h]
\begin{align*}
X_{n,l}:=
q^{n-1+l}\Hgen{
\begin{tikzpicture}[anchorbase,scale=.8]
\draw (0,0)--(0,1)--(2,1)--(2,0)--(0,0);
\draw (1,0.5) node {\small$\sigma_{n-1}\cdots\sigma_{1}$};
\draw[thick] (0.25,1) to  (0.25,2) to [out=90,in=180] (2.5,3.25) to [out=0,in=90] (4.75,2) to (4.75,0) to [out=270,in=0] (2.5,-1.25) to [out=180,in=270] (0.25,0); 
\draw[thick] (1.25,1) to  (1.25,2) to [out=90,in=180] (2.75,3) to [out=0,in=90] (4.25,2) to (4.25,0) to [out=270,in=0] (2.75,-1) to [out=180,in=270] (1.25,0); 
\draw[thick] (1.75,1) to (1.75,1.25) to [out=90,in=270] (2,1.5) to (2,1.75) to [out=90,in=270] (1.75,2) to [out=90,in=180] (2.75,2.75) to [out=0,in=90] (3.75,2) to (3.75,0) to [out=270,in=0] (2.75,-.75) to [out=180,in=270] (1.75,0); 
\draw[thick] (2.25,1) to (2.25,1.25) to [out=90,in=270] (2,1.5) to (2,1.75) to [out=90,in=270] (2.25,2) to [out=90,in=180] (2.75,2.5) to [out=0,in=90] (3.25,2) to (3.25,0) to [out=270,in=0] (2.75,-.5) to [out=180,in=270] (2.25,0) to (2.25,1);
\draw[thick]  (2,1.5) to (2,1.75) ;
\draw (3.1,1) node {\tiny $l$};
\draw (.75,1.5) node {$\dots$};
\draw (2.5,1.625) node {\tiny $l+1$};
\draw (2.75,0.5) node {$\times$};
\end{tikzpicture}
} [n-1]
\end{align*}
\caption{
\label{fig:Xannular}}
\end{figure}

\begin{lem}
For $l\geq 0$ and $n>1$ we have
\[
X_{1,l}\cong [l+1]_q\bV^{l+1}(\E),\quad \ X_{n,l}\simeq [X_{n-1,l+1}\to \uwave{\overline{C}^+_{n-1}\otimes \bV^{l+1}(\E)}]
\]
where $[l+1]_q=(1+q^2+\ldots+q^{2l})$ is an asymmetric quantum integer.
\end{lem}

\begin{proof}
The isomorphism for $X_{1,l}$ is due to a bigon removal. To check the homotopy equivalence for $X_{n,l}$, we resolve the right-most crossing $\sigma_{n-1}$ and simplify as follows.
\begin{align*}
& \left[ q^{n-1+l}\;
\uwave{
\Hgen{
\begin{tikzpicture}[anchorbase,scale=.8]
\draw (-.125,0)--(-.125,.5)--(1.625,.5)--(1.625,0)--(-.125,0);
\draw[thick] (.75,0.25) node {\tiny $\sigma_{n-2}\cdots \sigma_1$};
\draw[thick, dotted] (0.25,1) to  (0.25,1.5);
\draw[thick, dotted] (0.25,-.5) to  (0.25,-1);
\draw[thick] (0.25,.5) to  (0.25,1);
\draw[thick] (0.25,0) to  (0.25,-.5);
\draw[thick] (1.25,.5) to [out=90,in=270] (1.5,.75) to (1.5,1) to [out=90,in=270] (1.25,1.25) to (1.25,2) to [out=90,in=180] (2.75,3) to [out=0,in=90] (4.25,2) to (4.25,0) to [out=270,in=0] (2.75,-1) to [out=180,in=270] (1.25,0); 
\draw[thick] (1.75,.5) to [out=90,in=270] (1.5,.75) to (1.5,1) to [out=90,in=270] (2,1.5) to (2,1.75) to [out=90,in=270] (1.75,2) to [out=90,in=180] (2.75,2.75) to [out=0,in=90] (3.75,2) to (3.75,0) to [out=270,in=0] (2.75,-.75) to [out=180,in=270] (1.75,0) to (1.75,.5); 
\draw[thick] (2.25,1) to (2.25,1.25) to [out=90,in=270] (2,1.5) to (2,1.75) to [out=90,in=270] (2.25,2) to [out=90,in=180] (2.75,2.5) to [out=0,in=90] (3.25,2) to (3.25,0) to [out=270,in=0] (2.75,-.5) to [out=180,in=270] (2.25,0) to (2.25,1);
\draw[thick]  (2,1.5) to (2,1.75) ;
\draw[thick]  (1.5,.75) to (1.5,1) ;
\draw (3.1,1) node {\tiny $l$};
\draw (.75,.75) node {$\dots$};
\draw (2.5,1.625) node {\tiny $l+1$};
\draw (2.75,0.5) node {$\times$};
\end{tikzpicture}
}
}
\to
q^{n-2+l}
\Hgen{
\begin{tikzpicture}[anchorbase,scale=.8]
\draw (-.125,0)--(-.125,.5)--(1.625,.5)--(1.625,0)--(-.125,0);
\draw[thick] (.75,0.25) node {\tiny $\sigma_{n-2}\cdots \sigma_1$};
\draw[thick, dotted] (0.25,1) to  (0.25,1.5);
\draw[thick, dotted] (0.25,-.5) to  (0.25,-1);
\draw[thick] (0.25,.5) to  (0.25,1);
\draw[thick] (0.25,0) to  (0.25,-.5);
\draw[thick] (1.25,.5) to (1.25,1.25) to (1.25,2) to [out=90,in=180] (2.75,3) to [out=0,in=90] (4.25,2) to (4.25,0) to [out=270,in=0] (2.75,-1) to [out=180,in=270] (1.25,0); 
\draw[thick] (1.75,.5) to (1.75,1.25) to [out=90,in=270] (2,1.5) to (2,1.75) to [out=90,in=270] (1.75,2) to [out=90,in=180] (2.75,2.75) to [out=0,in=90] (3.75,2) to (3.75,0) to [out=270,in=0] (2.75,-.75) to [out=180,in=270] (1.75,0) to (1.75,.5); 
\draw[thick] (2.25,1) to (2.25,1.25) to [out=90,in=270] (2,1.5) to (2,1.75) to [out=90,in=270] (2.25,2) to [out=90,in=180] (2.75,2.5) to [out=0,in=90] (3.25,2) to (3.25,0) to [out=270,in=0] (2.75,-.5) to [out=180,in=270] (2.25,0) to (2.25,1);
\draw[thick]  (2,1.5) to (2,1.75) ;
\draw (3.1,1) node {\tiny $l$};
\draw (.75,.75) node {$\dots$};
\draw (2.5,1.625) node {\tiny $l+1$};
\draw (2.75,0.5) node {$\times$};
\end{tikzpicture}
}
\;
\right]
[n-1]
\\
\cong& \left[\;
q^{n-1+l} \Hgen{
\begin{tikzpicture}[anchorbase,scale=.8]
\draw (-.125,0)--(-.125,.5)--(1.625,.5)--(1.625,0)--(-.125,0);
\draw[thick] (.75,0.25) node {\tiny $\sigma_{n-2}\cdots \sigma_1$};
\draw[thick, dotted] (0.25,1) to  (0.25,1.5);
\draw[thick, dotted] (0.25,-.5) to  (0.25,-1);
\draw[thick] (0.25,.5) to  (0.25,1);
\draw[thick] (0.25,0) to  (0.25,-.5);
\draw[thick] (1.25,.5) to [out=90,in=270] (1.5,.75) to (1.5,1) to [out=90,in=270] (1.25,1.25) to (1.25,2) to [out=90,in=180] (2.75,3) to [out=0,in=90] (4.25,2) to (4.25,0) to [out=270,in=0] (2.75,-1) to [out=180,in=270] (1.25,0); 
\draw[thick] (2,.25) to [out=90,in=270] (1.5,.75) to (1.5,1) to [out=90,in=270] (2,1.5); 
\draw[thick] (2,1) to (2,1.5) to (2,1.75) to [out=90,in=270] (2,1.75) to [out=90,in=180] (2.75,2.5) to [out=0,in=90] (3.5,1.75) to (3.5,.25) to [out=270,in=0] (2.75,-.5) to [out=180,in=270] (2,0.25) to (2,1);
\draw[thick]  (1.5,.75) to (1.5,1) ;
\draw (2.18,1) node {\tiny $l$};
\draw (.75,.75) node {$\dots$};
\draw (2.5,1.625) node {\tiny $l+1$};
\draw (2.75,0.5) node {$\times$};
\end{tikzpicture}
} [n-2]
\to
\uwave{
[l+1]_q \overline{C}^+_{n-1}\otimes \bV^{l+1} (\E)
}
\right] 
\\
\cong&
\left[X_{n-1,l+1} \oplus q^2 [l]_q \overline{C}^+_{n-1}\otimes \bV^{l+1}(\E) \to \uwave{[l+1]_q  \overline{C}^+_{n-1}\otimes \bV^{l+1}(\E)}
\right]
\end{align*}
In the second step we have used the bigon relation, and in the last step the square-switch relation. The degree zero components of the differential between the copies of $\overline{C}^+_{n-1}\otimes \bV^{l+1}(\E)$ are identities up to non-zero scalars \cite[Direct Sum Decomposition (IV)]{Wu2}, and after Gaussian elimination, we obtain the claimed form.
\end{proof}

\begin{cor}
There is a natural map $a\colon \overline{C}^+_{n-1}\otimes \E\to X_{n,0}=\overline{C}^+_n$.
\end{cor}

\begin{cor}
One can write
\begin{equation}
\label{En recursion}
\overline{C}^+_{n}\simeq\left[[n]_q\bV^n(\E)\oplus \bigoplus_{i=1}^{n-1} \overline{C}^+_i\otimes \bV^{n-i}(\E), D\right],
\end{equation}
where $[n]_q\bV^n(\E):=\bV^n(\E)\oplus q^2 \bV^n(\E)\oplus\cdots \oplus q^{2n-2} \bV^n(\E)$. The differential $D$ consists of the internal differential in $\overline{C}^+_i$, degree zero maps $\overline{C}^+_{k-1}\otimes \bV^{n-k+1}(\E)\to \overline{C}^+_{k}\otimes \bV^{n-k}(\E)$ obtained as compositions
\[
\overline{C}^+_{k-1}\otimes \bV^{n-k+1}(\E)\to \overline{C}^+_{k-1}\otimes \E \otimes \bV^{n-k}(\E)\xrightarrow{a} \overline{C}^+_{k}\otimes \bV^{n-k}(\E),
\]
and some differentials out of $[n]_q\bV^n(\E)$, as well as possibly some higher differentials.
\end{cor}

\begin{proof}
We prove by induction the existence of a more general expression
\[
X_{n,l}\simeq\left[[n+l]_q\bV^{n+l}(\E)\oplus \bigoplus_{i=1}^{n-1} \overline{C}^+_i\otimes \bV^{n+l-i}(\E), D\right],
\]
where the differentials are described as above. Indeed, for $n=1$ we get $X_{1,l}\cong[l+1]_{q}\bV^{l+1}(\E)$, and for $n>1$ we use the induction hypothesis and 
\[
X_{n,l}\simeq[X_{n-1,l+1}\to \uwave{\overline{C}^+_{n-1}\otimes \bV^{l+1}(\E)}].
\] 
Now, for $l=0$ we get $X_{n,0}=\overline{C}^+_n$. 
\end{proof}

\begin{thm}
\label{thm:En}
We have
\begin{equation}
\label{En formula}
\overline{C}^+_n\simeq[q^{2n-2}\bV^{n}(\E)\to  \ldots \to q^{2}~\Sch^{n-1,1}(\E)\to \uwave{S^n(\E)}].
\end{equation}
\end{thm}

Note that Theorem~\ref{thm:posnegcox} follows from Theorem~\ref{thm:En} by grading shifts and symmetry. We will prove Theorem~\ref{thm:En} by induction in $n$ using the recursive description \eqref{En recursion}. To illustrate this, we first consider the examples $n=2,3$.

\begin{exa}
For $n=2$ the complex \eqref{En recursion} has the form $\overline{C}^+_2\cong[(1+q^2)\bV^2(\E)\to\uwave{\E\otimes \E}]$, and after cancellation we get 
$\overline{C}^+_2\simeq q^2\bV^2(\E)\to \uwave{S^2(\E)}$.
\end{exa}

\begin{exa}
For $n=3$ the complex \eqref{En recursion} has the form 
\[
\overline{C}^+_3\simeq\left[
\begin{tikzcd}
 &  \E\otimes \bV^2(\E) \arrow{r}& \uwave{S^2(\E)\otimes \E}\\
(1+q^2+q^4)\bV^3(\E) \arrow{ur} \arrow{r} &  q^2\bV^2(\E)\otimes \E \arrow{ur}& \\
\end{tikzcd}
\right]
\]
The degree zero differentials are organized in three subquotient complexes:
\[
[\bV^3(\E)\to \E\otimes \bV^2(\E)\to \uwave{S^2(\E)\otimes \E}]\simeq \uwave{S^3(\E)},
\]
\[
[q^2\bV^3(\E)\to q^2\bV^2(\E)\otimes \E]\simeq q^2\Sch^{2,1}(\E),
\]
and $q^4\bV^3(\E)$. Here, in cancelling, we assume that the shown differentials are non-zero. Then we get
\[
\overline{C}^+_3\simeq[q^4\bV^3(\E)\to q^2\Sch^{2,1}(\E)\to S^3(\E)].
\]
The case where some of the shown degree zero differentials are zero can be excluded, because then, in specialization $\E=\C[x]/x^N$, the homology would be larger than expected, contradicting Lemma~\ref{lem:planarunknot}.
\end{exa}

\begin{proof}[Proof of Theorem~\ref{thm:En}]
Assume that \eqref{En formula} holds for all $k<n$. The terms in \eqref{En recursion} in homological degree $k$ are assembled from the terms in $\overline{C}^+_{n-j}\otimes \bV^{j}(\E)$ in homological degree $k+1-j$. By the induction hypothesis, the latter is homotopic to $q^{2k+2-2j}\Sch^{n-k-1,1^{k+1-j}}(\E)\otimes \bV^j(\E)$. The differential decreases the homological degree $k$ by one, and acts between $\overline{C}^+_{n-j}\otimes \bV^{j}(\E)\to \overline{C}^+_{n-j'}\otimes \bV^{j'}(\E)$ with $j'\le j$. For $j'<j-1$ we get $k-j'>k+1-j$, and such a differential can be ruled out as it would have negative $q$-degree. Therefore, in this presentation, the only surviving differentials are internal for $\overline{C}^+_{n-j}$ (that is, $j'=j$) of $q$-degree two, and between the neighbors $j'=j-1$, of $q$-degree zero. The example for $n=3$ above illustrates this point.

This means that we have explicitly identified all differentials in \eqref{En recursion} except for the ones connecting the leftmost copies of $\bV^n(\E)$ to $\overline{C}^+_{n-j}\otimes \bV^{j}(\E)$. 

To simplify this complex, we first consider the degree zero differentials. In $q$-degree $2t$ we get the following complex (ignoring the leftmost term), terminating in homological degree $t$:
\[
\bV^{t+1}(\E)\otimes \bV^{n-t-1}(\E)\to \Sch^{2,1^{t}}\otimes \bV^{n-t-2}(\E)\to \cdots \to \Sch^{n-1-t,1^{t}}\otimes \E
\]
The differentials are induced by compositions
\begin{multline*}
\Sch^{m,1^{t}}(\E)\otimes \bV^{n-t-m}(\E)\to  S^{m}(\E)\otimes \bV^{t}(\E)\otimes \bV^{n-t-m}(\E)\to \\ \to S^{m-1}(\E)\otimes \bV^{t}(\E)\otimes \bV^{n-t-m+1}(\E)\to \Sch^{m-1,1^{t}}(\E)\otimes \bV^{n-t-m+1}(\E),
\end{multline*}
and, in particular, they are non-trivial. One can check that these cancel almost everything, except for a copy of $q^{2t}\bV^n(\E)$ in homological degree $(n-2)$ and $q^{2t}\Sch^{n-t,1^{t}}(\E)$ in homological degree $t$. 

We claim that the leftmost term $[n]_q\bV^n(\E)$ cancels all $q^{2j}\bV^n(\E)$ except for $q^{2n-2}\bV^n(\E)$. Moreover, the remaining differentials are all non-zero and, thus, determined up to scalars. Both claims hold because otherwise, in specialization $\CE=\C[X]/X^N$, the homology would be larger than expected, contradicting Lemma~\ref{lem:planarunknot}.
\end{proof}

\begin{rem}
The annular $\slnn{2}$ homology of the stabilized unknot was computed by Grigsby, Licata and Wehrli \cite{GLW}. It agrees with our computation up to conventions, as we shall now explain. Let $V_n$ denote the $(n+1)$-dimensional irreducible representation of $\slnn{2}$. Note that $\Sch^{n}(V_1)\cong V_{n}$, $\Sch^{n-1,1}(V_1)\cong V_{n-2}$ and $\Sch^{n-i,1^{i}}(V_2)=0$ for $i>1$. Thus, for $\E=V_1$ and $x=0$ we evaluate $C_{n+1}^+\simeq[ \uwave{0}\to \ldots \to 0 \to q^{2-n} V_{n-1} \xrightarrow{0} q^{-n} V_{n+1}]$. As in \cite[Section 9.2]{GLW}, the annular $\slnn{2}$ homology of the $n$-fold stabilized unknot consists of the two irreducible $\slnn{2}$-representations $V_{n+1}$ and $V_{n-1}$ in adjacent homological degrees, with a difference in $q$-degrees of $2$.  
\end{rem}

\subsection{Morphisms}
We shall now describe the hom spaces between the annular complexes associated to closures of Coxeter braids. We start by describing basic chain maps between complexes associated to braids. In doing so, we again borrow notation from the theory of Soergel bimodules. 

\begin{defi} Consider the following chain map:
\[\xymatrix@R-1pc{
\Hgen{\sigma}\ar@{.}[rr] \ar[d]^{f}& & \uwave{B} \ar[r] \ar[d]^{\id}  & q^{-1} R   \\
\Hgen{\sigma^{-1}}\ar@{.}[r]& q R   \ar[r]& \uwave{B} & }
\]
The distinguished triangle $\xymatrix{
\Hgen{\sigma} \ar[r]^{f}&\Hgen{\sigma^{-1}}\ar[r] & \mathrm{Cone}(f)\ar[r]& \Hgen{\sigma}[1]}$ is called the skein triangle. Here we have \[\mathrm{Cone}(f)\simeq[\xymatrix{q R \ar[r]^{x_1-x_2} & q^{-1} \uwave{R}}].\]
\end{defi}

Note that the annular closure of $\mathrm{Cone}(f)$ is precisely $\Cube_2$. We have seen that $\Cube_2 \cong \Cube_2^{S_2}\oplus \Cube_2^{\mathrm{sign}} \simeq  C^+_2[1] \oplus C^-_2$. In other words, under annular closure, the skein triangle splits: 
 
\[
\xymatrix{ 0 \ar[r]&\Hgen{\hat{\sigma}^{-1}}\ar@<.5ex>[r] & \mathrm{Cone}(\hat{f})\ar@<.5ex>[r]\ar@{-->}@<.5ex>[l]& \Hgen{\hat{\sigma}}[1] \ar[r]\ar@{-->}@<.5ex>[l] & 0}
\]
Here we want to emphasize that the dashed maps only appear in the annular closure. 

\begin{rem}\label{rem:twistband} There are also non-zero cobordism-induced maps $R \to q \Hgen{\sigma} [1]$ and $q^{-1}\Hgen{\sigma^{-1}}[-1] \to R$, which can be interpreted as gluing in a twisted band that increases the writhe:
\[\xymatrix@R-1pc{
 R\ar@{.}[rr] \ar[d] & & \uwave{R} \ar[d]^{\id}    \\
q \Hgen{\sigma}[1]\ar@{.}[r]&q B \ar[r]&  \uwave{R}  }
\quad\quad \quad \quad
\xymatrix@R-1pc{
q^{-1} \Hgen{\sigma^{-1}}[-1]\ar@{.}[r] \ar[d]& \uwave{R}\ar[r] \ar[d]^{\id}& q^{-1} B     \\
R\ar@{.}[r]& \uwave{R} &   }
\] 

In the $\glnn{N}$-evaluations, there also exist non-trivial maps $R \to q^{1-2N} \Hgen{\sigma^{-1}}[-1]$ and $q^{2N-1}\Hgen{\sigma} [1]\to R$ associated to twisted bands that decrease the writhe. 

In both cases, these cobordism-induced chain maps are unrelated to the maps in the skein triangle. 
\end{rem}

\begin{rem} Under $\glnn{N}$-evaluation, we have a partially topological description of the chain map $f$ in the skein triangle. We start with $\Hgen{\sigma}$ and follow the Reidemeister II chain map to $\Hgen{\sigma^{-1}\sigma \sigma}$ and then a saddle cobordism map to $q^{1-N}\Hgen{\sigma^{-1}\#H}$ where $H$ denotes a positive Hopf link and the connect sum is taken on the new over-strand. Finally, the projection to the top degree generator of the reduced Hopf link homology (which is not cobordism-induced) induces an onward map to $\Hgen{\sigma^{-1}}$. The composition is the chain map $f$.
\end{rem}

\begin{lem}
\label{lem: solomon}
Let $U_n$ be the $(n-1)$-dimensional reflection representation of $S_n$. Let $p_1,\ldots,p_n$ be an algebraically independent generating set for 
$(\mathbb{C}[x_1,\ldots,x_n])^{S_n}$, for example, power sum symmetric polynomials.
Then $\Hom_{S_n}(\bV^kU_n,\mathbb{C}[x_1,\ldots,x_n])$ is a free module over $(\mathbb{C}[x_1,\ldots,x_n])^{S_n}$ generated by the coefficients of $dp_{i_1}\wedge\cdots\wedge dp_{i_k}$ for all $2\le i_1<i_2<\ldots<i_k\le n$. 
\end{lem} 
 
\begin{proof}
Clearly, $\Hom_{S_n}(\bV^kU_n,\mathbb{C}[x_1,\ldots,x_n])$ is free over $\sum x_i$, so we can consider 
$\Hom_{S_n}(\bV^kU_n,\mathbb{C}[U_n])$ instead. It can be identified with the space of $S_n$--invariant differential forms on $U_n$, which by a theorem of Solomon \cite{Solomon2} is isomorphic to the space of differential forms on $U_n/S_n=\mathrm{Spec}\ \C[p_2,\ldots,p_n]$.
\end{proof} 
 
Let $\Cube_n$, as before, denote the Koszul complex for $x_i-x_{i+1}$ acting on $\E^{\otimes n}$.
Then we have:

\begin{thm}
\label{thm:endcube}

The morphisms between the tensor products of $\Cube_n$ can be described as follows:
\begin{itemize}
\item The endomorphism algebra of $\Cube_n$ is:
\[
\End(\Cube_n)\cong \bV^{\bullet}(U_n)\otimes \C[x]\rtimes \C[S_n],
\]
where $U_n$ is the $(n-1)$-dimensional reflection representation of $S_n$ and $S_n$ acts trivially on $x$.  Moreover, $x$ is of $q$-degree $2$ and $U_n$ is supported in homological degree $-1$ and $q$-degree $-2$.
\item The spaces of morphisms between $\Cube_{n}\otimes \Cube_{m}$ and $\Cube_{n+m}$ are generated by the canonical maps implicit in the description of 
$$\Cube_{n+m}\cong \mathrm{Cone}(q \Cube_{n}\otimes \Cube_{m}\xrightarrow{x_n-x_{n+1}}q^{-1}\uwave{\Cube_n\otimes \Cube_{m}}).$$

 The actions of $x\in \End(\Cube_i)$ on this space agree (up to homotopy) for $i=n,m,n+m$, and the actions of exterior algebras are naturally identified under the induction and restriction maps between $U_n\oplus U_m$ and $U_{n+m}$.
\item All other morphisms are induced by these.
\end{itemize}
\end{thm}

\begin{proof}
Let us compute the endomorphism ring of $\Cube_n$. We have $\End(\E^{\otimes n})=\C[x_1,\ldots,x_n]\rtimes \C[S_n]$, so $\End(\Cube_n)$ is isomorphic to a complex built out of these. Since the differential does not involve the action of $S_n$, we can ignore the $\C[S_n]$ factor for a while.
Now
\[
\End(\bV^{\bullet}(U_n)\otimes \E^{\otimes n})\cong \bV^{\bullet}(U_n)\otimes \bV^{\bullet}(U^*_n)\otimes \End(\E^{\otimes n})\cong \bV^{\bullet}(U_n)\otimes \bV^{\bullet}(U_n)\otimes \C[x_1,\ldots,x_n]\rtimes \C[S_n]. 
\] 
Here we have identified $U_n$ with its dual via the $S_n$--invariant nondegenerate bilinear form.
This space of maps carries the natural differential
\begin{equation}
\label{eq: def D}
D(\alpha\otimes\beta\otimes f)=d(\alpha)\otimes \beta\otimes f \pm \alpha\otimes d(\beta)\otimes f,
\end{equation}
 where $d$ is the Koszul differential on $\bV^{\bullet}(U_n)\otimes
 \C[x_1,\ldots,x_n]$. Since $(x_1-x_2,\dots,x_{n-1}-x_n)$ is a regular sequence
 in $R=\C[x_1,\ldots,x_n]$, the homology of $(\bV^{\bullet}(U_n)\otimes R,d)$ is isomorphic to $\C[x]$.
 
Equation \eqref{eq: def D} presents $D$ as a sum of two anticommuting differentials, which induces a spectral sequence. 
 The first differential has homology $\bV^{\bullet}(U_n)\otimes \C[x]\rtimes \C[S_n]$. Now the second differential vanishes, so the spectral sequence collapses at $E_2$ page, and 
\[
\End(\Cube_n)\cong\bV^{\bullet}(U_n)\otimes \C[x]\rtimes \C[S_n]. 
\]
Here $x$ has $q$-degree 2 and homological degree 0 while the generators $\epsilon_1,\ldots,\epsilon_{n-1}$ of 
$\bV^{\bullet}(U_n)$ have homological degree $-1$ and $q$-degree $-2$.
See also Example \ref{ex: end cube n as dga} for an alternative computation of $\End(\Cube_n)$.

We can apply the same method in a more general situation. To compute $\Hom(\otimes_i \Cube_{n_i},\otimes_j \Cube_{m_j})$ (with $\sum n_i=\sum m_j=n$), we first observe that both complexes consist of several copies of $\E^{\otimes n}$. Now we replace the $\Hom$ space between two such copies by $\End(\E^{\otimes n})=\C[x_1,\ldots,x_n]\rtimes \C[S_n]$, and write two sets of differentials. The first differential is given by multiplication by $x_i-x_{i+1}$ if $i,i+1$ are in the same block of the partition $n=\sum n_i$, and the second is given by multiplication by $x_i-x_{i+1}$ if $i,i+1$ are in the same block of the partition $n=\sum m_i$. 

Although the differentials do not involve $\C[S_n]$, we still need to keep track of its action. Let $$\Hom_{pol}(\otimes_i \Cube_{n_i},\otimes_j \Cube_{m_j})$$ denote the space of {\it polynomial} maps, that is, the ones induced by the polynomial action on $E^{\otimes n}$. Any endomorphism of $E^{\otimes n}$ can be uniquely written as $f=\sum_{\sigma\in S_n} f_{\sigma}\sigma$ for some polynomials $f_{\sigma}$. Similarly, any morphism $\Hom(\otimes_i \Cube_{n_i},\otimes_j \Cube_{m_j})$ can be uniquely written as $f=\sum_{\sigma\in S_n} f_{\sigma}\sigma$ where $f_{\sigma}$ are polynomial chain maps. 
The space of such $f_{\sigma}$ is isomorphic to $\Hom_{pol}(\sigma(\otimes_i \Cube_{n_i}),\otimes_j \Cube_{m_j})$.
To sum up, to describe all morphisms between products of cubes it is sufficient to describe the polynomial morphisms 
between products of cubes where the variables in one product are possibly relabeled. 
Note that before we did not have this problem since $S_n$ preserves $\Cube_n$.
 
After relabeling, we get two set partitions $\Pi$ and $\Pi'$ with $r$ blocks of size $n_i$ and $s$ blocks of size $m_j$ respectively. 
We will refer to the products of cubes as to $\Cube_{\Pi}$ and $\Cube_{\Pi'}$. 
Let $\Pi''$ be the finest set partition which is a coarsening of both $\Pi$ and $\Pi'$. If $\Pi''$ has more than one block then
$\Hom_{pol}(\Cube_{\Pi},\Cube_{\Pi'})$ factors over the blocks of $\Pi''$  and we can proceed by induction.

From now on we will assume that $\Pi''=\{1,\ldots,n\}$. Let us compute $\Hom_{pol}(\Cube_{\Pi},\Cube_{\Pi'})$ using the spectral sequence as above. 
 After applying the first differential we get a polynomial algebra with one variable per block in $\Pi$. After applying the second differential we identify all these variables and obtain an exterior algebra with generators $\epsilon_{ij}$ for all $i,j$ such that $i,j$ are in the same block in both partitions $\Pi,\Pi'$. 
 
We can describe all these chain maps and their gradings more explicitly. Recall that 
 $$
 \Cube_{a+b}\simeq \mathrm{Cone}(q \Cube_{a}\otimes \Cube_{b}\xrightarrow{x_a-x_{a+1}}q^{-1}\uwave{\Cube_a\otimes \Cube_{b}}),
 $$
so there are natural chain maps 
$$
q^{-1}\Cube_a\otimes \Cube_b\to \Cube_{a+b},\ \Cube_{a+b}\to q\Cube_a\otimes \Cube_b[1].
$$
By combining these, we get maps
$$
q^{1-r}\Cube_{\Pi}\to \Cube_n,\ \Cube_n\to q^{s-1}\Cube_{\Pi'}[s-1].
$$
Every polynomial morphism from $\Cube_{\Pi}$ to $\Cube_{\Pi'}$ can be obtained as a composition of these merge and split maps with a polynomial endomorphism in $\End_{pol}(\Cube_n)=\bV^{\bullet}(U_n)\otimes \C[x].$
In particular, the identity on $\Cube_n$ induces a chain map of $q$-degree $s+r-2$ and homological degree $s-1$. 
The odd variable $\epsilon_{ij}$ can be identified with $\epsilon_i+\ldots+\epsilon_{j-1}$ in $U_n$  dual to 
$$
x_i-x_j=(x_i-x_{i+1})+\ldots+(x_{j-1}-x_j).
$$
which acts on $\Cube_n$. Note that if $i$ and $j$ are not in the same block for $\Pi$ or $\Pi'$ then $\epsilon_{ij}$ acts by 0. For $i,j,k$ in the same block for both partitions $\Pi,\Pi'$ the actions of
 $\epsilon_{ij}$, $\epsilon_{jk}$ and $\epsilon_{ik}$ satisfy an obvious linear relation.
\end{proof}

\begin{cor} $\dim_q\Hom(\Cube_n^\lambda, \Cube_n^\mu) \in \delta_{\lambda, \mu}+q^2\N[q]$.
\end{cor}

\begin{exa}
\label{ex: end cube 2 as dga}
Let us describe the endomorphisms of $\Cube_2=[qE^{\otimes 2}\xrightarrow{x_1-x_2}q^{-1}E^{\otimes 2}].$
In homological degree zero we have $\C[x_1,x_2]\ltimes \C[S_2]$. In homological degree $-1$ we have a chain map $\epsilon$  of 
$q$-degree $-2$ which sends the first copy of $E^{\otimes 2}$ to the second one, and the right copy to zero:
\begin{center}
\begin{tikzcd}
qE^{\otimes 2}\arrow[r]\arrow[dr,"\epsilon"] & q^{-1}E^{\otimes 2}\\
qE^{\otimes 2}\arrow[r]  & q^{-1}E^{\otimes 2}
\end{tikzcd}
\end{center}
Note that in this case the projection to the first copy of $E^{\otimes 2}$ yields the split map 
$\Cube_2\to qE^2[1]$ while the inclusion of the second copy yields the merge $q^{-1}E^2\to \Cube_2$.
The composition of split and merge coincides with $\epsilon$. There is also another map $h$ of homological degree one:
\begin{center}
\begin{tikzcd}
qE^{\otimes 2}\arrow[r] & q^{-1}E^{\otimes 2} \arrow[dl,"h"] \\
qE^{\otimes 2}\arrow[r]  & q^{-1}E^{\otimes 2}
\end{tikzcd}
\end{center}
This is not a chain map, but $[d,h]=x_1-x_2$.  Similarly, $[d,h\epsilon]=(x_1-x_2)\epsilon=d$. 
So the endomorphism ring of $\Cube_2$ in the homotopy category is isomorphic to
$$
\C[x_1,x_2]\otimes \bV(\epsilon)\rtimes \C[S_2]/(x_1-x_2)\cong\C[x]\otimes \bV(\epsilon)\rtimes \C[S_2].
$$
\end{exa}

\begin{exa}
\label{ex: end cube n as dga}
Similarly to Example \ref{ex: end cube 2 as dga}, for $\Cube_n$ we have chain maps $\epsilon_1,\ldots,\epsilon_{n-1}$ and homotopies $h_1,\ldots,h_{n-1}$,
and $[d,h_i]=x_i-x_{i+1}$, so $\End(\Cube_n)\cong\C[x]\otimes \bV(\epsilon_1,\ldots,\epsilon_{n-1})\rtimes \C[S_n].$
As above, $\epsilon_i$ span a copy of the reflection representation $U_n$.
\end{exa}

\begin{exa}
Let us illustrate the difference between polynomial morphisms (which were discussed in the proof of Theorem \ref{thm:endcube}) and all morphisms. For example, let us compute $$\Hom_{pol}(\Cube_1\otimes \Cube_2, \Cube_2\otimes \Cube_1).$$  
We have the following diagram:
$$
\begin{tikzcd}
qE^{\otimes 3}\arrow{r}{x_2-x_3}\arrow[near start]{dr}{\epsilon}\arrow{d}{\alpha} & q^{-1}E^{\otimes 3}\arrow[near start,swap]{dl}{h}\arrow{d}{\beta} \\
qE^{\otimes 3}\arrow{r}{x_1-x_2}  & q^{-1}E^{\otimes 3}
\end{tikzcd}
$$
Now
$[d,\alpha]=(x_1-x_2)\epsilon,\ [d,\beta]=-(x_2-x_3)\epsilon,\ [d,h]=(x_1-x_2)\beta-(x_2-x_3)\alpha,\ [d,\epsilon]=0$.
The homology of $(\C[x_1,x_2,x_3]\langle\alpha,\beta,h,\epsilon\rangle,[d,-])$ is isomorphic to 
$$
\frac{\C[x_1,x_2,x_3]}{(x_1-x_2,x_2-x_3)}\langle \epsilon\rangle\cong\C[x]\langle \epsilon\rangle
$$
Here $\epsilon$ has $q$-degree $-2$ and homological degree $-1$.

Alternatively, $\epsilon$ can be obtained as a composition of the split and merge maps:
$$
\begin{tikzcd}[row sep=large, column sep=huge]
 & E^{\otimes 3}\arrow{r}{x_2-x_3} \arrow{d} & q^{-2}E^{\otimes 3} \arrow{d}  & q^{-1}\Cube_1\otimes \Cube_2\arrow{d}\\
 q^2E^{\otimes 3}\arrow{r}{(x_1-x_2,x_2-x_3)}\arrow{d} & E^{\otimes 3}\oplus E^{\otimes 3}\arrow{r}{(x_2-x_3,x_1-x_2)} \arrow{d}& q^{-2}E^{\otimes 3} & \Cube_3\arrow{d} \\
q^2E^{\otimes 3}\arrow{r}{x_1-x_2}  & E^{\otimes 3} & & q\Cube_2\otimes\Cube_1[1]
\end{tikzcd}
$$
Finally, observe that the transposition $(1\ 3)\in S_3$ yields an obvious degree zero isomorphism between 
$\Cube_1\otimes \Cube_2$ and $\Cube_2\otimes \Cube_1$. This isomorphism is not polynomial, in fact, the above computation shows that there are no polynomial morphisms of degree zero. 
\end{exa}

\begin{exa}
Let us use the description of $\End(\Cube_2)$ to describe $\End(C^+_2)$. Recall that $C^+_2[1]\simeq(\Cube_2)^{S_2}$.
The maps $\epsilon$ and $h$ are not $S_2$--invariant and vanish when symmetrized. However, $\epsilon(x_1-x_2)$ and 
$h(x_1-x_2)$ are $S_2$--invariant. Since $[d,h(x_1-x_2)]\cong(x_1-x_2)^2$, we get 
$$
\End(C_2^+)\cong\C[x_1,x_2]^{S_2}\otimes \bV(\epsilon(x_1-x_2))/(x_1-x_2)^2\cong\C[x]\otimes \bV(\xi),
$$
where $\xi=\epsilon(x_1-x_2)$ and $x=x_1+x_2$. Note that $\xi$ has $q$-degree zero and homological degree $-1$. 
\end{exa}

We can now describe $\End(C^+_n)$ in a similar fashion. 

\begin{thm}
We have $\End(C^+_n)\cong\bV(\xi_1,\ldots,\xi_{n-1})\otimes \C[x]$,
where $\xi_i$ have $q$-degree $2i-2$ and homological degree $-1$. 
\end{thm}
\begin{proof}
Observe that the $h_i$ and $\epsilon_i$ from Example~\ref{ex: end cube n as dga} both span copies of the reflection representation $U_n$ under the action of $S_n$. Therefore by Lemma \ref{lem: solomon} $S_n$--invariant chain endomorphisms of $\Cube_n$ are given by
$$
(\bV(\epsilon_1,\ldots,\epsilon_{n-1})\otimes \C[x_1,\ldots,x_n])^{S_n}\cong\bV(\xi_1,\ldots,\xi_{n-1})\otimes \C[x_1,\ldots,x_n]^{S_n},
$$
where 
$$
\xi_i=\sum_{j}\epsilon_j\frac{\partial}{\partial (x_j-x_{j+1})}p_{i+1}.
$$
Note that $p_{i+1}$ has $q$-degree $2i+2$, so its partial derivatives  
have degree $2i$ and hence $\xi_i$ have $q$-degree $2i-2$ and homological degree $-1$. 
Similarly, the $S_n$--invariant homotopies are built out of $h_i$ so that
$$
(\bV(h_1,\ldots,h_{n-1})\otimes \C[x_1,\ldots,x_n])^{S_n}\cong\bV(H_1,\ldots,H_{n-1})\otimes \C[x_1,\ldots,x_n])^{S_n}.
$$
By construction, $[d,h_i]=(x_i-x_{i+1})$, so $[d,H_i]=p_{i+1}(x_1,\ldots,x_n).$ Therefore we have
$$
\End(C^+_n)=\bV(\xi_1,\ldots,\xi_{n-1})\otimes \C[x_1,\ldots,x_n]^{S_n}/(p_2,\ldots,p_n)\cong\bV(\xi_1,\ldots,\xi_{n-1})\otimes \C[x].\vspace{-.5cm}
$$
\end{proof}

\begin{rem}
It is easy to see that the split and merge maps between $\Cube_n$ induce similar split and merge maps between $C^+_n$.
Since $\epsilon_i$ can be obtained as a composition $\Cube_n\to q\Cube_{i}\otimes \Cube_{n-i}[1]\to q^2\Cube_n[1]$,
the seemingly mysterious endomorphisms $\xi_k$ can be obtained as sums over all $i$ of compositions 
$$
\Cube_n\to q\Cube_{i}\otimes \Cube_{n-i}[1]\xrightarrow{\phi_{i,k}(x)} q^{1-2k}\Cube_{i}\otimes \Cube_{n-i}[1]\to q^{2-2k}\Cube_n[1]
$$
for some explicit polynomials $\phi_{i,k}(x)$ of degree $k$. 

It is likely that all morphisms between various tensor products of $C^+_n$ are generated by splits, merges and the action of polynomials. It would be interesting to describe all relations between these morphisms, categorifying Turaev's description of the skein of the annulus (Theorem \ref{thm: lambda}). We plan to pursue this in a future work.
\end{rem}

\begin{exa}
Let us describe the maps from $C^-_2=[q S^2\E\to q^{-1}\uwave{\bV^2\E}]$ to $q C^-_1\otimes C^-_1[1]=[q \E\otimes \E \to \uwave{0}]$. Since $C^-_2$ is a summand in $\Cube_2$, every such map factors through $\Cube_2$. Thus we have a map
\begin{center}
\begin{tikzcd}
q S^2 \E\arrow{r} \arrow{d} & q^{-1}\uwave{\bV^2\E}\arrow{d}\\
q\E\otimes \E\arrow{r} \arrow{d}& q^{-1}\uwave{\E\otimes \E}\\
  q \E\otimes \E&\\
\end{tikzcd}
\end{center}
Now we have $\Hom(\Cube_2,\Cube_1^2)=\C[x]\otimes \C[S_2]$. But $C^-_2$ is the antisymmetric component of $\Cube_2$, so we get $\Hom(C^-_2, q C^-_1\otimes C^-_1[1])\cong \C[x]$. A generator of minimal degree is given by the twisted band map from Remark~\ref{rem:twistband}.
\end{exa}

Let $\e_{\lambda}\in S_n$ denote our chosen Young symmetrizer in $\C[S_n]$ of shape $\lambda$. As before, we denote by $\Cube_n^\lambda=\e_{\lambda} \Cube_n$ the direct summand of $\Cube_n$ cut out by the action of $\e_{\lambda}$. Then we have the following corollary of Theorem~\ref{thm:endcube}.

\begin{rem} The category $\Komh(\Proph)$ has a t-structure, whose heart is given by the complexes whose chain groups are $q$-shifted by twice the homological degree. $\Cube_n$ as well as all $\Cube_n^\lambda$ are shifted perverse.
\end{rem}

\subsection{Other Coxeter braids}
\label{sec:coxlift}

To categorify the formula from Theorem~\ref{thm:gencoxdecat}, we would like to give a more categorical perspective on ribbon skew Schur functions, following Solomon \cite{Solomon}. Given a binary sequence $\epsilon$ of length $n$, we can define two parabolic subgroups $W_{\epsilon},W'_{\epsilon}$ of $S_n$ generated by simple reflections with positive (resp. negative) signs.
Let $s_{\epsilon}$ and $\overline{s}_{\epsilon}$ denote the symmetrizer for $W_{\epsilon}$ and antisymmetrizer for $W'_{\epsilon}$.

\begin{thm}[\cite{Solomon}]
\label{thm:solomon}
The group algebra $\C[S_n]$ can be presented as a direct sum of left ideals:
\begin{equation}
\label{eqn: solomon decomposition}
\C[S_n]=\bigoplus_{\epsilon\in \{\pm 1\}^n}\C[S_n]s_{\epsilon}\overline{s}_{\epsilon}
\end{equation}

Furthermore,  
the character of the $S_n$--representation $\C[S_n]s_{\epsilon}\overline{s}_{\epsilon}$ equals the ribbon skew Schur function $\Psi(a)$ for the composition $a$ corresponding to $\epsilon$.
\end{thm}

We denote by $p_{\epsilon}\in \C[S_n]$ the idempotent projecting to $\C[S_n]s_{\epsilon}\overline{s}_{\epsilon}$. Now we are ready to describe the annular invariants of the Coxeter braids $\sigma_\epsilon = \sigma_1^{\epsilon_1}\cdots\sigma_{n-1}^{\epsilon_{n-1}} $. 

\begin{thm}
\label{thm:allCoxeter}
The annular complex $C_{\epsilon}$ of the Coxeter braid $\sigma_\epsilon$ is determined by
$$
C_{\epsilon}[|\epsilon|_+]\simeq p_{\epsilon}\Cube_n.
$$
\end{thm}

\begin{proof}

We induct on the length of $\epsilon$ and the number of minus signs in $\epsilon$. Suppose there is just one minus sign in the $a$-th place. Then the skein triangle gives us a homotopy equivalence:
\begin{equation}
\label{eq: skein one minus}
\mathrm{Cone}(C^+_n \to C_{\epsilon})[n-2] \simeq \mathrm{Cone}(q C^+_a\otimes C^+_b \xrightarrow{x-y} q^{-1}C^+_a\otimes C^+_b )[(a-1)+(b-1)] 
\end{equation}
Now we use that the right-hand side is a direct summand in 
\[\mathrm{Cone}(q \Cube_a \otimes \Cube_b \xrightarrow{x-y} q^{-1} \Cube_a \otimes \Cube_b )\cong \Cube_n,\]
cut out by the idempotent $p_a\otimes p_b\in \C[S_{a-1}\times S_{b-1}] \subset \C[S_{n-1}]$.
We also know that $C^+_n[n-1]$ is a direct summand in $\Cube_n$ cut out by the idempotent $p_n\in \C[S_{n-1}]$. The projection onto this summand $\Cube_n\to C^+_n[n-1]$ factors through the right hand side of \eqref{eq: skein one minus}, so the skein triangle \eqref{eq: skein one minus} splits and $C^+_n[n-1]$ is a direct summand in the right hand side. Hence $C_{\epsilon}[n-2]$ is also a direct summand in $\Cube_n$ defined by the difference of the two idempotents $p_n-p_a\otimes p_b = p_\epsilon$. 

A similar argument works for the induction step. Here we use the skein triangle to get:
\[\mathrm{Cone}(C_{\epsilon'} \to C_{\epsilon})[|\epsilon|_+] \simeq \mathrm{Cone}(q C_{\alpha}\otimes C_{\beta} \xrightarrow{x-y} q^{-1}C_{\alpha}\otimes C^+_{\beta} )[|\epsilon|_+] \] 
Then we use the induction hypothesis to find $C_{\alpha}\otimes C_{\beta}[|\epsilon|_+]$ as a direct summand in $\Cube_{a}\otimes\Cube_{b}$, such that the inclusion intertwines the operators $x-y$. If follows that the cones are direct summands of $\Cube_n$, and so is $C_{\epsilon'}[|\epsilon'|_+]$ and hence $C_{\epsilon}[|\epsilon|_+]$. 

It remains to check that the projectors for all these summands agree with $p_{\epsilon}$. Indeed, they can be computed recursively by successively subtracting induced smaller projectors from the bigger ones (this categorifies \eqref{phi recursion}). On the other hand, $p_{\epsilon}$ satisfy the same recursion by \cite[Theorem 3]{Solomon}.
\end{proof}

We are now in a position to prove a conjecture of Hunt--Keese--Licata--Morrison about the annular Khovanov homology of Coxeter braids and the spectral sequence to planar Khovanov homology.

For this, we will use that the annular Khovanov homology can be computed via annular evaluation along the functor $\AKhR(V_1,0)$ where $V_1$ is considered as the vector representation of $\slnn{2}$ with graded dimension $q z+ q^{-1}z^{-1}$ and $z$ encodes the weight space grading. The planar Khovanov homology can similarly be obtained via $\AKhR(V_1,e)$, with $e\in \End(V_1)$ provided by the $\slnn{2}$ action. The spectral sequence from annular to planar Khovanov homology arises by filtering $\AKhR(V_1,e)$ along the weight space grading.

\begin{thm}[{\cite[Conjecture 4.1]{HKLM}}]
The generators of the annular Khovanov homology of $C_\epsilon$ that survive to planar Khovanov homology have tri-degree $ (t q^2 z)^{n-1-2|\epsilon|_+}(q z+q^{-1}z^{-1})$. 
\end{thm}
\begin{proof}
We first use that the symmetric function $(-1)^{|\epsilon|_+}\Psi(a)[X(q^{-1}-q)]/(q^{-1}-q)$ evaluates on the variables $(q,q^{-1},0,\dots)$ to the Jones polynomial of $C_\epsilon$, namely $(-q^2)^{n-1-2|\epsilon|_+}(q+q^{-1})$. Framing considerations imply that the surviving generators are supported in homological degree $t^{n-1-2|\epsilon|_+}$. Since the annular complex $C_\epsilon$ becomes perverse after a shift by $q^{n-1} t^{|\epsilon|_+} $, this implies that the surviving generators live in the chain group with shift $(t q)^{n-1-2|\epsilon|_+}$. We also have the following bigraded dimension of the $\slnn{2}$-evaluations $\dim_{q,z}(\Sch^{n-i,i}(V_1))=\dim_{q,z}(V_{n-2i})=h_{n-2i}(q z, q^{-1} z^{-1})$. In particular, the $\slnn{2}$-weight shift inside these Schur functor evaluations is always equal the internal shift in $q$-grading. Thus, the $z=1$ and $t=-1$ specialization $(-q^2)^{n-1-2|\epsilon|_+}(q+q^{-1})$ of the desired formula and the knowledge of the chain group shift $(t q)^{n-1-2|\epsilon|_+}$ determine the internal $q$-shift uniquely as $q^{n-1-2|\epsilon|_+}$, and thus also $z^{n-1-2|\epsilon|_+}$. 
\end{proof}

\section{Annuli in tangle diagrams}
\label{sec:wrapped}
In this section we study applications of annular evaluation to
Khovanov--Rozansky invariants of tangles which contain a cabling of a framed
unknot as a sublink. This includes tangles obtained by wrapping an annular link
around a tangle as in \eqref{eqn:wrap}.

\subsection{A symmetric group action on cables}

The following theorem is due to Grigsby--Wehrli--Licata in the context of
Khovanov homology \cite{GLW}. The version here applies to all sufficiently
functorial Khovanov--Rozansky link homologies of type A.

\begin{thm} 
\label{thm:circlebraiding}
Let $T$ be a link or a tangle, which has $n$ parallel closed $1$-colored
components. Then $T$ carries an action of $Br_n$ by endo-cobordisms that braid
these parallel components around each other. Let $\KhR$ denote a
Khovanov--Rozansky-type invariant, which is functorial under such
cobordisms\footnote{E.g. for the triply-graded homology, we require that $L$ is
represented as a partial braid closure.}. Then the induced action of $Br_n$ on
$\KhR(T)$ factors through $S_n$.
\end{thm}
\begin{proof}
It suffices to prove that the braiding is symmetric on two parallel components. We have already seen this in the proof of Theorem~\ref{annular functor braided} for the case when $L$ has the two components as a disjoint split factor. Now, we consider the general case, which can be modelled as follows. 

\begin{equation}
\label{braidingmovie2}
\begin{tikzpicture}[anchorbase, scale=.5]
\draw [red, thick, directed=.6, directed=.55] (2.5,-.5) to (2.5,1.5);
\draw [red, thick, directed=.6, directed=.55] (4.5,-.5) to (4.5,1.5);
\draw  (4.5,1.5) to (2.5,1.5);
\draw  (4.5,-.5) to (2.5,-.5);
\draw [very thick, ->] (4.5,0)to (4,0) (3,0) to (2.5,0);
\draw [very thick, ->] (4.5,1)to (4,1) (3,1) to (2.5,1);
\draw (3,-.25) to (4,-.25) to (4,1.25) to (3,1.25) to (3,-.25);
\end{tikzpicture}
\xrightarrow{RII}
\begin{tikzpicture}[anchorbase, scale=.5]
\draw [red, thick, directed=.6, directed=.55] (0,-.5) to (0,1.5);
\draw [red, thick, directed=.6, directed=.55] (4.5,-.5) to (4.5,1.5);
\draw  (4.5,1.5) to (0,1.5);
\draw  (4.5,-.5) to (0,-.5);
\draw [very thick, ->] (4.5,0)to (4,0) (3,0) to [out=180,in=0] (1.5,1) to [out=180,in=1]
(0,0);
 \draw [white,line width=.15cm] (4.5,1)to (4,1) (3,1) to [out=180,in=0] (1.5,0) to [out=180,in=0] 
(0,1);
\draw [very thick, ->] (4.5,1)to (4,1) (3,1) to [out=180,in=0] (1.5,0) to [out=180,in=0]  (0,1);
\draw (3,-.25) to (4,-.25) to (4,1.25) to (3,1.25) to (3,-.25);
\draw (.75,1) node {\small $1$};
\draw (2.25,1) node {\small $2$};
\end{tikzpicture}
\xrightarrow{\text{isotopy}}
\begin{tikzpicture}[anchorbase, scale=.5]
\draw [red, thick, directed=.6, directed=.55] (0,-.5) to (0,1.5);
\draw [red, thick, directed=.6, directed=.55] (4.5,-.5) to (4.5,1.5);
\draw  (4.5,1.5) to (0,1.5);
\draw  (4.5,-.5) to (0,-.5);
\draw [very thick, ->] (4.5,0)to [out=180,in=0](3,1) (2,1) to (1.5,1) to [out=180,in=0] (0,0);
 \draw [white,line width=.15cm] (4.5,1)to [out=180,in=0] (3,0) (2,0)to (1.5,0) to [out=180,in=0](0,1);
\draw [very thick, ->] (4.5,1)to [out=180,in=0] (3,0) (2,0)to (1.5,0) to [out=180,in=0](0,1);
\draw (.75,1) node {\small $2$};
\draw (3.75,1) node {\small $1$};
\draw (2,-.25) to (3,-.25) to (3,1.25) to (2,1.25) to (2,-.25);
\end{tikzpicture}
\xrightarrow{\text{RIIIs}}
\begin{tikzpicture}[anchorbase, scale=.5]
\draw [red, thick, directed=.6, directed=.55] (0,-.5) to (0,1.5);
\draw [red, thick, directed=.6, directed=.55] (4.5,-.5) to (4.5,1.5);
\draw  (4.5,1.5) to (0,1.5);
\draw  (4.5,-.5) to (0,-.5);
\draw [very thick, ->] (4.5,1)to (4,1)(3,1) to [out=180,in=0] (1.5,0) to [out=180,in=1]
(0,1);
 \draw [white,line width=.15cm] (4.5,0)to (4,0) (3,0) to [out=180,in=0] (1.5,1) to [out=180,in=0] 
(0,0);
\draw [very thick, ->] (4.5,0)to (4,0) (3,0) to [out=180,in=0] (1.5,1) to [out=180,in=0]  (0,0);
\draw (3,-.25) to (4,-.25) to (4,1.25) to (3,1.25) to (3,-.25);
\draw (.75,1) node {\small $2$};
\draw (2.25,1) node {\small $1$};
\end{tikzpicture}
\xrightarrow{RII}
\begin{tikzpicture}[anchorbase, scale=.5]
\draw [red, thick, directed=.6, directed=.55] (2.5,-.5) to (2.5,1.5);
\draw [red, thick, directed=.6, directed=.55] (4.5,-.5) to (4.5,1.5);
\draw  (4.5,1.5) to (2.5,1.5);
\draw  (4.5,-.5) to (2.5,-.5);
\draw (3,-.25) to (4,-.25) to (4,1.25) to (3,1.25) to (3,-.25);
\draw [very thick, ->] (4.5,0)to (4,0) (3,0) to (2.5,0);
\draw [very thick, ->] (4.5,1)to (4,1) (3,1) to (2.5,1);
\end{tikzpicture}
\end{equation}
Here, $T$ is compressed into the small box shown, except for the two parallel components in question (if the tangle is not a link, then some additional strands might connect this box to the boundary). The braiding of the two circles starts with a Reidemeister II move, followed by isotoping the $1$-labeled crossing all the way through the rest of $T$, and then eliminating both crossings again by an inverse Reidemeister II move. Contrary to the case treated in Theorem~\ref{annular functor braided}, we do not intend to compute this chain map $\sigma$ explicitly. We only need to show that it is equivalent to its inverse $\sigma^{-1}$, which has a movie description as in \eqref{braidingmovie2}, except that the bottom strand passes over the top strand first, and it is a negative crossing instead of a positive crossing that slides all the way through $T$. In the absence of other components, we have seen that these chain maps are plainly equal, since the different Reidemeister II chain maps uses in these variants agree (up to to cancelling signs). 

In the present case, we additionally have to take into account moving the crossing labeled $1$ through the box, i.e. through the rest of $T$. There are three key observations which allow to compare the contributions of this process to $\sigma$ and $\sigma^{-1}$. 

First, isotoping the crossing through the rest of $T$ is realized as a sequence of braid-like Reidemeister III moves. A braid-like Reidemeister III move is one in which the relevant local tangle $6$-ended tangle has the following sequence of boundary orientations up to cyclic reordering: out-out-out-in-in-in. In contrast, a star-like Reidemeister III tangle would have an alternating sequence of boundary orientations out-in-out-in-out-in. 

Second, the intermediate chain complexes in \eqref{braidingmovie2} can be seen as total complexes of double complexes, with a horizontal differential contributed by the crossings in $T$, and a vertical differential contributed by the extra crossings created by the initial Reidemeister II move. Note that the initial and the final chain complex in this sequence are supported in the single vertical degree zero. 

Third, the chain maps associated to the braid-like Reidemeister III moves are filtered with respect to the vertical degree. This means that these chain maps are sums of components that preserve the vertical degree, and components which, at most, increase the vertical degree, but never decrease it. Moreover, in a pair of Reidemeister III moves, which differ only in the sign of the $1$-labeled crossing which is pushed under (or over) another strand in $T$, the filtration-preserving components agree.
For $1$-colored strands, this is well-known to experts and can be read off from the explicit descriptions of Reidemeister III chain maps for Rouquier complexes in \cite{EK}. The general case follows via the strategy of exploding strands of higher color into $1$-colored strands before sliding the crossing, see e.g. \cite[Section 14.1]{Wu1}. 

The chain maps obtained by isotoping a positive or a negative crossing through the rest of $T$ are both filtered, and their filtration-preserving components agree. Finally, $\sigma$ and $\sigma^{-1}$ are obtained from these chain maps by pre- and postcomposing with Reidemeister II chain maps. Since the latter have non-zero components only in vertical degree zero, these composite only depend on the filtration-preserving parts of the intermediate Reidemeister III chain maps. As noted above, these agree.
\end{proof}

In particular, Theorem~\ref{thm:circlebraiding} holds for the the $\glnn{N}$ Khovanov--Rozansky invariants $\Hgenuf{T}$ and $\Hgen{T}$ valued in $\Komh(\Nfoam)$.  
In the following, we write $T=T(E^{\otimes{n}})$ for tangles as in the theorem. 

\begin{cor} 
 The Schur-colored invariants $\Hgenuf{T(\Sch^\lambda(E))}$ and $\Hgen{T(\Sch^\lambda(E))}$ are well defined in $\Komh(\Kar(\Nfoam))$. \end{cor}

\begin{proof}
By Theorem \ref{th:homotopy karoubian}, we have that  $\Komh(\Kar(\Nfoam))$ is Karoubian. By Theorem~\ref{thm: ETW functoriality} there is a braid group action on $\Hgen{T(E^{\otimes n})}$, which factors through the symmetric group. Hence by Proposition \ref{prop: homotopy schur} the Schur functors are well defined up to homotopy equivalence. 
\end{proof}

\begin{cor} For a tangle as in the theorem, $\Hgenuf{T(E^{\otimes n})}$ and $\Hgen{T(E^{\otimes n})}$ have actions of $\C[x_1,\dots, x_n]\rtimes \C[S_n]$.
\end{cor}
\begin{proof} The $\C[S_n]$ part is obtained by linearising the symmetric group action from the theorem. In the polynomial part, $x_i$ acts by a dot on the $i$-th component of the cable. The proof of the theorem and the fact that dots slide through crossings up to homotopy implies that stated compatibility.
\end{proof}

We can summarize the two corollaries as follows.

\begin{cor} Each tangle as in Theorem~\ref{thm:circlebraiding} provides an additive functor from $\Proph$ to $\Komh(\Kar(\Nfoam))$.
\end{cor}

As before, we also get a version of Lemma~\ref{lem:antisymm-def-matchup} in the presence of other strands. For this, let $T(E_n)$ denote the tangle $T$ with a $n$-colored component in place of the $n$ parallel uncolored components. 

\begin{cor} 
\label{cor:antisymm-def-matchuptwo} 
We have isomorphisms $\Hgenuf{T(\bV^n(E))}\cong \Hgenuf{T(E_n)}$ and $\Hgen{T(S^n(E))}\cong \Hgen{T(E_n)}$ in $\Komh(\Kar(\Nfoam))$.
\end{cor} 
\begin{proof} The proof proceeds analogous to the one for Lemma~\ref{lem:antisymm-def-matchup} by identifying the chain map for the $k=2$ anti-symmetrizer on $\Hgen{T(E\otimes E)}$ with the projection onto $\Hgen{T(E_2)}$.
\end{proof}

This implies that cobordism-induced braiding is also symmetric for colored circles, as proved for uncolored circles in Theorem~\ref{thm:circlebraiding}.

\subsection{Annular simplification}

If an annular link $L$ appears as a sublink of a tangle $T$ which is a cabling of a framed unknot, then the associated Khovanov--Rozansky chain complex $\Hgen{T}$ can be simplified to a complex in which the annular link $L$ is replaced by the a complex of $\bV$-colored concentric circles or Schur functors. Here we prove that this induces filtrations and spectral sequences as claimed in Theorem~\ref{th:intro filtered} and Corollary~\ref{th: intro spectral sequence}.

\begin{prop}\label{prop:filtration} Let $L$ denote a link diagram in the thickened annulus, $T$ a tangle diagram with a blackboard-framed unknot component without self-crossings, and $T(L)$ the tangle diagram obtained by cabling this unknot component in $T$ by $L$. 
Then the chain complex $\Hgen{T(L)}$ is isomorphic in $\Komh(\Nfoam)$ to a filtered chain complex $\tilde{C}$, whose associated graded is isomorphic to a formal direct sum of grading shifts of chain complexes of the form $\Hgen{T(CC)}$
where $CC$ denotes the collections of concentric $\bV$-colored circles that appear in the annular simplification of $L$. Moreover, the component of the differential that increases the filtration degree by one is induced by the corresponding annular differential. 
\end{prop}
Interesting examples of tangles $T(L)$ are tangles obtained by wrapping as in \eqref{eqn:wrap} and cabled Hopf links $H(L,L_2)$ as in the introduction.

\begin{proof} We write $C:=\Hgen{T(L)}$. The key idea of the proof is that annular simplification is still possible in the presence of additional strands. Indeed, the annular simplification algorithm of Queffelec--Rose \cite[Proposition 5.1]{QR2} utilizes two types of web isomorphisms, which both continue to hold in these settings: namely certain local isomorphisms (rung combination and square switch) which hold on the nose, and the global rung slide move, which uses fork-slide moves in the presence of additional strands.

The chain complex $C$ can be viewed as a total complex of a tricomplex with one direction (horizontal) corresponding to crossings internal to the annular link $L_1$, the second direction (vertical) to crossings of that annular link and the rest, and the third direction (depth) to crossings purely in the rest. Since the third direction will not play an important role, we will suppress it and consider $C$ as total complex of a bicomplex $C^{*,*}$. The columns $C^{*,i}$ in such bicomplexes are complexes in their own right, which are isomorphic to the invariants of the annular webs appearing in the cube of resolutions of $L_1$, interacting with other additional link and tangle components. 

By annular simplification, each column $C^{*,i}$ is homotopy equivalent to the total complex $\tilde{C^{*,i}}$ of a bicomplex whose columns are of the form $T(CC)$, where $CC$ is a collection of concentric circles. Now we substitute the columns in the bicomplex $C^{*,*}$ by the homotopy equivalent complexes $\tilde{C}^{*,i}$. In doing so, we collapse the two ``horizontal'' directions: the one already present in $C^{*,*}$ and the additional direction in each $\tilde{C}^{*,i}$. Because of the column substitutions, $\tilde{C}^{*,*}$ will typically no longer be a bicomplex. Besides the vertical differential $d_0\colon \tilde{C}^{*,*} \to \tilde{C}^{*+1,*}$ and the horizontal component $d_1\colon \tilde{C}^{*,*}\to \tilde{C}^{*,*+1}$, there are now also higher components $d_k \colon \tilde{C}^{*,*} \to \tilde{C}^{*+1-k,*+k}$. In  Figure~\ref{fig:perturbedbicomplex}, we illustrate the result of a single column substitution.

\begin{figure}[h]
\begin{center}
\comm{
\begin{tikzcd}
 C^{i-1,j-1} \arrow{r}{d_h}\arrow{d} & C^{i-1,j}\arrow{r}{d_h}\arrow{d} & C^{i-1,j+1}\arrow{d}\\
 C^{i,j-1} \arrow{r}{d_h}\arrow{d} & C^{i,j}\arrow{r}{d_h}\arrow{d} & C^{i,j+1}\arrow{d}\\
 C^{i+1,j-1}\arrow{r}{d_h}& C^{i+1,j}\arrow{r}{d_h} & C^{i+1,j+1}\\
\end{tikzcd}\quad}
\begin{tikzcd}
 C^{i-1,j-1} \arrow{r}{f \circ d_h}\arrow{d} & D^{i-1,j}\arrow{r}{d_h\circ g}\arrow{d} & C^{i-1,j+1}\arrow{d}\\
 C^{i,j-1} \arrow{r}{f \circ d_h}\arrow{d}\arrow{urr}{d_h\circ h\circ d_h} & D^{i,j}\arrow{r}{d_h\circ g}\arrow{d} & C^{i,j+1}\arrow{d}\\
 C^{i+1,j-1}\arrow{r}{f \circ d_h} \arrow{urr}{d_h\circ h\circ d_h}& D^{i+1,j}\arrow{r}{d_h\circ g} & C^{i+1,j+1}
\end{tikzcd}
\end{center}
\caption{The result of substituting a single column $(C^{i,*}, d_v)$ in a bicomplex by a homotopy equivalent complex $(D^{i,*},d)$ along chain homotopy equivalences $f$ and $g$ with $g \circ f +  d_v \circ h + h \circ d_v =0$. }
\label{fig:perturbedbicomplex}
\end{figure}

The perturbed bicomplex $\tilde{C^{*,*}}$ still carries the horizontal filtration $\cal{F}_{j} = \bigoplus_{j'\geq j} \tilde{C}^{*,j'}$, whose associated graded is isomorphic to the direct sum of the columns, with differential $d_0$, which we identify with the invariants of collections of colored circles interacting with the remaining strands. The filtration degree one component of the total differential is $d_1$ and its components originate from crossings in the annular link or the resolution of annular webs by concentric circles--- in this sense, it is induced by the annular differential computed by the annular simplification algorithm of Queffelec--Rose. 
\end{proof}

\begin{exa} If we apply Proposition~\ref{prop:filtration} in the case of the Hopf link cable $H(L_1, \emptyset)$, we obtain $\tilde{C}^{*,*}=\tilde{C}^{0,*}$ and the only non-trivial component of the differential is $d_1$.  
\end{exa}

\begin{cor}
\label{cor:schur}
The complex $\tilde{C}$ from Proposition~\ref{prop:filtration}, considered as an object of $\Komh(\Kar(\Nfoam))$, can be decomposed further into a filtered complex $C'$ with associated graded given by Schur-colored unknots interacting with the remaining strands.
\end{cor}
\begin{proof}
This follows from Proposition~\ref{prop:filtration} and Corollary~\ref{cor:antisymm-def-matchuptwo}, which identifies colored circles with tensor products of antisymmetric Schur functors, which we can then decompose further. 
\end{proof}

This completes the proof of Theorem~\ref{th:intro filtered} and implies Corollary~\ref{th: intro spectral sequence}.

\begin{rem} Corollary~\ref{cor:schur} can also be proved directly following the strategy of the proof of Proposition~\ref{prop:filtration}, but with the Queffelec--Rose annular evaluation algorithm replaced by the alternative annular evaluation algorithm outlined in Remark~\ref{rem:annularsimpl}. 
\end{rem}

\subsection{Generalized Hopf links, categorified}

Here we show how the above results can be used to compute Khovanov-Rozansky homologies of generalized Hopf links. 
First, we consider annular links wrapped around a single vertical strand colored by $\bV^k$. We reduce on this vertical strand, so that the corresponding tangle has no non-trivial endomorphisms, and the invariants in question are valued in complexes of graded vector spaces. For the definition of reduced colored Khovanov--Rozansky homologies, we refer to \cite{Wed3}.

\begin{thm}
\label{thm:wedge reduced}
Let $L$ be an annular link diagram and let $T(\bV^i,L)$ be the tangle consisting
of $L$ wrapped around the {\it reduced } vertical strand colored by $\bV^i$.
Consider the following bigraded vector space with an action of $\C[X]$:
$$
\CE_{\bV^i}=q^{N-1}\C[X]/X^{N-i}\oplus t^{-2} q^{2i-3-N}\C[X]/X^i.
$$
Then there is a spectral sequence with the $E_2$ page given by the evaluation of
the annular complex of $L$ at $\CE_{\bV^i}$ and $E_\infty$ page isomorphic to
$\Hgen{T(\bV^i,L)}$.
\end{thm}

Note that we have $\dim_{q,t}\CE_{\bV^i} = q^i [N-i] + t^{-2}q^{i-2-N} [i]$, where $[n]=\frac{q^n-q^{-n}}{q-q^{-1}}$.

\begin{proof}
If $L$ is a single $\bV^j$-colored unknot, then the invariant of $T(\bV^i,L)$
was computed by second author in \cite[Proposition 4.15]{Wed1}. For $j=1$, it
agrees with $\CE_{\bV^i}$ as a bigraded vector space. The action of the dot on
$L$ can be easily computed, and it agrees with the action of $X$ above.

Suppose that now $L$ is an arbitrary annular link.  By Proposition \ref{prop:filtration} the Khovanov-Rozansky complex of 
$T(\bV^i,L)$ is filtered with associated graded given by the evaluation of the annular complex of $L$ at $(\CE_{\bV^i},X)$.
More precisely, the differential splits into two parts: the annular differential $d_{\mathrm{ann}}$ for $L$ and the additional differential
$d_{\mathrm{wrap}}$ responsible for the crossings between the webs in the resolution of $L$ and the vertical strand. We get a spectral sequence by first applying $d_{\mathrm{wrap}}$ and then the induced differential $d_{\mathrm{ann}}^*$. It converges to the homology of the total complex.
By Proposition \ref{prop:filtration} the homology with respect to $d_{\mathrm{wrap}}$ is isomorphic to the evaluation of the annular complex for $L$ 
at $(\CE_{\bV^i},X)$ (with no differential). On the next page of the spectral sequence we compute the homology with respect to $d^*_{\mathrm{ann}}$, which is just the homology of the annular complex for $L$ evaluated at $(\CE_{\bV^i},X)$.
\end{proof}

\begin{cor}
Let $L$ be a $\bV^i$-colored unknot, and $T(\bV^i,\bV^j)=T(\bV^i,L)$ as above. Then
the Khovanov-Rozansky homology of $T(\bV^i,\bV^j)$ is isomorphic to $\bV^j(\CE_{\bV^i})$ as above.
In particular, its graded dimension is
$$
\sum_{k=0}^{i}q^{ij-k(2+N)}t^{-2k}{N-i\brack k}{i\brack j-k}.
$$
This agrees with \cite[Proposition 4.15]{Wed1}.
\end{cor}

We expect that Theorem \ref{thm:wedge reduced} can be generalized to other
projectors, categorifying Lemma \ref{lem: q lambda rho}. Specifically, Elias and
Hogancamp recently constructed \cite{EH} a family of projectors $P_{\lambda}$ in
the homotopy category of Soergel bimodules which categorify the projectors
$p_{\lambda}$ from Section \ref{sec:classical}. These are idempotent complexes
which are bounded from above.  Let $\langle P_{\lambda}\rangle$ be the smallest
triangulated subcategory of the homotopy category containing $P_{\lambda}$.
After specialising to the $\glN$ theory, we expect the following.

\begin{conj}
\label{con: refined S}
Let $L$ be an annular link and let $T(P_{\lambda},L)$ denote the tangle consisting of $L$ wrapped around $P_{\lambda}$.
Then $\Hgen{T(P_{\lambda},L)}$ is an object of the category $\langle P_{\lambda}\rangle$. If $L$ is a single unknot then 
\begin{equation}
\label{eq:refined S}
\Hgen{T(P_{\lambda},L)}\simeq\left[(q^{1-N}(qt)^{-2\lambda_1}+q^{3-N}(qt)^{-2\lambda_2}+\ldots+q^{N-1}(qt)^{-2\lambda_N})P_{\lambda},D\right]
\end{equation}
for some differential $D$. 
\end{conj}

\begin{exa}
If $\lambda=(1^i)$ then \eqref{eq:refined S} can be interpreted as saying that $L$ acts on $P_{\lambda}$  with ``eigenvalue''
$$
q^{1-N}(qt)^{-2}+\ldots+q^{2i-N-1}(qt)^{-2}+q^{2i-N+1}+\ldots+q^{N-1}=q^{i}[N-i]+t^{-2}q^{i-2-N}[i].
$$
This agrees with Theorem \ref{thm:wedge reduced}.
\end{exa}

\begin{rem}
At $t=-1$ equation \eqref{eq:refined S} specializes to Lemma \ref{lem: q lambda rho}.
\end{rem}

\begin{rem}
This conjecture gives a precise categorical context to the ``refined $S$-matrix'' defined by Aganagic and Shakirov, see \cite{AS}. Specifically, they conjecture that (a) the projectors $P_{\lambda}$ in certain sense correspond to Macdonald polynomials $H_{\lambda}(x;q,t)$, and (b) the ``refined Chern-Simons invariant'' of the generalized Hopf link with components labeled by $P_{\lambda}$ and $P_{\mu}$ equals 
$$
H_{\lambda}(q^{1-N},\ldots,q^{N-1})H_{\mu}(q^{1-N}(qt)^{-2\lambda_1},q^{3-N}(qt)^{-2\lambda_2},\ldots,q^{N-1}(qt)^{-2\lambda_N}).
$$
While we are unable to comment on (a) at the moment, we can interpret (b) by cutting the component with $P_{\lambda}$ open. Then the invariant of the corresponding tangle equals
$$
P_{\lambda}H_{\mu}(q^{1-N}(qt)^{-2\lambda_1},q^{3-N}(qt)^{-2\lambda_2},\ldots,q^{N-1}(qt)^{-2\lambda_N}).
$$
Since this is linear in $H_{\mu}$, we can instead consider a tangle where one component is colored by $P_{\lambda}$ and the other is a closed circle colored by an arbitrary symmetric function $f$. The ``refined Chern-Simons invariant'' of this tangle equals
$$
P_{\lambda}f(q^{1-N}(qt)^{-2\lambda_1},q^{3-N}(qt)^{-2\lambda_2},\ldots,q^{N-1}(qt)^{-2\lambda_N}),
$$
which agrees with a certain decategorification of \eqref{eq:refined S}.
\end{rem}

We would like to comment on possible (but yet mostly conjectural) connections between the results of this paper and 
the work of the first author, Negu\cb{t} and Rasmussen \cite{GNR}, as well as the series of papers of Oblomkov and Rozansky \cite{OR1,OR2,OR3}. One of the main conjectures of \cite{GNR}
assumes the existence of a monoidal functor 
$$
\iota^*: D^b\Coh(\Hilb^n(\C^2))\to \Komh(\SBim_n),
$$
where $\Hilb^n(\C^2)$ is the Hilbert scheme of $n$ points on the plane and $D^b\Coh$ denotes the derived category of coherent sheaves. On the Hilbert scheme of points we have two important sheaves: $\CT$ is the tautological bundle of rank $n$ while $\CI$ is the tautological ideal sheaf (of infinite rank). The fibers of $\CT$ and $\CI$ over a given ideal $I\subset \C[x,y]$ are equal to $\C[x,y]/I$ and to $I$, respectively. Both $\CT$ and $\CI$ enjoy the action of two commuting endomorphisms $X$ and $Y$. 

\begin{conj}
\label{conj:gnr pullback}
The $\glN$ invariant of a single unknot wrapped around $n$ vertical strands is isomorphic to the $\glN$ reduction of 
the object
$
\iota^*(\CI/(Y,X^N)\CI).
$
\end{conj}

As explained in \cite{GNR}, the projectors $P_{\lambda}$ should correspond to
the fixed points of the torus action on $\Hilb^n(\C^2)$, that is, to the
monomial ideals $I_{\lambda^t}$\footnote{Note that the diagram $\lambda$ should
be transposed.}. At such a monomial ideal, the fiber of $\CI/(Y,X^N)\CI$ has a
bigraded character which agrees with \eqref{eq:refined S}. This means that
Conjectures \ref{con: refined S} and \ref{conj:gnr pullback} are compatible with
each other. See also \cite{Nakajima} for more detailed relation between the
refined $S$-matrix and the geometry of the Hilbert scheme of points.

Finally, we would like to comment on the relation between this work and
\cite{Elias}. There, Elias constructed a family of objects $\CX_{\lambda}$
(labeled by Young diagrams $\lambda$) in the Drinfeld center of the category of
(extended) affine Soergel bimodules. It is expected that $\CX_{\lambda}$ descend
to the homotopy category of Soergel bimodules, and their images are  filtered by
the products of Jucys-Murphy braids $\CL_i$ according to the weight
decomposition of the irreducible representation $V_{\lambda}$ of $\glN$. For
example, for $\lambda=\square$ the complex $\CX_{\square}$ is filtered by
$\CL_i$, each with multiplicity one. 

We expect $\CX_{\lambda}$ to be closely related, but not identical to our
annular links wrapped around vertical strands. In the notations of Conjecture
\ref{conj:gnr pullback} we expect
$$
\CX_{\lambda}=\iota^*(\Sch^{\lambda}(\CT)),
$$
in particular, $\CX_{\square}=\iota^*(\CT)$. This relation is expected to categorify Lemma \ref{power sum in center}.

\subsection{A note on wrapping}

The initial motivation for this paper was to categorify the wrapping operation
\eqref{eq:encircle}. In HOMFLY-PT skein theory, the action of encircling braids
by positive annular links descends to an action of the cocenter of all Hecke
algebras $H_m$ of type A on the center of $H_n\otimes \kb$. On the topological
level, and with a view towards categorification, the encircling operation can be
described as a functor from $\Alinkp$, the 1-cocenter (horizontal trace, see Section \ref{sec: traces}) of the
braided monoidal 2-category of braids and their cobordisms, to the centralizer
$\mathcal{Z}(\braid_n)$ of the 2-category of braids (and their cobordisms) on $n$
coherently oriented strands inside the 2-category of tangles $\tang_n$ with the
same boundary data.

In this paper, we have described and studied the universal target for the
currently available Khovanov--Rozansky functors for positive annular links,
namely the category $\Komh(\Proph)$. The categorified analog of the Hecke
algebra $H_n$ is the homotopy category of Soergel bimodules $\Komh(\SBim_n)$. A
first approximation to what a categorification of the wrapping operation could
be is given in Figure~\ref{fig:bigpicture}, ignoring the second column.

\begin{figure}[h]\[
\xymatrix{
\Alinkp  \ar[r]^(0.35){\AKhR} \ar[d]& \bigoplus_m
\mathrm{Tr}(\SBim^{\mathrm{dg}}_m) \ar[r] \ar[d]  & \Komh(\Proph)
\ar[r]^(0.35){K_0}\ar@{-->}[d]& \bigoplus_{m} H_m/[H_m,H_m]\ar[d]
     \\
 \mathcal{Z}(\braid_n)   \ar[r]\ar[d]&
 \mathcal{Z}^{\mathrm{dg}}(\SBim^{\mathrm{dg}}_n) \ar[d] \ar[r]&
  \mathcal{Z}(\Komh(\SBim_n))\ar[r]^{K_0}\ar[d]&
    \mathrm{Z}[H_n]\ar[d]\\
\tang_n   \ar[r]^(.45){\KhR} & \SBim^{\mathrm{dg}}_n\ar[r]& \Komh(\SBim_n) \ar[r]^{K_0} & H_n 
}
\]
\caption{
\label{fig:bigpicture}}
\end{figure}

Unfortunately, $\Komh(\Proph)$ does not seem to be rich enough to admit a
functor to the Drinfeld center $\mathcal{Z}(\Komh(\SBim_n))$ that intertwines the
Khovanov--Rozansky functors for annular links and partial braid closures, as we
will explain next. 

\begin{exa} Let $L$ be an annular link and $T$ a tangle. Consider the cobordism
that rotates $L$ once around the annulus. This cobordism induces the identity
map on the annular invariant in $\Komh(\Proph)$. However, after wrapping $L$
around the tangle $T$, the cobordism that rotates $L$ around $T$ is not expected
to induce the identity map on the tangle invariant in $\Komh(\SBim_n)$.
\end{exa}

To get a categorified wrapping operation, we thus need an upgraded annular
Khovanov--Rozansky functor with a target category that remembers such rotation
cobordisms. A natural candidate for such a category is the derived horizontal
trace of the dg category $\SBim^{\mathrm{dg}}_n$ of complexes of Soergel
bimodules. This and a related notion of derived center feature in the second
column of Figure~\ref{fig:bigpicture} and are the focus of the follow-up paper
\cite{2002.06110} of the authors with Matthew Hogancamp.

\section{Traces outside of type $A$}
\label{sec:other types}

\subsection{Categorical traces}
\label{sec: traces}

We briefly review the definitions of categorical traces following \cite{BHLZ}. 

If $\CC$ is a $\kk$-linear category, its {\bf vertical trace} $\vvTr(\CC)$ (also known as zeroth Hochschild homology) is a $\kk$-vector space spanned by all possible $f\in \End_\CC(X)$ for $X\in \Ob\ \CC$ modulo the relations $fg\sim gf$ for any $f\in \Hom(X,Y)$ and $g\in \Hom(Y,X)$. If $\CC$ is monoidal then $\vvTr(\CC)$ has a natural algebra structure.

If $\CC$ is a monoidal $\kk$-linear category, one can also define its {\bf horizontal trace} $\hTr(\CC)$. This is a $\kk$-linear category where the objects are the same as the objects in $\CC$, and the morphisms are defined by 
$$
\Hom_{\hTr(\CC)}(X,Y)=\bigoplus_{Z}\Hom_{\CC}(X\otimes Z,Z\otimes Y)/\sim
$$ 
where for any $f\in \Hom_{\CC}(X\otimes Z,W\otimes Y)$ and $g\in \Hom_{\CC}(W,Z)$ we identify the compositions
$$
X\otimes Z\xrightarrow{f} W\otimes Y\xrightarrow{g\otimes \id_Y} Z\otimes Y
$$
and
$$
X\otimes W\xrightarrow{\id_X\otimes g} X\otimes Z\xrightarrow{f} W\otimes Y.
$$
It is easy to see from this definition that
$$
\Hom_{\hTr(\CC)}(\one,\one)=\vvTr(\CC).
$$
For the definition of composition of morphisms and further details we refer to \cite{BHLZ}.  There is a natural trace functor
$$
\hTr: \CC\to \hTr(\CC)
$$
which sends any object of $\CC$ to the namesake object in $\hTr(\CC)$. If $\CC$ has duals then $\hTr(X\otimes Y)\simeq \hTr(Y\otimes X)$. 

Informally, one can think of objects of $\hTr(\CC)$ as of annular closures of objects in $\CC$. In particular, the horizontal trace for the category of webs (with morphisms given by foams) is the category of annular webs (with morphisms given by annular foams).

Finally, there is a derived version of the above definitions developed in detail in the follow-up paper \cite{2002.06110}. The vertical trace is replaced by full Hochschild homology of $\CC$, while the horizontal trace becomes a dg category.

\subsection{Cubes and Coxeter braids in other types}

Let $(W,S)$ be a Coxeter system of rank $r$ with a realization
(\cite{Elias-Williamson--Soergel-Calc}) consisting of a $\R$-linear
representation $\hh=\bigoplus_{s\in S} \R \alpha_s^\vee$ of $W$ and simple roots
$\{ \alpha_s \; | \; s \in S\} \subset \hh^* = \Hom_{\R}(\hh,\R)$ defined such
that $\langle \alpha_t^\vee, \alpha_s \rangle = -2\cos(\pi/m_{st})$ where
$m_{ss} =1$ and $\pi/\infty = 0$. The Coxeter group $W$ acts on $\hh$ by
reflections.
\[s(v) := v - \langle v, \alpha_s\rangle \alpha_s^{\vee}\] for $s\in S$ and $v
\in \hh$. We let $R:=\C[\hh]= \bigoplus_{k\geq 0}\Sym_{\C}^k(\hh^*\otimes \C)$
denote the coordinate ring of the representation, i.e. the polynomial ring
generated by the simple roots. We let $\SBim$ (and $\SSBim$) denote the category
of (singular) Soergel bimodules associated to $(W,S)$ and the above realization. 

\begin{defi}
Let $\Cube_W$ denote the Koszul complex of $\C[W]$-modules determined by its degree one differential $\hh^*\otimes_{\C} R \to R$ given by multiplication $m \colon \alpha_s\otimes x \mapsto \alpha_s x$. More explicitly
 \[\Cube_W=\bV^\bullet_R(\hh^*\otimes R) = \left[ q^{r}\bV^r(\hh^*)\otimes R \to q^{r-2}\bV^{r-1}(\hh^*)\otimes R \to \cdots \to q^{2-r}\hh^*\otimes R \to q^{-r}\uwave{R} \right] \]
 with differentials induced by co-multiplication and multiplication
 \[\bV^{k}(\hh^*)\otimes R \xrightarrow{\Delta_{k-1,1}\otimes \id} \bV^{k-1}(\hh^*) \otimes \hh^* \otimes R \xrightarrow{\id \otimes m} \bV^{k-1}(\hh^*)\otimes R\]  
 \end{defi}

Recall that $R$ is the monoidal unit in $\SBim$, and we will think of $\Cube_W$ as
a complex in $\Komh(\SBim)$, and in particular, as a complex of $R-R$-bimodules.
To this, we can apply the horizontal trace functor term-wise. Here, the
horizontal trace is nothing but $HH_0$, i.e. the functor of tensoring with $R$
over $R\otimes R$, which identifies the left- and right actions. We consider the
resulting objects as $R$-modules. 

Since $\Cube_W$ is built from copies of $R$ and $\hTr(R)=R$, we could identify it with its image $\hTr(\Cube_W)$ in $\Komh(\hTr(\SBim))$. Note that $\C[W]\ltimes R$ acts on $R$ and the differentials in $\hTr(\Cube_W)$ are equivariant for this action, so we will consider $\hTr(\Cube_W)$ as a complex of $\C[W]\ltimes R$-modules.

Next we chose a total ordering on $S$ and consider the corresponding Coxeter element $s_k\cdots s_1\in W$ as well as the Coxeter braids $\sigma_{\epsilon}:=\sigma_k^{\epsilon_k}\cdots \sigma_1^{\epsilon_1}$ for $\epsilon=(\epsilon_1,\dots, \epsilon_k)\in \{\pm 1\}^k$ in the Artin--Tits group corresponding to $(W,S)$. 

\begin{defi} Let $C_{\epsilon}$ denote the chain complex in $\Komh(\hTr(\SBim))$ obtained from the (suitably normalized\footnote{We use the following convention for Rouquier complexes: $\sigma_s\mapsto [\uwave{B_s} \to q^{-1}R]$ for $s\in S$.}) Rouquier complex \cite{Rou1} of $\sigma_{\epsilon}$ by applying the horizontal trace functor term-wise. 
\end{defi}

 Let $\epsilon\in \{\pm 1\}^k$ and partition the set $S=S_+\sqcup S_-$ according to the chosen order of simple roots. We denote the corresponding parabolic subgroups by $W_\epsilon$ and $W'_\epsilon$. Let $s_\epsilon\in \C[W]$ be the symmetrizer corresponding to $W_\epsilon$, and $\overline{s}_{\epsilon}\in \C[W]$ the anti-symmetrizer corresponding to $W'_\epsilon$. The following is a generalization of Theorem~\ref{thm:solomon}.

\begin{thm}[{\cite[Theorem 2]{Solomon}}]
The group algebra $\C[W]$ can be presented as a direct sum of left ideals:
\begin{equation}
\label{eqn: solomon decomposition2}
\C[W]=\bigoplus_{\epsilon\in \{\pm 1\}^r}\C[W]s_{\epsilon}\overline{s}_{\epsilon}
\end{equation}
\end{thm}

Let $p_\epsilon\in \C[W]$ denote the idempotent projecting onto the summand $\C[W]s_{\epsilon}\overline{s}_{\epsilon}$.

\begin{conj} 
\label{conj: coxeters other types}
$C_{\epsilon}[|\epsilon|_+]\simeq  p_{\epsilon}\hTr(\Cube_W)$ in $\Komh(\hTr(\SBim))$.
\end{conj}

In particular, we expect that $C_{-1,\cdots ,-1}$ is homotopy equivalent to $\hTr(\Cube_W)^{\mathrm{sign}}$ and $C_{+1,\cdots ,+1}[r]$ is homotopy equivalent to $\hTr(\Cube_W)^{W}$.

\subsection{Annular simplification in other types}
In type $A_{n-1}$, we know (and have made ample use of the fact) that $\Kar(\hTr(\SBim))\cong \C[S_n]\ltimes R-\mathrm{gpmod}$. In this section, we pursue an analogous description for other finite Coxeter groups.

A key tool is Elias-Lauda's computation \cite{EL} of the vertical trace decategorification of $\SBim$. To describe this, we consider $\SBim^*$, the category whose objects are objects in $\SBim$ without grading shifts, and hom spaces are graded by $\Hom_{\SBim^*}(A,B)\cong\bigoplus_m \Hom_{\SBim}(A,q^{-m}B)$. The vertical trace is the quotient 
\[\vvTr(\SBim^*)=\bigoplus_{A\in \mathrm{Ob}(\SBim^*)} \End_{\SBim^*}(A)\bigg/ \mathrm{span}\{fg-gf\}  \] where the span is taken over pairs of $f\in \Hom_{\SBim^*}(A,B)$ and $g \in \Hom_{\SBim^*}(B,A)$. Since $\SBim^*$ is graded and monoidal, $\vvTr(\SBim^*)$ has the structure of a graded algebra. 

\begin{thm}[{\cite[Theorem 3.2]{EL}}] 
\label{thm:EliasLauda}
There is an isomorphism $\phi\colon \vvTr(\SBim^*) \to \C[W] \ltimes R$ of graded algebras.
\end{thm}

Recall that the 2-category of singular Soergel bimodules $\SSBim$ for $(W,S)$ is the closure under grading shifts, taking direct sums and summands of the 2-category of bimodules generated by singular Bott-Samelson bimodules $R^I\otimes_{R^{I\cup J}} R^J$ for $I,J\subset S$. Here we denote by $R^I$ the ring of invariants for the parabolic subgroup $W_I\subset W$ generated by reflections in $I$. 

We identify the objects of $\SSBim$ with subsets $I\subset S$ and 1-morphisms from $J$ to $I$ are $R^I-R^J$-bimodules. The full 2-subcategory of $\SSBim$ generated by the object $\emptyset\subset S$ is canonically identified with $\SBim$. We can think of $\SSBim$ as a partial idempotent completion of $\SBim$ in the 1-morphism direction.

\begin{lem} Considering the vertical trace of the bicategory $\SSBim$ as an idempotented algebra, we have an algebra isomorphism given on idempotent truncations  by $\psi\colon \vvTr(\SSBim(I,J)) \cong 1_I (\C[W] \ltimes R) 1_J$. Here $I$ and $J$ denote subsets of $S$ and $1_I$ and $1_J$ are the corresponding symmetrizers in $\C[W]$.
\end{lem}
\begin{proof} 
All singular Soergel $(R^I,R^J)$-bimodules $B, B'$ can be turned into ordinary Soergel bimodules by tensoring on both sides with $R$, considered as an $(R,R^I)$-bimodule or as an $(R^J,R)$-bimodule respectively. For $I\subset S$ we let $r_I$ denote the rank of $R^I\otimes R \otimes R^I$ as a free $R^I$-module (this is the size of the double coset $W_I\backslash W/W_I$). For morphisms $B\xrightarrow{f}B'\xrightarrow{g}B$, we now define 

\[\psi([f\circ g]):=r^{-1}_{I}r^{-1}_{J}\phi([\id_{R}\otimes (f\circ g) \otimes \id_{R}]).\] 

It is straightforward to check that this defines an algebra map and $\psi$ agrees with $\phi$ on the traces of endomorphisms of Soergel bimodules. Furthermore, the image of $\vvTr(\SSBim(I,J))$ under $\psi$ lands in $1_I (\C[q^{\pm 1}][W] \ltimes R) 1_J$ since $\phi([\id_{R\otimes_{R^I} R}])=r_I 1_I\in \C[W]$. An analogous argument shows that $\psi$ is injective and surjective.
\end{proof}

The following is straightforward:

\begin{prop}
The natural functors $\vvTr(\SSBim)\hookrightarrow \hTr(\SBim)$ and $\Komh(\vvTr(\SSBim)) \hookrightarrow \Komh(\hTr(\SBim))$ are fully faithful. 
\end{prop}

In type A, these functors  yield a subcategory of the horizontal trace generated by collections of circles colored by $\bV^i$, and by complexes thereof.  Proposition \ref{prop:rotationequiv} then implies that the latter functor is an equivalence. As explained in Section  \ref{sec:KhR}, this is related to the fact that every representation of $S_n$ can be resolved by representations induced from the trivial representations of the parabolic subgroups.

Outside of type A, this is no longer true. For example, if $W=I_n$ is a dihedral group  of order $2n$, then it has four parabolic subgroups $\{e\},\{s\},\{t\},W$. The corresponding induced trivial representations have dimensions $2n$, $n$, $n$ and 1 and it is easy to see that for $n>3$ the irreducible two-dimensional representation $\hh$ cannot be resolved by these. On the other hand, by Conjecture \ref{conj: coxeters other types} the horizontal trace of the positive Coxeter braid 
corresponds to the complex
$$
\left[\bV^2\hh\to \hh\to \text{triv}\right],
$$
where we identify an irreducible representation $\tau$ of $W$ with $\Hom_{W}(\tau, \hTr(R))$. 
Therefore we do not expect the functor $\Komh(\vvTr(\SSBim)) \hookrightarrow \Komh(\hTr(\SBim))$ to be essentially surjective.
 





\begin{rem}
The category $\Komh(\hTr(\SBim))$ is expected to be closely related to the
category of character sheaves \cite{Lusztig,RR} in the corresponding type. The
object $E$, or the trace of identity object in $\SBim$, corresponds to so-called
Springer sheaf, and its  endomorphisms match the vertical trace $\vvTr(\SBim^*)$
from Theorem \ref{thm:EliasLauda}. In type $A$, it is known that the summands of
the Springer sheaf generate the category of character sheaves, but this is no
longer true in other types due to the existence of so-called {\bf cuspidal
sheaves}. 

By analogy, we do not expect the functor $\Komh(\vvTr(\SSBim)) \hookrightarrow
\Komh(\hTr(\SBim))$ to become essentially surjective even after Karoubi
completion on both sides. The cuspidal sheaves should correspond to certain
objects in $\Kar(\Komh(\hTr(\SBim)))$ which do not belong to the essential image
of  $\Kar(\Komh(\vvTr(\SSBim)))$.  It would be very interesting to construct
cuspidal objects explicitly by Soergel-theoretic methods. See also
\cite[Section 1.4]{2002.06110} and \cite{BT1,BT2} for a further discussion.
\end{rem}

 \begin{appendix}
\section{Some facts from homological algebra}
\label{sec:appendix}
\subsection{Thomason's theorem} 
 
Suppose that $\CC$ is a full triangulated subcategory of a triangulated category $\CA$. Following Thomason \cite{Thomason}, we say that $\CC$ is {\it dense} in $\CA$ if every object of $\CA$ is a direct summand of an object isomorphic to an object in $\CC$.  
 
\begin{thm}[\cite{Thomason}]
\label{thm:thomason}
Let $\CA$ be a triangulated category. There is a bijective correspondence between full dense triangulated subcategories of $\CA$ and the subgroups of the Grothendieck group $K_0(\CA)$.
\end{thm} 

Given a subgroup $H\subset K_0(\CA)$, the corresponding full subcategory $\CC_H$ consists of objects of $\CA$ with equivalence classes in $H$. The theorem states that $\CC_H$ is actually triangulated and dense in $\CA$, and all full dense triangulated subcategories appear this way.

\begin{cor}
\label{cor:thomason}
Suppose that $\CC$ is a full dense triangulated subcategory of $\CA$ and $K_0(\CC)=K_0(\CA)$. Then $\CC=\CA$.
\end{cor}
 
\subsection{Strict idempotents} 
 
 Suppose now that $\CC$ is Karoubian. Suppose that we are given an idempotent endomorphism $\epsilon: X\to X$.
 Then $\overline{\epsilon}=1-\epsilon$ is also an idempotent.  There is a canonical splitting
\begin{equation}
\label{eq: idempotent splitting}
X=X_{\epsilon}\oplus X_{\overline{\epsilon}}
\end{equation}
such that $\epsilon=\id$ on $X_{\epsilon}$ and $\epsilon=0$ on $X_{\overline{\epsilon}}$. 
 
\begin{lem}
\label{lem: splitting morphism}
Suppose that $a\colon X\to Y$ is a morphism in $\CC$, and $X,Y$ have idempotent endomorphisms $\epsilon$ (which we will denote by the same letter) such that $a\epsilon=\epsilon a$. Then $a$ preserves the splitting \eqref{eq: idempotent splitting}.
\end{lem} 

\begin{proof}
Since $\CC$ is additive, a morphism between direct sums is determined by its components. It is easy to see that  
the components $X_{\epsilon}\to Y_{\overline{\epsilon}}$ and $X_{\overline{\epsilon}}\to Y_{\epsilon}$ vanish.
\end{proof}
 
 Similarly, if $A$ is a chain complex over $\CC$ and $\epsilon\colon A\to A$ is an idempotent 
 chain endomorphism of $A$ then by Lemma \ref{lem: splitting morphism} we have a splitting $A=A_{\epsilon}\oplus A_{\overline{\epsilon}}$.   If $f:A\to B$ is a chain map and $A,B$ have two idempotent endomorphisms $\epsilon:A\to A$ and $\epsilon:B\to B$ then $f$ preserves the splitting. 
 
 \begin{rem}
 Here we need to use that $\epsilon^2=\id$ exactly, not up to a homotopy.
 \end{rem}
 
Suppose now that $\CC$ is not only Karoubian, but also symmetric monoidal.
 
\begin{thm}
\label{thm:Schurhomotopy}
Suppose that $A$ and $B$ are two chain complexes over $\CC$, then the following are true:
\begin{itemize}
\item[(a)] If $f:A\to B$ is a chain map, then there is a chain map $$\Sch^{\lambda}(f):\Sch^{\lambda}(A)\to \Sch^{\lambda}(B).$$
\item[(b)] If $f$ and $g$ are homotopic then $\Sch^{\lambda}(f)$ and $\Sch^{\lambda}(g)$ are homotopic
\item[(c)] If $A$ and $B$ are homotopy equivalent, then so are $\Sch^{\lambda}(A)$ and $\Sch^{\lambda}(B)$.
\end{itemize}
\end{thm}

\begin{proof}
For (a) observe that there is an $S_n$-equivariant morphism $f^{\otimes n}:A^{\otimes n}\to B^{\otimes n}$.
Since it is $S_n$-equivariant, it commutes with all the idempotents in $\C[S_n]$, and hence defines a map between Schur functors. For (b), observe that there is an $S_n$-equivariant homotopy between $f^{\otimes n}$ and $g^{\otimes n}$,
so it defines a homotopy between $\Sch^{\lambda}(f)$ and $\Sch^{\lambda}(g)$. Finally, (c) is a straightforward consequence of (b).
\end{proof}

\subsection{Homotopy idempotents}
\label{sec:Kar}
Recall that to any additive category $\KK$, one can associate another category $\Kar(\KK)$ called its Karoubi completion.
The objects of $\Kar(\KK)$ are pairs $(A,e)$ where $e\colon A\to A$ is an idempotent. A morphism between $(A,e)$ and $(A',e')$ is a morphism $f\colon A\to A'$ such that $fe=e'f=f$. There is a natural functor $i\colon \KK\to \Kar(\KK)$ which sends $A$ to 
$(A,\id_A)$.

The following is well known (e.g. \cite[Proposition 1.3]{Balmer}).

\begin{prop}
\label{prop:fully faithful}
Let $\KK$ be an additive category, let $\Kar(\KK)$ denote the Karoubi completion of $\KK$. Then  $\Kar(\KK)$ is additive and Karoubian. The natural functor $i:\KK\to \Kar(\KK)$ is additive and fully faithful.
\end{prop}

The next theorem is the main result of \cite{Balmer}.


\begin{thm}[{\cite[Theorem 1.12]{Balmer}}]
\label{th:karoubian triangulated} 
Let $\KK$ be a triangulated category. Then $\Kar(\KK)$ is also triangulated, and the natural functor $\KK\to \Kar(\KK)$ is 
triangulated. 
\end{thm}

Let $\CC$ be an additive category, and let $\KK$ be the bounded homotopy category of $\CC$. 

\begin{lem}
\label{lem:pn qn} Let $A$ be an object in $\KK$ with an idempotent endomorphism represented by a chain map $p\colon A\to A$---in other words, $p^2$ is homotopic to $p$.
For all odd $n\ge 1$ there exist objects $P_n,Q_n$ in $\KK$ such that $P_n\oplus Q_n\simeq A\oplus A[n]$,
where $p$ acts as the identity on $P_n$ and by zero on $Q_n$.
\end{lem}

\begin{proof}
We construct $P_n$ and $Q_n$ inductively. For $n=1$ let $P_1:=\Cone(1-p)$ and $Q_1:=\Cone(p)$. 
It is easy to see that they satisfy the desired properties.

Assume that we constructed $P_n$ and $Q_n$. We will construct $P_{n+2}$ and $Q_{n+2}$ as cones 
$P_{n+2}=\Cone[P_1[n]\xrightarrow{f_n} P_{n}]$ and $Q_{n+2}=\Cone[Q_1[n]\xrightarrow{g_n} Q_n]$
for certain chain maps $f_n$ and $g_n$, which we will also construct.

Let us embed $\KK$ into its Karoubi completion $\Kar(\KK)$. By Theorem \ref{th:karoubian triangulated} the latter is triangulated. Since $A$ has a homotopy idempotent $p$, we can split $A\simeq P'\oplus Q'$ for some objects 
$P',Q'$ in $\Kar(\KK)$ such that $p$ acts by 1 on $P'$ and by 0 on $Q'$. Now 
$$
P_1=\Cone[A\xrightarrow{1-p}A]\simeq_{\Kar(\KK)} \Cone[P'\oplus Q'\xrightarrow{\left(\begin{matrix}0 & 0\\ 0 & 1\\\end{matrix}\right)}P'\oplus Q']\simeq P'\oplus P'[1].
$$
Similarly, $Q_1\simeq Q'\oplus Q[1]$. Observe that in $\Kar(\KK)$ there is a chain map $f_1:P_1[1]\to P_1$ such that 
$\Cone(f_1)\simeq P'\oplus P'[3]$. Since the embedding $\KK\to \Kar(\KK)$ is fully faithful, the map $f_1$
is well defined in $\KK$, and we can define $P_3:=\Cone(f_1)$. Similarly, if we already defined 
$P_n\simeq P'\oplus P'[n]$ then in $\Kar(\KK)$ there is a chain map $f_n: P_1[n]\to P_n$ such that 
$$
P_{n+2}:=\Cone(f_n)\simeq P'\oplus P'[n+2].
$$
Again, since the embedding is fully faithful the map $f_n$ (and hence $P_{n+2}$) is well defined in $\KK$.

Analogously, one can define $Q_n$ such that   $Q_n\simeq Q'\oplus Q'[n]$ in $\Kar(\KK)$. Then $P_n\oplus Q_n\simeq A\oplus A[n]$.
\end{proof}

\begin{rem}
One can write
$$
P_n=[A\xrightarrow{1-p}A\xrightarrow{p}A\cdots A\xrightarrow{1-p}A],
$$
$$
Q_n=[A\xrightarrow{p}A\xrightarrow{1-p}A\cdots A\xrightarrow{p}A].
$$
Since $p(1-p)$ vanishes up to homotopy, one can hope that the above sequences can be 
lifted to actual complexes by adding higher differentials. It is proved in  \cite[Propositions 3.1 and 3.2]{Neeman} that this construction is
unobstructed.  
\end{rem}

\begin{thm}
\label{th:homotopy karoubian}
 The bounded homotopy category of a Karoubian category is Karoubian.
\end{thm}

\begin{proof}
As above, let $\CC$ be a Karoubian category and $\KK$ its bounded homotopy category. Suppose that $A$ is a complex in $\KK$ with a homotopy idempotent $p$, we need to prove that $A$ splits. 

Since $A$ is bounded, we can pick a large enough odd positive integer $n$ such that $A$ and $A[n]$ are supported in non-overlapping homological degrees. By Lemma \ref{lem:pn qn}, one can decompose $A\oplus A[n]\simeq P\oplus Q$ where $p$ is homotopic to identity on $P$ and to 0 on $Q$. Let us pick some homological degree $i$ such that $A$ is supported in degrees strictly smaller than $i$ and $A[n]$ is supported in degrees strictly bigger than $i$. Then $(A\oplus A[n])^i=0$.

Let $h$ be a homotopy between $p$ and identity on $P$. Since $(A\oplus A[n])^i=0$ we get $dh^i+h^{i+1}d=1$ as endomorphisms of $P^i$. Let $q=dh^i$ and $q'=h^{i+1}d$, then $q+q'=1$, $q^2=dh^idh^i=(dh^i+h^{i+1}d)dh^i=dh^i=q$ and similarly 
$(q')^2=q'$. Therefore $q$ and $q'$ are orthogonal idempotents acting on $P^i$, so (since $\CC$ is Karoubian) we can rewrite $P^i=(P^i)'+(P^i)''$.

Moreover, we can split $P$ into two parts: $P'=P^{<i}\to (P^i)'$, and $P''=(P^i)''\to P^{>i}$. The same splitting works for $Q$.
It is now easy to see that the map $h^{<i}$ induces a homotopy between $(A\oplus A[n])^{\le i}=A$ and $P'\oplus Q'$.
\end{proof}

\begin{rem}
Similarly to \cite{Balmer}, one can instead deduce Theorem \ref{th:homotopy karoubian} from Theorem \ref{th:karoubian triangulated} and Theorem \ref{thm:thomason}. The proof presented here is slightly more explicit, following Proposition 1.5.6(iii) of \cite{BeiVol}.
\end{rem}

\subsection{Schur functors in homotopy categories}

\begin{prop}
\label{prop: homotopy schur} Let $\CC$ be a Karoubian monoidal category and let $\KK$ be its homotopy category. Suppose that $(E,s)$ is self-commuting in the sense of Remark~\ref{rem:self-commuting}. Then the Schur functors $\Sch^{\lambda}(E)$ are well defined up to homotopy equivalence.
\end{prop}

\begin{proof}
By the assumption, there is an action of $S_n$ on $E^{\otimes n}$. If $\e_{\lambda}$ is an idempotent in $\C[S_n]$
then by Theorem \ref{th:homotopy karoubian} there exists a splitting $E^{\otimes n}\simeq \Sch^{\lambda}(E)\oplus E'$
where $\e_{\lambda}$ acts by identity on $\Sch^{\lambda}(E)$ and by 0 on $E'$. This splitting is unique up to isomorphism in $\KK$. Since $\{\e_{\lambda}\}$ form an orthogonal system of idempotents, it is easy to see that
 $E^{\otimes n}\simeq \bigoplus_{|\lambda|=n}\Sch^{\lambda}(E)$.
\end{proof}

\end{appendix}

\bibliographystyle{alpha}
\bibliography{pw}
\end{document}